\documentclass{article}
\usepackage[utf8]{inputenc}
\pdfoutput=1

\usepackage{lipsum}
\usepackage{amsfonts}
\usepackage{graphicx}
\usepackage{epstopdf}
\ifpdf
  \DeclareGraphicsExtensions{.eps,.pdf,.png,.jpg}
\else
  \DeclareGraphicsExtensions{.eps}
\fi

\usepackage{tikz}

\usepackage[color=blue!20!white,textsize=tiny]{todonotes}
\usepackage{graphicx}
\usepackage{caption}
\usepackage{subcaption}
\usepackage{bm}
\usepackage{amsmath}
\usepackage{stmaryrd}

\usepackage{algorithm}
\usepackage[noend]{algpseudocode}

\usepackage[margin=1in]{geometry}

\usepackage{cleveref}
\usepackage{ntheorem}
\newtheorem{lemma}{Lemma}
\theoremstyle{nonumberplain}
\newtheorem{proof}{Proof} 
\newcommand{\qed}{\hfill \ensuremath{\Box}}

\DeclareMathOperator{\spn}{span}
\DeclareMathOperator{\diag}{diag}

\newcommand*{\horzbar}{\rule[.5ex]{2.5ex}{0.5pt}}

\newcommand{\edits}[1]{{\color{black}{#1}}}

\title{Graph-theoretic algorithms for Kolmogorov operators: Approximating solutions and their gradients in elliptic and parabolic problems on manifolds}
\author{Andrew D. Davis\thanks{Corresponding author: \texttt{davisad@alum.mit.edu}}\,\, and Dimitrios Giannakis \\ \\
 Department of Mathematics and Center for Atmosphere Ocean Science, \\ Courant Institute of Mathematical Sciences, New York University
}

\usepackage{fancyhdr}
\usepackage{lipsum}

\fancypagestyle{specialfooter}{%
  \fancyhf{}
  
  \fancyfoot[C]{This research was supported by the Multidisciplinary University Research Initiatives (MURI) Program, Office of Naval Research (ONR) grant number N00014-19-1-2421. The authors would also like to thank Georg Stadler (NYU, Courant) for many helpful discussions.}
}

\begin{document}
\maketitle
\thispagestyle{specialfooter}

\begin{abstract}
    We employ kernel-based approaches that use samples from a probability distribution to approximate a Kolmogorov operator on a manifold. The self-tuning variable-bandwidth kernel method [Berry \& Harlim, \emph{Appl.\ Comput.\ Harmon.\ Anal.}, 40(1):68--96, 2016] computes a large, sparse matrix that approximates the differential operator. Here, we use the eigendecomposition of the discretization to (i) invert the operator, solving a differential equation, and (ii) represent gradient vector fields on the manifold. These methods only require samples from the underlying distribution and, therefore, can be applied in high dimensions or on geometrically complex manifolds when spatial discretizations are not available. We also employ an efficient $k$-$d$ tree algorithm to compute the sparse kernel matrix, which is a computational bottleneck.
\end{abstract}

\section{Introduction}

Graph Laplacian and kernel methods that estimate properties of large data sets are widely used tools in machine learning, including supervised and unsupervised problems such as classification \cite{bertozzietal2018}, clustering \cite{ngetal2002,shimalik2000}, dimension reduction \cite{belkinniyogi2003}, and forecasting \cite{berryetal2015}. In manifold learning applications, these tasks are carried out by approximating the action of a differential operator given samples from a distribution $\psi$ concentrated on or near a manifold in the ambient data space. Importantly, the methods only require samples from the underlying distribution $\psi$, without needing to evaluate the probability density function. In this work, we are interested in finding solutions to differential equations of the form $\mathcal{L}_{c, \psi} f = g$, where 
\begin{displaymath}
    \mathcal{L}_{c, \psi} f = \Delta f + c \nabla f \cdot \frac{\nabla \psi}{\psi}
\end{displaymath}
is an elliptic Kolmogorov operator (defined rigorously in Section~\ref{sec:Kolmogorov-problem}). Here, $c \in \mathbb{R}$ is a parameter, $\Delta$ is the Laplace-Beltrami operator, and $\nabla$ is the gradient operator. Another objective we address is representing the gradient vector field $\nabla f$ of the solution.  

    
Evaluating exponentially decaying kernels of the form $K(\epsilon, \boldsymbol{x}, \boldsymbol{y}) = K(\|\boldsymbol{x}-\boldsymbol{y}\|^2/\epsilon^2)$ on pairs of points $\boldsymbol{x}, \boldsymbol{y} \in \Omega \subseteq \mathbb R^m$ allows us to approximate geometric operators for functions defined over a manifold $\Omega$ \cite{belkinniyogi2003,coifmanlafon2006,singer2006}. Define the integral operator 
\begin{equation}
    \mathcal{I}_{\epsilon} f(\boldsymbol{x}) = \int_{\Omega} \widetilde K(\epsilon, \boldsymbol{x}, \boldsymbol{x^{\prime}}) f(\boldsymbol{x^{\prime}}) \psi(\boldsymbol{x^{\prime}}) \, d \boldsymbol{x^{\prime}} \approx \sum_{i=1}^{n} \widetilde K(\epsilon, \boldsymbol{x}, \boldsymbol{x}^{(i)}) f(\boldsymbol{x}^{(i)}),
    \label{eq:integral-kernel-operator}
\end{equation}
where $\widetilde K $ is a kernel constructed by a normalization (e.g., Markov normalization) of $K$, $\psi$ is a probability density function defined on $\Omega$, and $\boldsymbol{x}^{(i)} \sim \psi$. Intuitively, the exponential decay of the kernel localizes this integral in an $\epsilon$-ball around $\boldsymbol{x}$. For small $\epsilon$, we have the asymptotic expansion in the form of a Taylor series
\begin{equation}
    \mathcal{I}_\epsilon f(\boldsymbol{x}) = m_\psi( \boldsymbol x )f(\boldsymbol{x})  + \epsilon^2 \mathcal{L}_{\psi} f(\boldsymbol{x}) + O(\epsilon^4),
    \label{eq:integral-Taylor-expansion}
\end{equation}
where $m_\psi$ is a function and  $\mathcal{L}_{\psi}$ is a second-order differential operator that may depend on the density $\psi$ and the normalization of the kernel. 

Kernel Density Estimation (KDE) is a common application of the Taylor expansion in~\eqref{eq:integral-Taylor-expansion}. A suitable kernel normalization leads to $m_\psi = \psi $, and we recover the density $\psi(\boldsymbol{x})$---to leading (zeroth) order in $\epsilon$---by applying the operator $\lim_{\epsilon \rightarrow 0} \mathcal{I}_{\epsilon} f$ to the constant function $f(\boldsymbol{x}) = 1$ \cite{kimpark2013,ozakingray2009,pelletier2005}. In the limit of infinite samples $n \rightarrow \infty$, the estimate in~\eqref{eq:integral-kernel-operator} converges to the true density $\psi(\boldsymbol{x})$. Given a bandwidth function $\rho(\boldsymbol{x})$, define the variable-bandwidth kernel
\begin{equation}
    K_{\rho}(\epsilon, \boldsymbol{x}, \boldsymbol{y}) = K\left( \frac{\|\boldsymbol{x}-\boldsymbol{y}\|^2}{\epsilon^2 \rho(\boldsymbol{x}) \rho(\boldsymbol{y})} \right).
    \label{eq:variable-bandwidth-kernel}
\end{equation}
Density estimation methods are significantly more accurate using variable-bandwidth kernels, especially in the tails \cite{rosenblatt1956,parzen1962,sainscott1996,terrellscott1992}.

In addition to density estimation, kernel methods can approximate differential operators acting on smooth functions. For instance, the kernel $K$ can be normalized such that $m_\psi = 1$, and the Taylor expansion in~\eqref{eq:integral-Taylor-expansion} leads to the approximation 
\begin{equation}
    \mathcal{L}_{\psi} f(\boldsymbol{x}) \approx \frac{\mathcal{I}_{\epsilon} f(\boldsymbol{x}) - f(\boldsymbol{x})}{\epsilon^2}.
\end{equation}
Laplacian eigenmaps \cite{belkinniyogi2003} and diffusion maps \cite{coifmanlafon2006} are commonly used tools to estimate the Laplace-Beltrami operator $\Delta$ on a compact Riemannian manifold using variants of this approach. Berry and  Harlim~\cite{berryharlim2016} develop a variable-bandwidth kernel method that uses $n$ samples from an underlying distribution $\psi$ on a potentially non-compact manifold to approximate a Kolmogorov operator $\mathcal{L}_{c,\psi}$. Discretizations of the differential operator $\mathcal{L}_{c, \psi}$ use the $n \times n$ kernel matrix $\boldsymbol{K}$, evaluated at each sample pair such that $K^{(ij)} = K(\epsilon, \boldsymbol{x}^{(i)}, \boldsymbol{x}^{(j)})$, to define a corresponding discrete Kolmogorov operator $\boldsymbol{L}_{c, \psi}$. In this application, \cite{berryharlim2016} shows pointwise convergence: The matrix--vector product $\boldsymbol{g} = \boldsymbol{L}_{c, \psi} \boldsymbol{f}$, such that $$\bm f = ( f^{(1)}, \ldots, f^{(n)})^\top$$ with $f^{(i)} = f(\boldsymbol{x}^{(i)})$, approximates the integral operator $\mathcal{L}_{c, \psi} f(\boldsymbol{x}^{(i)}) \approx g^{(i)} $, where $$( g^{(1)}, \ldots, g^{(n)} )^\top = \bm g.$$ There is an extensive literature on kernel methods for approximating differential operators on manifolds in a variety of functional settings, addressing issues such as pointwise convergence \cite{HeinEtAl07}, spectral convergence \cite{TrillosEtAl20}, and boundary conditions \cite{JiangHarlim20,vaughnetal2019}.

In this paper, we discuss how to practically use approximations of the elliptic Kolmogorov operator $\mathcal L_{c,\psi} $ given samples from the underlying distribution $\psi$. The aims and contributions of this work are to:
\begin{enumerate}
    \item Solve the Kolmogorov problem $\mathcal{L}_{c, \psi} f = g$ such that $\int_{\Omega} f \psi^c \, d \boldsymbol{x} = 0$ using $n$ samples $\{\boldsymbol{x}^{(i)}\}_{i=1}^{n}$ ($\boldsymbol{x}^{(i)} \sim \psi$),
    \item Represent the gradient vector field $\nabla f$ for any function $f$ that we can evaluate (or approximate) on the samples. 
    \item Present an algorithm that computes sparse approximations of the kernel matrix $\bm K$ with $\mathcal{O}(n \log{(n}))$ kernel evaluations.
    \item Provide an open-source implementation as part of the software package MUQ (\texttt{muq.mit.edu}).
\end{enumerate}

We solve the Kolmogorov problem $\mathcal{L}_{c, \psi} f = g$ such that $\int_{\Omega} f \psi^c \, d \boldsymbol{x} = 0$ using $n$ samples $\{\boldsymbol{x}^{(i)}\}_{i=1}^{n}$ ($\boldsymbol{x}^{(i)} \sim \psi$) by representing the functions $f$ and $g$ as a linear combination of eigenfunctions $\phi_j$ of $\mathcal{L}_{c, \psi}$.  For that, we compute eigenvectors $\boldsymbol{L}_{c, \psi} \bm \phi_j = \hat \lambda_j \bm \phi_j $, with $ \bm \phi_j = (\phi_j^{(1)}, \ldots, \phi_j^{(n)})^\top $, and approximate the $j^{th}$ eigenfunction of $\mathcal L_{c,\psi}$ evaluated at the $i^{th}$ sample as $ \phi_j(\bm x^{(i)}) \approx \phi_j^{(i)}$. We then approximate the solution $f$ through a least-squares solution of the linear system $\boldsymbol{L}_{c, \psi} \boldsymbol{f} = \boldsymbol{g}$ in an $\ell$-dimensional space ($ \ell \ll n $) spanned by the leading $\ell$ eigenvectors $\bm \phi_j$. Importantly, we can apply our method in high-dimensional ambient spaces or geometrically complex manifolds where spatial discretizations (e.g., finite-difference stencils) are not available. We represent $\nabla f$ as a linear combination of gradients of the eigenfunctions of $\mathcal{L}_{c,\psi}$, obtained using a Carr\'e du Champ identity (product rule) \cite{bakryetal2013,BerryGiannakis20}.

We employ a computationally efficient way of computing the sparse kernel matrix $\boldsymbol{K}$. In the worst-case scenario, computing the kernel matrix costs $n^2$ kernel function evaluations. However, since the kernel decays exponentially, we approximate $\boldsymbol{K}$ with a sparse matrix whose entries are zero when the kernel is smaller than a specified threshold. For each sample $\boldsymbol{x}^{(i)}$, we use a binary search tree to find its nearest neighbors such that the kernel evaluated at neighbor pairs exceed the threshold. Binary search trees are widely used to find nearest neighbors and within kernel based algorithms \cite{aryaetal1998,jonesetal2011}. We present our algorithm that computes sparse approximations of $\bm K$ with $\mathcal{O}(n \log{(n}))$ kernel evaluations for the sake of completeness and to explicitly write down a method that uses nearest neighbors search algorithms to construct discretizations of differential operators.

This paper is organized as follows. Section~\ref{sec:Kolmogorov-problem} provides an overview of existing kernel-based density estimation and approximations of Kolmogorov operators. In Section~\ref{sec:solution-gradients}, we present the main contributions of this paper, namely, solving the Kolmogorov problem and representing gradient vector fields using the eigenbasis of the Kolmogorov operator. In Section \ref{sec:implementation}, we discuss how to efficiently compute the kernel matrix $\boldsymbol{K}$ (the third contribution of this paper). Finally, in Sections~\ref{sec:examples}~and~\ref{sec:conclusions} we provide practical examples of our algorithms and concluding remarks, respectively.

\section{Background: Density estimation and the Kolmogorov operator} \label{sec:Kolmogorov-problem}

Let $\Omega \subseteq \mathbb{R}^{m}$ be a smooth, $d$-dimensional Riemannian manifold without boundary---note that $m$ is the dimension of the ambient space and $d$ is the intrinsic manifold dimension. Suppose that $\Omega$ is equipped with a probability measure with a smooth density $\psi: \Omega \to \mathbb{R}^{+}$ relative to the volume form of $\Omega$. Given $c \geq 0$, let $\mathcal{H}_{\psi, c}(\Omega)$ be the Hilbert space of real-valued functions on $\Omega$ whose inner product is defined by $\psi$, $\langle f, g \rangle_{\psi,c} = \int_{\Omega} f(\boldsymbol{x}) g(\boldsymbol{x}) \psi^c(\boldsymbol{x}) \, d\boldsymbol{x}$. We restrict choices of $c$ such that $\int_{\Omega} \psi^{c}(\bm x) \, d \boldsymbol{x} < \infty$---clearly, $c=1$ is always valid since $\psi$ is a probability density function. Given a smooth function $f \in \mathcal{H}_{\psi, c}(\Omega)$, we focus on elliptic Kolmogorov operators of the form
\begin{equation}
    \mathcal{L}_{\psi, c} f = \Delta f + c \nabla f \cdot \frac{\nabla \psi}{\psi},
    \label{eq:Kolmogorov}
\end{equation}
where $\Delta$ and $\nabla$ denote the Laplacian and gradient operators, respectively, and the dot operator $\cdot$ denotes the Riemannian inner product between tangent vectors on $\Omega$. The operator $\mathcal{L}_{\psi, c}$ is symmetric on $\mathcal{H}_{\psi, c}$ (i.e., $\langle f, \mathcal{L}_{\psi, c} g \rangle_{\psi,c} = \langle \mathcal{L}_{\psi, c} f, g \rangle_{\psi,c}$) and diagonalizable with an orthonormal basis of smooth eigenfunctions, 
\begin{equation}
    \label{eqEigs}
    \mathcal{L}_{\psi, c} \phi_j = \lambda_j \phi_j, \quad 0=\lambda_0 > \lambda_1 \geq \lambda_2 \geq \cdots \searrow -\infty. 
\end{equation}
When $c=0$, we additionally require that $\Omega$ is compact. 

For any valid $c$, define $\widetilde{\psi} = a \psi^{c}$, where $a>0$ is a normalizing constant. The operator $\mathcal L_{\tilde \psi, 1}$ satisfies 
\begin{equation}
    \mathcal{L}_{\widetilde{\psi}, 1} f = \Delta^2 f + c \nabla f \cdot \frac{\psi^{c-1} \nabla \psi}{\psi^{c}} = \mathcal{L}_{\psi, c} f
    \label{eq:Kolmogorov-operator-similarity}
\end{equation}
and, therefore, we could assume $c=1$ without loss of generality. However, we retain $c$ as a parameter that determines the degree with which the distribution $\psi$ biases the operator. 

Given a function $g \in \mathcal{H}_{\psi, c}$, we seek solutions to
\begin{equation}
    \mathcal{L}_{\psi, c} f = g, \quad \text{such that} \quad  \int_{\Omega} f(\boldsymbol{x}) \psi^c(\boldsymbol{x}) \, d \boldsymbol{x} = 0.
    \label{eq:Kolmogorov-problem}
\end{equation}
The requirement that $\int_{\Omega} f \psi^c \, d \bm{x} = 0$ makes this problem well-posed since constant functions are in the nullspace of $\mathcal{L}_{\psi, c}$. Additionally, we seek to compute the gradient vector field $\nabla f$ associated with the Riemannian metric of $\Omega$. Assume we have independent samples $\{\boldsymbol{x}^{(i)}\}_{i=1}^{n}$ from $\psi$. Notationally, given a function $f: \Omega \mapsto \mathbb{R}$, let the column vector $\boldsymbol{f} = ( f^{(1)}, \ldots, f^{(n)})^\top \in \mathbb{R}^{n}$ represent function values (or approximations) at each sample, i.e., $f(\boldsymbol{x}^{(i)}) = f^{(i)}$. Using these samples, Section \ref{sec:density-estimation} estimates the density $\psi$, and Section \ref{sec:discrete-kolmogorov} defines a matrix $\boldsymbol{L}_{\psi, c} \in \mathbb{R}^{n \times n}$ that approximates the Kolmogorov operator $\mathcal{L}_{\psi, c}$, meaning applying the matrix 
\begin{equation}
    \label{eqDiscreteKolmogorov}
    \boldsymbol{L}_{\psi,c} \boldsymbol{f} = \bm g
\end{equation}
implies $g^{(i)} \approx \mathcal{L}_{\psi, c} f(\boldsymbol{x}^{(i)})$. In Section \ref{sec:solution-gradients}, we use the eigenvectors of $\boldsymbol{L}_{\psi, c}$ as a basis to represent $f$ and its gradient $\nabla f$. 

\subsection{Density estimation} \label{sec:density-estimation}

We estimate the density $\psi$ using the procedure described in \cite{berryharlim2016}---additionally, see Algorithm~1 in \cite{giannakis2019}. Define the bandwidth function 
\begin{equation}
    b^2(\bm x) = \sum_{k=1}^{k_{nn}} \| \boldsymbol{x} - \boldsymbol{x}^{(I(\bm x,k))} \|^2,
    \label{eq:bandwdith-parameter}
\end{equation}
where $k_{nn}$ is an integer parameter that determines a local neighborhood centered at $\boldsymbol{x}$ and $I(\bm x,k)$ is the $k^{th}$ closest sample to $\boldsymbol{x}$. We use the convention that $I(\bm x^{(i)},0)=i$. Let $\epsilon$ be the bandwidth parameter and define the kernel matrix $\boldsymbol{K}_{\epsilon} \in \mathbb{R}^{n \times n} $ with entries 
\begin{equation}
    K_{\epsilon}^{(ij)} = K_b(\epsilon,\bm x^{(i)}, \bm x^{(j)} ) :=  \exp{\left( - \frac{\| \boldsymbol{x}^{(i)} - \boldsymbol{x}^{(j)} \|^2}{\epsilon^2 b(\bm x^{(i)}) b(\bm x^{(j)})} \right)},
    \label{eq:density-kernel-matrix}
\end{equation}
noting that $K_{b}$ is a variable-bandwidth kernel (defined generally in~\eqref{eq:variable-bandwidth-kernel}). Since the entries decay exponentially with square distance $\| \boldsymbol{x}^{(i)} - \boldsymbol{x}^{(j)} \|^2$, we typically truncate by setting $K_{\epsilon}^{(ij)} = 0$ if the entry is sufficiently small---this defines a large, but sparse, matrix. 

Recall from~\eqref{eq:integral-kernel-operator}~and~\eqref{eq:integral-Taylor-expansion} that performing a kernel normalization and applying the associated integral operator $\mathcal I_\epsilon$ to a constant function approximates the probability density function $\psi$. In the discrete case, a suitable normalizing function for density estimation is given by \cite{berryharlim2016}
\begin{displaymath}
    w(\bm x) = n (\pi \epsilon^2 b(\bm x)^2)^{d/2}. 
\end{displaymath}
This leads to the normalized kernel
\begin{displaymath}
\widetilde K_b(\epsilon,\boldsymbol x, \boldsymbol y) = \frac{K_b(\epsilon, \bm x, \bm y )}{w(\bm x)}  ,
\end{displaymath}
and the density estimate
\begin{equation}
    \label{eqDensEst}
    \hat\psi(\bm x) = \sum_{i=1}^n \widetilde K_b(\epsilon,\boldsymbol x, \boldsymbol x^{(i)} ) \approx \mathcal I_\epsilon 1,
\end{equation}
where we have approximated the integral operator in~\eqref{eq:integral-kernel-operator} applied to the constant function $f(\boldsymbol{x}) = 1$. Evaluating the density estimate at the sample points becomes a matrix-vector product by defining a diagonal normalizing matrix $\boldsymbol{W}$ with entries $W^{(ii)} = w(\boldsymbol{x}^{(i)})$, i.e.,
\begin{equation}
    \boldsymbol{\psi} = \boldsymbol{W}^{-1} \boldsymbol{K}_{\epsilon} \boldsymbol{1},
    \label{eq:density-estimation}
\end{equation}
where $\bm \psi = ( \hat \psi(\bm x^{(1)}), \ldots, \hat \psi(\bm x^{(n)})^\top $. Figure~\ref{fig:density-estimation} shows the output of this procedure using samples from a standard Gaussian distribution on $\Omega = \mathbb{R}^2$.

\begin{figure}[h!]
  \centering
  \begin{subfigure}{0.45\textwidth}
    \includegraphics[width=1.0\textwidth]{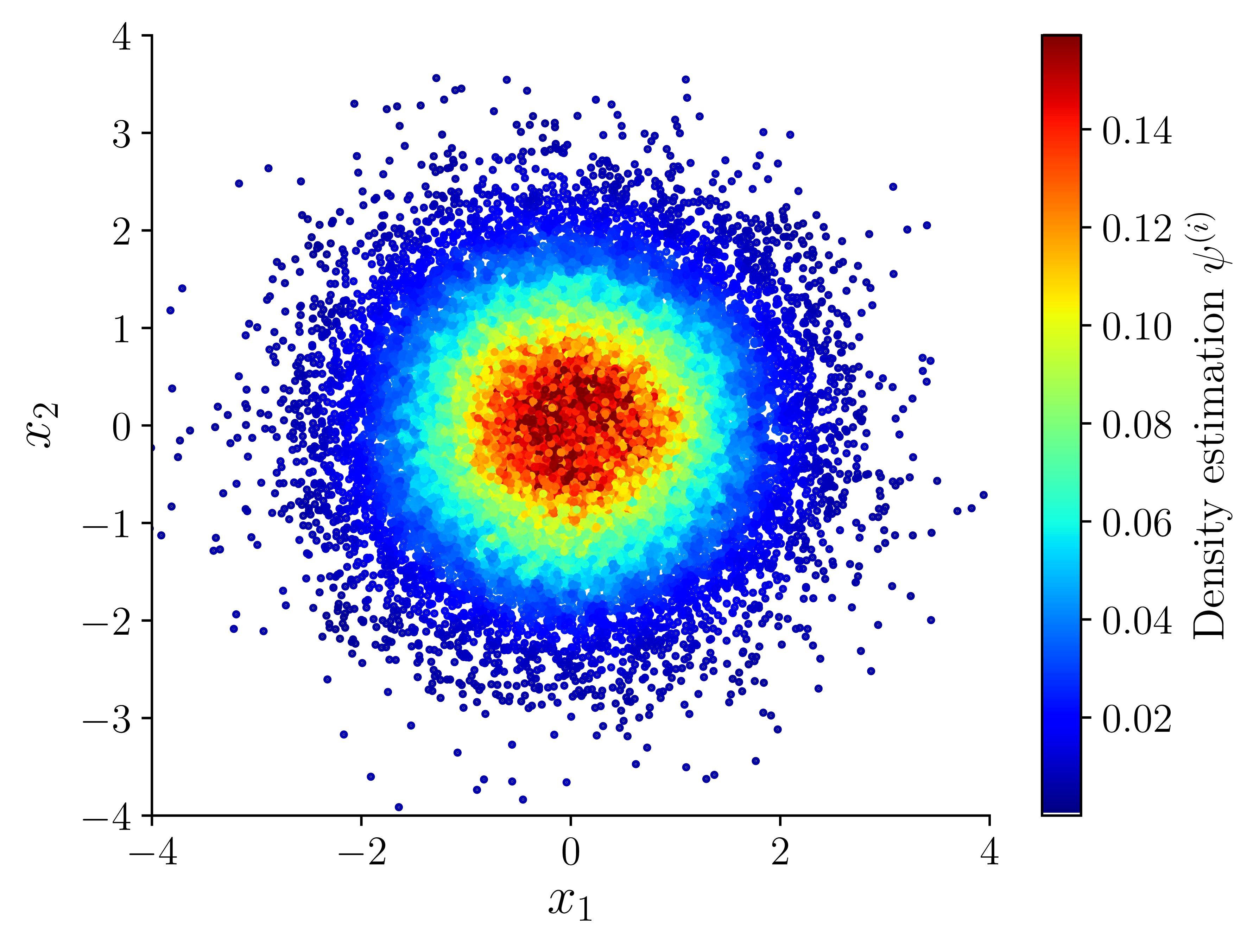}
    \caption{Estimated density $\boldsymbol{\psi}$}
  \end{subfigure}
  \begin{subfigure}{0.45\textwidth}
    \includegraphics[width=1.0\textwidth]{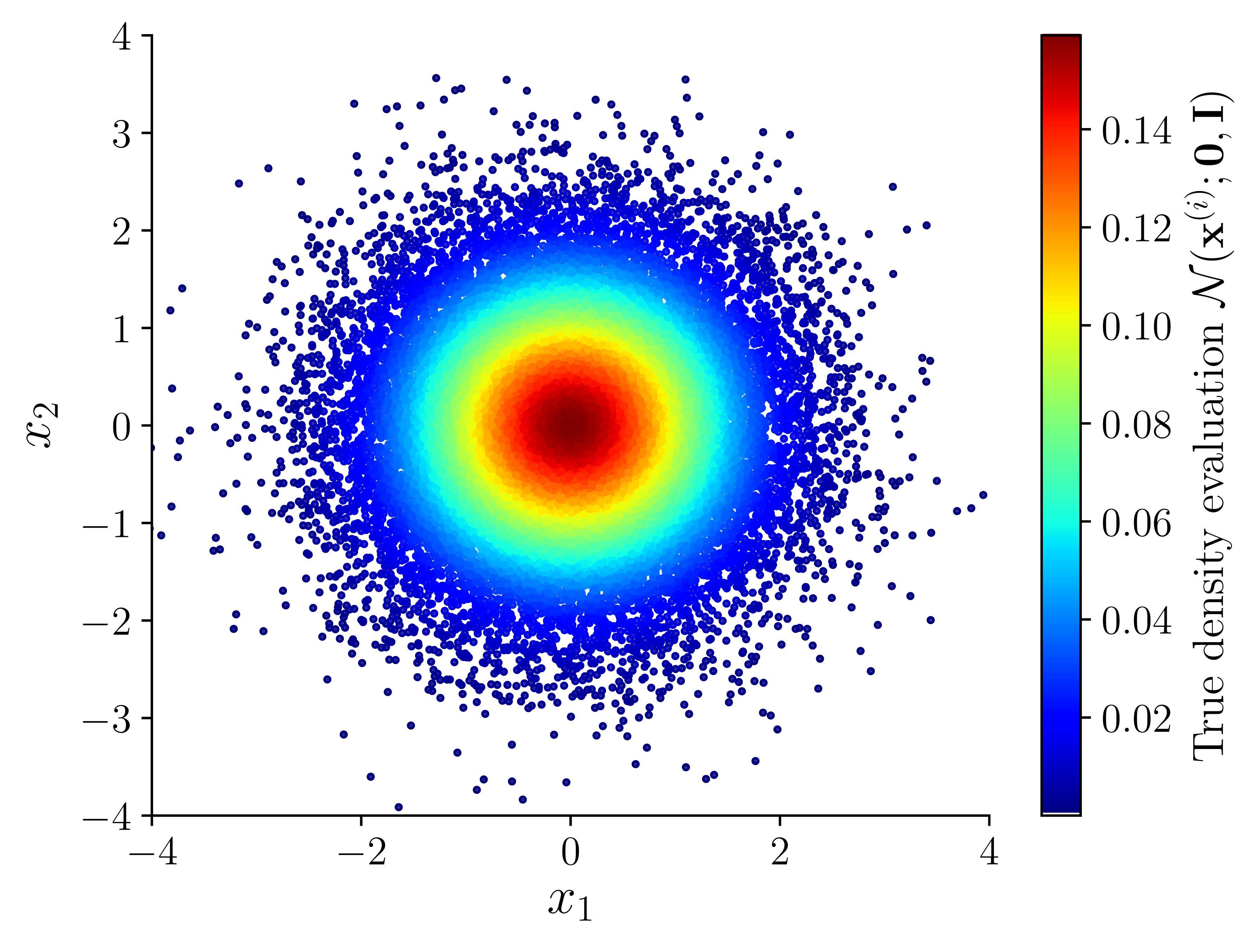}
    \caption{True density $\mathcal{N}(\boldsymbol{x}^{(i)}; \boldsymbol{0}, \boldsymbol{I})$}
  \end{subfigure}
  \caption{Density estimation given $n=2.5 \times 10^4$ samples in $\mathbb{R}^{2}$ from a standard Gaussian distribution $\mathcal{N}(\boldsymbol{0}, \boldsymbol{I})$. We compute the bandwidth parameter $b_i$ using $k_{nn} = 25$ nearest neighbors and set entries of the density kernel matrix $\boldsymbol{K}_{\epsilon}$ to zero if they are below the threshold $10^{-2}$, which makes $\boldsymbol{K}_{\epsilon}$ sparse. We estimate the optional bandwidth parameter $\epsilon$ using the procedure outlined in Section \ref{sec:optimal-bandwidth}.}
  \label{fig:density-estimation}
\end{figure}

\subsection{Discrete Kolmogorov operator} \label{sec:discrete-kolmogorov}

We use the procedure developed in \cite{berryharlim2016} to compute the discrete Kolmogorov operator $\boldsymbol{L}_{\psi, c}$. Given the density estimate $\hat \psi $ and a real parameter $\beta$, define the variable-bandwidth kernel matrix $\boldsymbol{K}_{\epsilon,\beta} \in \mathbb{R}^{n \times n}$ with entries
\begin{equation}
    K_{\epsilon,\beta}^{(ij)} = K_{\hat \psi}(\epsilon, \bm x^{(i)}, \bm x^{(j)}) := \exp{\left( - \frac{\| \boldsymbol{x}^{(i)} - \boldsymbol{x}^{(j)} \|^2}{4 \epsilon^2 (\hat \psi(\bm x^{(i)} ) \hat \psi( \bm x^{(j)}) )^{\beta}} \right)}
    \label{eq:variable-bandwidth-kernel-matrix}
\end{equation}
(again, this is a variable bandwidth kernel of the form~\eqref{eq:variable-bandwidth-kernel}) and normalizing function $q_{\epsilon, \beta} : \Omega \to \mathbb R^+ $ such that
\begin{displaymath}
    q_{\epsilon,\beta}(\bm x ) = \frac{1}{ \hat \psi^{\beta d}( \bm x )} \sum_{j=1}^{n} K_{\hat \psi}(\epsilon, \bm x, \bm x^{(j)}).
\end{displaymath}
We note that the bandwidth parameter $\epsilon$ need not be the same as the bandwidth parameter in the density estimation. Again, we typically set $K_{\epsilon, \beta}^{(ij)} = 0$ if the corresponding entry is below a prescribed threshold so that $\boldsymbol{K}_{\epsilon, \beta}$ is sparse. In the special case when $\beta = 0$,  $\boldsymbol{K}_{\epsilon, \beta}$ simplifies to a fixed-bandwidth kernel matrix. 

Next, define the normalized variable-bandwidth kernel given another real parameter $\alpha$
\begin{displaymath}
    \tilde K_{\hat\psi,\beta,\alpha}( \epsilon, \bm x, \bm y ) = \frac{K_{\hat \psi}(\epsilon,\bm x, \bm y)}{(q_{\epsilon, \beta} (\bm x )q_{\epsilon,\beta}(\bm y))^\alpha}
\end{displaymath}
and an additional normalizing function $q_{\epsilon, \beta, \alpha} : \Omega \to \mathbb R^+$ such that
\begin{displaymath}
    q_{\epsilon, \beta, \alpha}(\bm x) = \sum_{j=1}^{n} \tilde K_{\hat \psi, \beta, \alpha}(\epsilon, \bm x, \bm x^{(j)}).
\end{displaymath}
Let  $\boldsymbol{K}_{\epsilon,\beta,\alpha} \in \mathbb{R}^{n \times n} $ be the normalized kernel matrix with entries
\begin{equation}
    K_{\epsilon,\beta,\alpha}^{(ij)} = \tilde K_{\hat \psi,\beta,\alpha}(\epsilon, \bm x^{(i)}, \bm x^{(j)}),
    \label{eq:normalized-variable-bandwidth-kernel-matrix}
\end{equation}
and let $\boldsymbol{P}$, $\boldsymbol{D}$ be $n \times n$ diagonal matrices such that $P^{(ii)} = \hat \psi^\beta(\bm x^{(i)}) $ and $D^{(ii)} = q_{\epsilon, \beta, \alpha}(\bm x^{(i)})$. Berry and Harlim \cite{berryharlim2016} show that the matrix 
\begin{equation} 
    \boldsymbol{L}_{\psi, c} = \epsilon^{-2} \boldsymbol{P}^{-2} (\boldsymbol{D}^{-1} \boldsymbol{K}_{\epsilon, \beta, \alpha} - \boldsymbol{I})
    \label{eq:discrete-Kolomogorov-matrix}
\end{equation}
approximates the Kolmogorov operator $\mathcal L_{\psi, c}$ in~\eqref{eq:Kolmogorov} with $c = 2 - 2 \alpha + d \beta + 2 \beta$, in the sense that  
\begin{equation}
    \label{eqLApprox}
    (\bm L_{\psi, c} \bm f)^{(i)} \approx \mathcal L_{\psi, c} f(\bm x^{(i)}).
\end{equation}
\edits{A precise error estimate of the approximation in~\eqref{eqLApprox} can be found in Corollary~1 of \cite{berryharlim2016}. In broad terms, the approximation converges in a limit of $\epsilon \to 0 $ after $n \to \infty $ (or for an appropriate decreasing sequence $\epsilon_n$) for fixed $ f \in C^3(\Omega) \cap \mathcal H_{\psi,c}(\Omega)$ and $\bm x^{(i)} \in \Omega$. Moreover, the convergence is uniform with respect to $\bm x^{(i)}$ if $f$ and its derivatives vanish at infinity.}  Intuitively, this result follows from~\eqref{eq:integral-Taylor-expansion} since $\epsilon^2 \mathcal{L}_{\psi} f \approx \mathcal{I}_{\epsilon} f - f$. To approximate $\mathcal{L}_{\psi, c}$ using~\eqref{eq:discrete-Kolomogorov-matrix}, we first use $\boldsymbol{D}$ to normalize the kernel matrix---$\boldsymbol{D}^{-1} \boldsymbol{K}_{\epsilon, \beta, \alpha}$ is a Markov matrix. We approximate the Taylor expansion in~\eqref{eq:integral-Taylor-expansion} by subtracting the identity matrix, but need to weight the result by $\boldsymbol{P}$ to account for the measure $\psi^c$ in the integral operator of~\eqref{eq:integral-kernel-operator}.

\edits{Although the approximation in~\eqref{eqLApprox} converges for manifolds of arbitrary intrinsic dimension $d$, the rate of convergence may suffer from a ``curse of dimensionality.'' That is, without making additional assumptions on $f$, the dataset size $n$ required to attain a given approximation accuracy may grow exponentially in $d$. Moreover, the proportionality constants in the error estimates increase as the geometrical complexity of the embedding of $\Omega$ in $\mathbb R^d$ (quantified, e.g., by the second fundamental form) increases. These issues are common to many manifold learning algorithms and can be alleviated if $f$ possesses additional regularity properties, such as lying in a Sobolev space of sufficiently high order. Even without making such assumptions, the error estimates depend on the intrinsic dimension $d$ of $\Omega$, as opposed to the dimension $m$ of the ambient data space $\Omega$, which is oftentimes far greater than $d$. In Section~\ref{sec:Gaussian-toy-example}, we demonstrate that our methods are effective in computing numerical solutions to the differential equation $\mathcal{L}_{c, \psi} f = g$ in a moderate-dimensional setting.} 


For the purposes of this work, we are interested in approximating the eigenvalues and  eigenfunctions of $\mathcal L_{\psi, c}$ in~\eqref{eqEigs} by the eigenvalues and eigenvectors of $\bm L_{\psi,c}$, i.e.
\begin{equation}
    \label{eqDiscreteEigs}
    \bm L_{\psi,c} \bm \phi_j = \hat \lambda_j \bm \phi_j.
\end{equation}
We take advantage of specialized solvers by working with the symmetric matrix 
\begin{equation}    
    \boldsymbol{\hat{L}}_{\psi, c} = \epsilon^{-2} (\boldsymbol{S}^{-1} \boldsymbol{K}_{\epsilon, \beta, \alpha} \boldsymbol{S}^{-1} - \boldsymbol{P}^{-2}),
    \label{eq:symmetric-discrete-Kolmogorov-operator}
\end{equation}
where $\boldsymbol{S} = \boldsymbol{P} \boldsymbol{D}^{1/2}$ is a diagonal matrix. One can readily verify that $\boldsymbol{\hat{L}}_{\psi, c}$ and $\boldsymbol{L}_{\psi, c}$ are related by the similarity transformation $\boldsymbol{\hat{L}}_{\psi, c} = \boldsymbol{S} \boldsymbol{L}_{\psi, c} \boldsymbol{S}^{-1}$. Therefore, if $\boldsymbol{\hat{\phi}}_j$ is an eigenvector of $\boldsymbol{\hat{L}}_{\psi,c}$ at eigenvalue $\hat \lambda_j$, then $ \bm \phi_j = \boldsymbol{S}^{-1} \boldsymbol{\hat{\phi}}_j$ is an eigenvector of $\boldsymbol{L}_{\psi, c}$ at the same eigenvalue. The symmetry of $\boldsymbol{\hat{L}}_{\psi, c}$ implies that the eigenvectors $\boldsymbol{\hat{\phi}}_j$ can be chosen such that $\boldsymbol{\hat{\phi}}_i^\top \boldsymbol{\hat{\phi}}_j = n \delta_{ij}$. This normalization implies that the $\boldsymbol{\hat{\phi}}_i$ are orthonormal with respect to the empirical sampling measure of the data, which approximates the density $\psi$. The eigenvectors of $\boldsymbol{L}_{\psi,c}$ are then orthonormal with respect to the weighted inner product 
\begin{equation}
    \label{eqInnerProd}
    \langle \bm f, \bm g \rangle_{\bm S} := \bm f^\top \bm S^2 \bm g / n,
\end{equation}
which approximates the inner product of $\mathcal H_{\psi,c}$. That is, our numerical eigenvectors satisfy $ \langle \bm \phi_i,  \bm \phi_j \rangle_{\bm S} = \delta_{ij} $, which is the discrete analog of $\langle \phi_i, \phi_j \rangle_{\psi,c} = \delta_{ij} $. \edits{As representative eigendecomposition results,} Fig.~\ref{fig:laplace-operator-eigenvalues} shows the first 100 eigenvalues $\hat \lambda_i$, ordered in order of increasing magnitude, and the fourth eigenvector $\bm \phi_4$ of $\boldsymbol{L}_{\psi, c}$, computed using samples from a standard Gaussian distribution on $\Omega = \mathbb{R}^{2}$. \edits{We will discuss these results in more detail in Section~\ref{sec:Gaussian-toy-example}.}


\begin{figure}
  \centering
  \begin{subfigure}{0.45\textwidth}
    \includegraphics[width=1.0\textwidth]{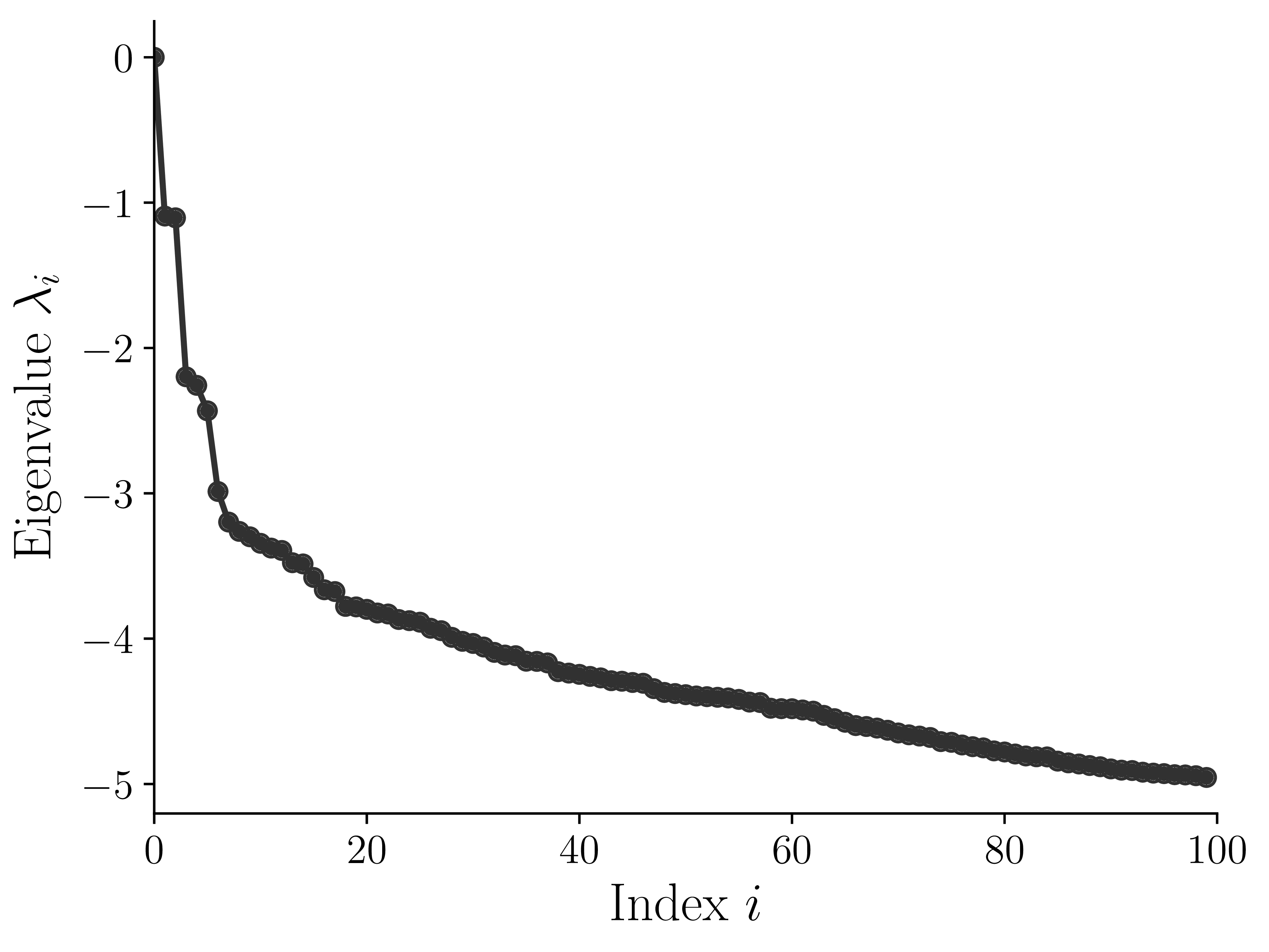}
    \caption{Eigenvalues of $\boldsymbol{L}_{\psi, c}$}
  \end{subfigure}
  \begin{subfigure}{0.45\textwidth}
    \includegraphics[width=1.0\textwidth]{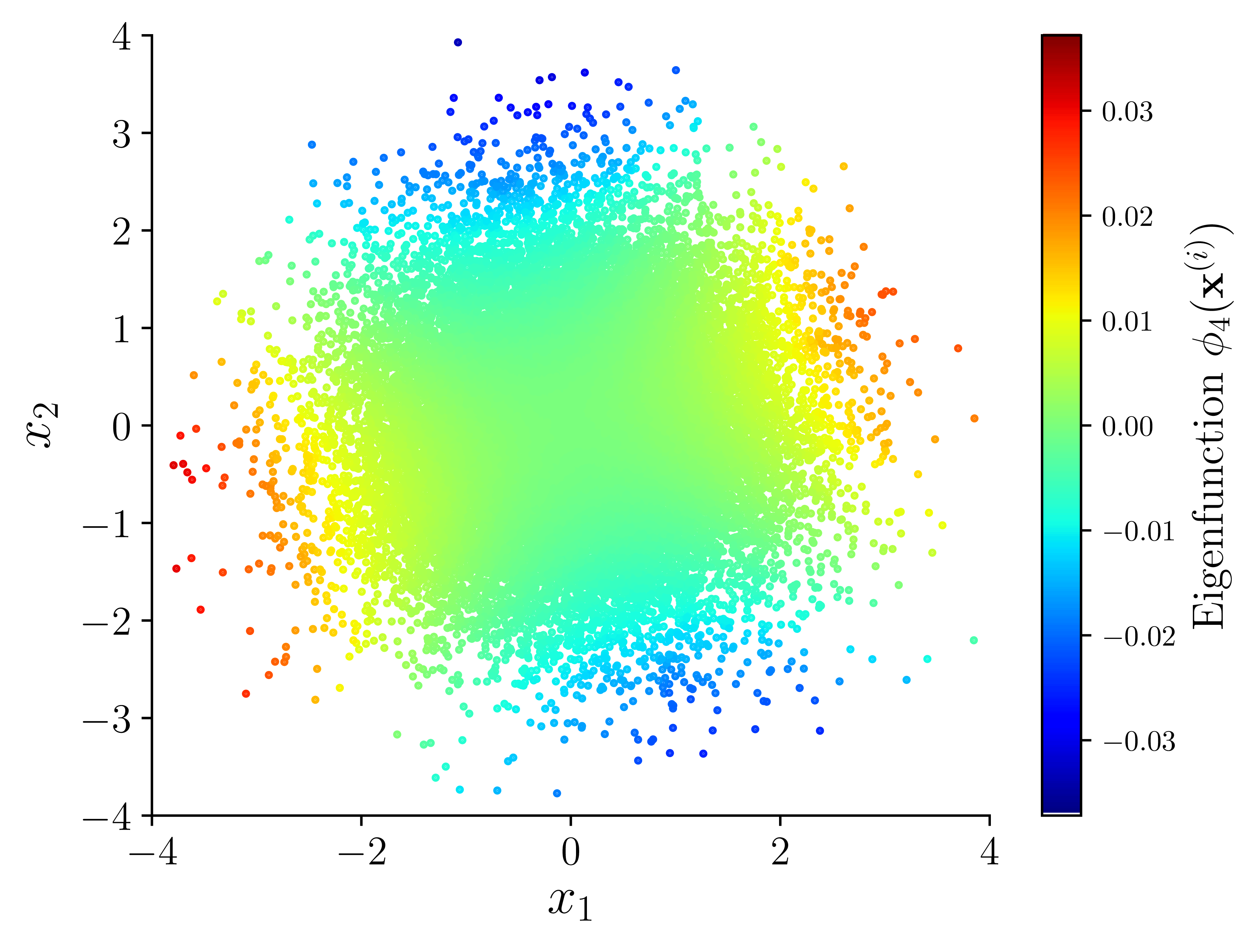}
    \caption{Eigenfunction $\bm\phi_4$ of $\boldsymbol{L}_{\psi, c}$}
    \label{fig:computed-fourth-eigenvector}
  \end{subfigure}
  \caption{\edits{Eigendecomposition of the discrete Kolmogorov operator $\boldsymbol{L}_{\psi, c}$ computed using $n=2.5 \times 10^{4}$ samples from a standard Gaussian distribution $\mathcal{N}(\boldsymbol{0}, \boldsymbol{I})$ on $\Omega = \mathbb{R}^{2}$. Panel~(a) shows the first 100 eigenvalues $\hat \lambda_i $ (in decreasing order). Panel~(b) shows} the eigenvector $\bm \phi_4$ corresponding to the $4^{th}$ smallest eigenvalue $\hat \lambda_j $. We use the same density estimation parameters as in Fig.~\ref{fig:density-estimation} along with numerical parameters $\beta = -0.25$ and $c = 1$ (implying $\alpha=0$ since $d=2$). We estimate the optimal bandwidth parameter $\epsilon$ using the procedure outlined in Section \ref{sec:optimal-bandwidth}.}
  \label{fig:laplace-operator-eigenvalues}
\end{figure}

\subsection{Bandwidth parameter tuning} \label{sec:optimal-bandwidth}

The procedures estimating the density $\psi$ and computing the discrete weighted Laplace operator $\boldsymbol{L}_{\psi, c}$ are sensitive to the choice of bandwidth parameter $\epsilon$. We, therefore, use an automatic tuning procedure described by Coifman et al.\ \cite{CoifmanEtAl08}, which is also employed in \cite{berryharlim2016}. See Algorithm~1 \cite{giannakis2019} for pseudocode. 

Let $\xi \in \mathbb{R}$ parameterize the bandwidth parameter $\epsilon_{\xi}^2 = 2^{\xi}$, and define
\begin{equation}
    \chi_{\xi} = \displaystyle\sum_{i,j=1}^{n} K_{\epsilon_{\xi}}^{(ij)}\quad \mbox{and}, \quad \chi_{\xi, \beta} = \sum_{i,j=1}^{n} K_{\epsilon_{\xi}, \beta}^{(ij)}.
\end{equation}
The optimal bandwidth parameters for the density estimation and operator estimation procedures maximize
\begin{equation}
    \chi_{\xi}^{\prime} = \frac{\log{(\chi_{\xi+\delta})}-\log{(\chi_{\xi})}}{\delta \log{(2)}} \quad \mbox{and} \quad  \chi_{\xi, \beta}^{\prime} = \frac{\log{(\chi_{\xi+\delta, \beta})}-\log{(\chi_{\xi, \beta})}}{\delta \log{(2)}},
    \label{eq:bandwidth-sensitivity}
\end{equation}
respectively. A scaling argument \cite{CoifmanEtAl08} shows that for a well-tuned kernel, the maximum values of $\chi^{\prime}_\xi$ and $\chi^{\prime}_{\xi,\beta}$ approximate the intrinsic dimension of the manifold $\Omega$. Figure~\ref{fig:bandwidth-tuning} plots $\chi_{\xi}'$ and $\chi_{\xi,\beta}'$ as a function of $\epsilon$, and shows the optimal bandwidth parameter used in the $2$-dimensional Gaussian examples of Figs.~\ref{fig:density-estimation}~and~\ref{fig:laplace-operator-eigenvalues}. Practically, we compute the maximum using a derivative-free optimization algorithm available in \texttt{NLopt} \cite{nlopt}.

\begin{figure}[h!]
  \centering
  \begin{subfigure}{0.45\textwidth}
    \includegraphics[width=1.0\textwidth]{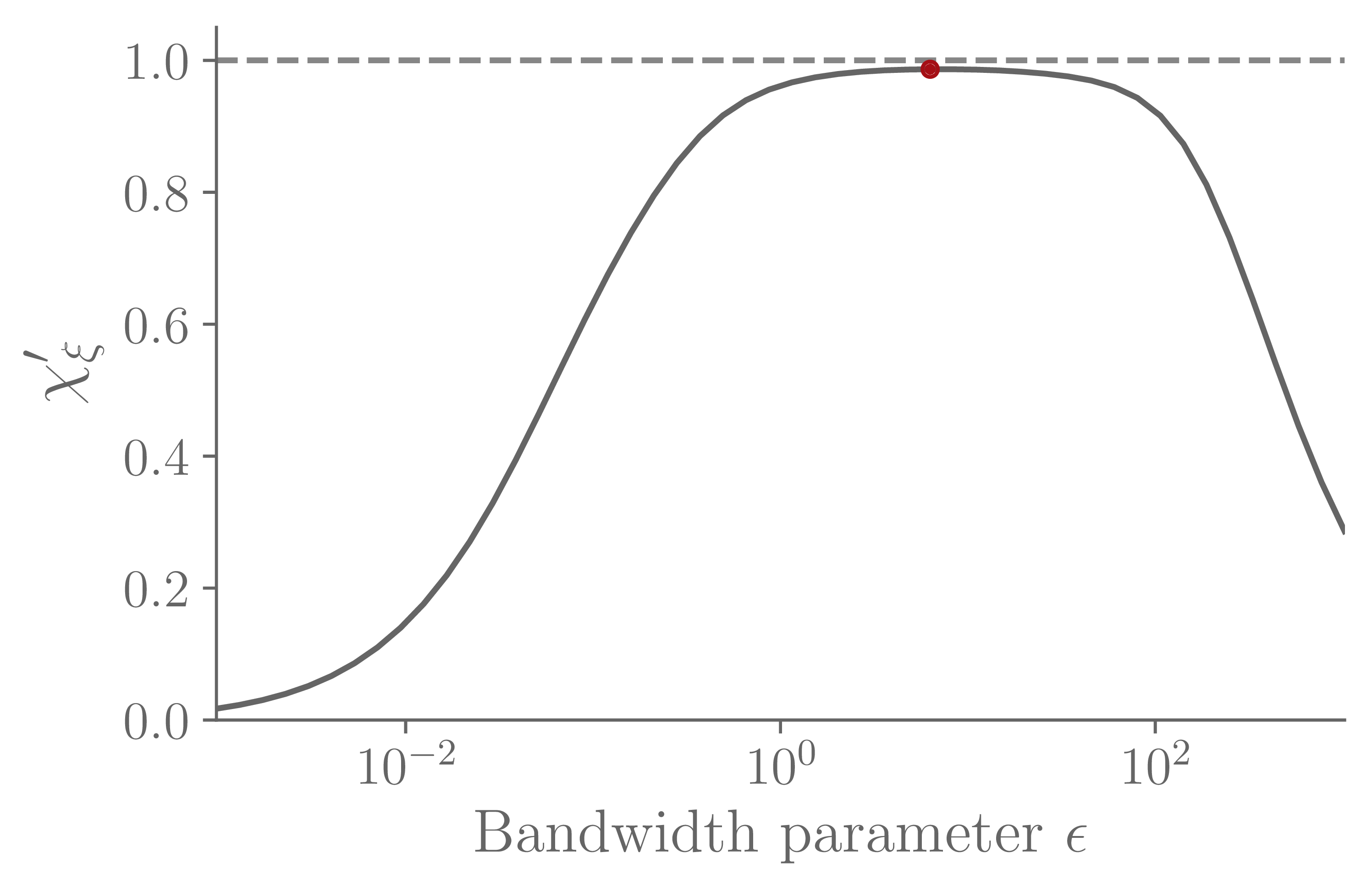}
    \caption{Density estimation}
  \end{subfigure}
  \begin{subfigure}{0.45\textwidth}
    \includegraphics[width=1.0\textwidth]{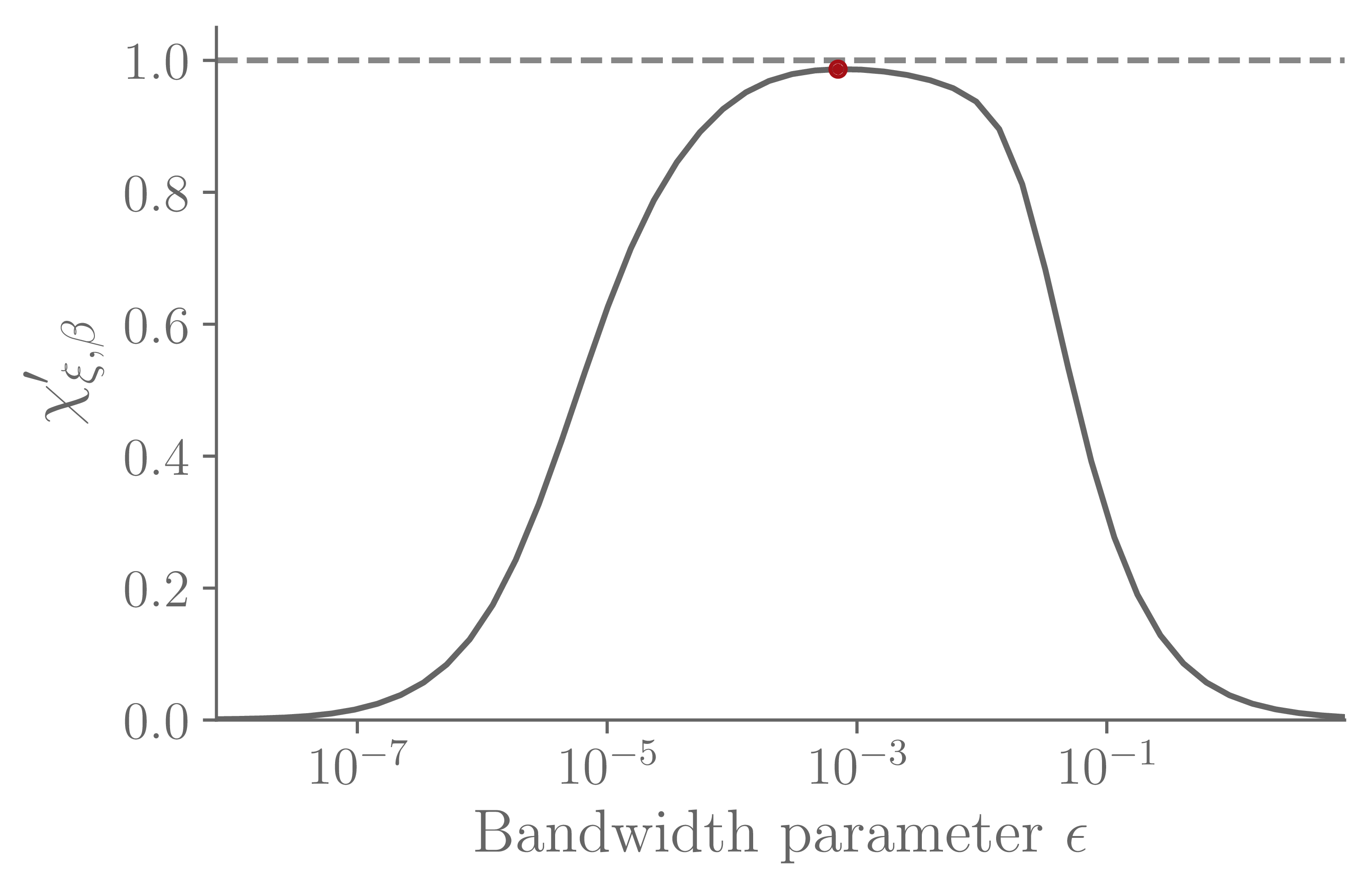}
    \caption{Operator estimation}
  \end{subfigure}
  \caption{The optimal bandwidth parameter $\epsilon$ for (a) density and (b) operator estimation maximizes $\chi_{\xi}'$ and $\chi_{\xi,\beta}'$ from~\eqref{eq:bandwidth-sensitivity}, respectively. In both cases, we set the step size parameter $\delta = 1$. The manifold dimension $d$ is estimated by $ 2 \max{(\chi_{\epsilon}^{\prime})} $ or $  2 \max{(\chi_{\epsilon, \beta}^{\prime})}$, shown in horizontal dashed lines. In this example, $d=2$.}
  \label{fig:bandwidth-tuning}
\end{figure}

\section{Spectral representation of the Kolmogorov problem} \label{sec:solution-gradients}

We compute solutions to the Kolmogorov problem in~\eqref{eq:Kolmogorov-problem}, and represent gradient vector fields, by expanding functions using the spectral basis of the Kolmogorov operator. We use estimates of the eigenvalues and eigenfunctions $( \lambda_j, \phi_j)$ from~\eqref{eqEigs} to solve the Kolmogorov problem (i.e., invert the operator) and/or estimate derivatives. Specifically, given function values $g(\boldsymbol{x}^{(i)})$ at samples from the probability measure $\boldsymbol{x}^{(i)} \sim \psi$, we approximate the solution $f$ of~\eqref{eq:Kolmogorov-problem}, and estimate the gradient vector fields $\nabla f$ and/or $\nabla g$. We first discuss how to use the eigendecomposition of the discrete Kolmogorov operator in~\eqref{eqDiscreteEigs} to solve the Kolmogorov problem, and then discuss how to represent gradient vector fields. 


\subsection{Solutions of the Kolmogorov problem} \label{sec:solution}

Let $H$ be the $n$-dimensional inner product space $\mathbb R^n$ equipped with the $\langle \cdot, \cdot \rangle_{\bm S}$ inner product from~\eqref{eqInnerProd} and the corresponding norm $ \lVert \bm f \rVert_{\bm S} = \sqrt{\langle \bm f, \bm f \rangle_{\bm S}}$. For any positive integer parameter $\ell \leq n$, define the $\ell$-dimensional subspace $H_\ell \subseteq H$ spanned by the leading $\ell$ nonzero eigenvectors of $\bm L_{\psi,c}$ with nonzero corresponding eigenvalues, $ H_\ell = \spn \{ \bm \phi_1, \ldots, \bm \phi_\ell \}$. Let also $ \bm Q_\ell $ be the $n \times \ell$ matrix whose columns are the eigenvectors $\bm \phi_j$, i.e., $\bm Q_\ell = ( \bm \phi_1 \cdots \bm \phi_\ell)$. Viewed as a linear map $\bm Q_\ell : \mathbb R^\ell \to H $, $\bm Q_\ell$ is a 1-1 map with range $H_\ell $, and preserves the standard inner product of $\mathbb R^\ell$, i.e., $ \bm u^\top \bm v = \langle \bm Q_\ell \bm u, \bm Q_\ell \bm v \rangle_{\bm S}$ for any $ \bm u, \bm v \in \mathbb R^\ell$.

Given the discrete Kolmogorov operator $\boldsymbol{L}_{\psi, c}$, we approximate the solution $f$ of the continuous problem as the minimizer $ \bm f \in H_\ell$ of the square error functional $\mathcal E_\ell : H_\ell \to \mathbb R^+ $ with
\begin{displaymath}
    \mathcal E_\ell(\bm h) = \lVert \bm L_{\psi, c} \bm h - \bm g \rVert^2_{\bm S}.
\end{displaymath}
Setting $ \bm h = \bm Q_\ell \widetilde{ \bm h } $, where $\widetilde{ \bm h } \in \mathbb R^\ell $ is the (unique) vector of expansion coefficients of $\bm h \in H_\ell$ with respect to the $ \{ \bm \phi_j \} $ basis, the minimization of $\mathcal E_{\ell}$ is equivalent to the minimization of $\widetilde {\mathcal E}_\ell : \mathbb R^\ell \to \mathbb R_+$, where 
\begin{displaymath}
    \widetilde {\mathcal E}_\ell(\widetilde{\bm h}) = \lVert \bm L_{\psi, c} \bm Q_{\ell} \widetilde{\bm h} - \bm g \rVert^2_{\bm S}.
\end{displaymath}
The minimizer of this functional is 
\begin{equation}
    \label{eqLSMinimizer}
    \tilde{\bm f} = (\bm L_{\psi, c} \bm Q_{\ell})^\dag \bm g = \bm \Lambda^{-1}_\ell \widetilde{\bm g}_\ell, \quad \widetilde{ \bm g}_\ell := \frac{1}{n} \bm Q_\ell^\top \bm S^2 \bm g,
\end{equation}
where $^\dag $ denotes the pseudoinverse of linear maps from $\mathbb R^\ell$ to $H$, and $\bm \Lambda_\ell$ is the diagonal matrix with diagonal entries $\Lambda^{(ii)} = \hat \lambda_i$. Moreover, the vector $\widetilde{\bm g}_\ell \in \mathbb R^\ell$ gives the expansion coefficients of the orthogonal projection $\bm g_\ell$ of $\bm g $ onto $H_\ell$ with respect to the $ \{ \bm \phi_j \} $ basis; specifically, $\bm g_\ell = \bm Q_\ell \widetilde{\bm g}_\ell $. Figure~\ref{fig:function-expansion} shows $\bm g_\ell $ for $\ell = 100$ associated with the continuous function $g(\boldsymbol{x}) = x_1$ on $\Omega = \mathbb R^2$ for a dataset of $n=10^4$ samples from a standard Gaussian distribution in $\mathbb{R}^{2}$ as in Fig.~\ref{fig:laplace-operator-eigenvalues}. \edits{See Section~\ref{sec:Gaussian-toy-example} for a discussion of these results.}

Having computed $\bm{\tilde{f}}$, we then construct the vector 
\begin{equation}
    \label{eqFLS}
    \bm f = \bm Q_\ell \widetilde{\bm f},
\end{equation}
which is the (unique) least-squares solution of the discrete Kolmogorov problem in~\eqref{eqDiscreteKolmogorov} lying in the $\ell$-dimensional hypothesis space $H_\ell$. Note that $H_\ell$ is, by construction, orthogonal to the constant eigenvector $\bm \phi_0$, which enforces the integral condition in the continuous problem in~\eqref{eq:Kolmogorov-problem} in the discrete setting. It should also be noted that in the special case $\ell = n $ with $ H_\ell = H $ the least-squares solution described above is actually a solution of~\eqref{eqDiscreteKolmogorov}; in particular, it is the unique solution orthogonal to $\bm \phi_0 $. While this solution can be computed efficiently without eigendecomposition of $\bm L_{\psi,c}$ (which would be a computationally expensive task for $\ell = n \gg 1 $), in what follows we will need the eigenvectors $\bm \phi_j $ in order to approximate the solution gradient $\nabla f $. Thus, the least-squares approach with $\ell \ll n $ is our solution method of choice in this paper, despite the additional cost of the eigendecomposition step. Restricting the solution in the subspace spanned by the leading (i.e., least oscillatory) eigenvectors of $\bm L_{\psi,c}$ also improves robustness to noise and sampling errors.


\begin{figure}
  \centering
  \begin{subfigure}{0.45\textwidth}
    \includegraphics[width=1.0\textwidth]{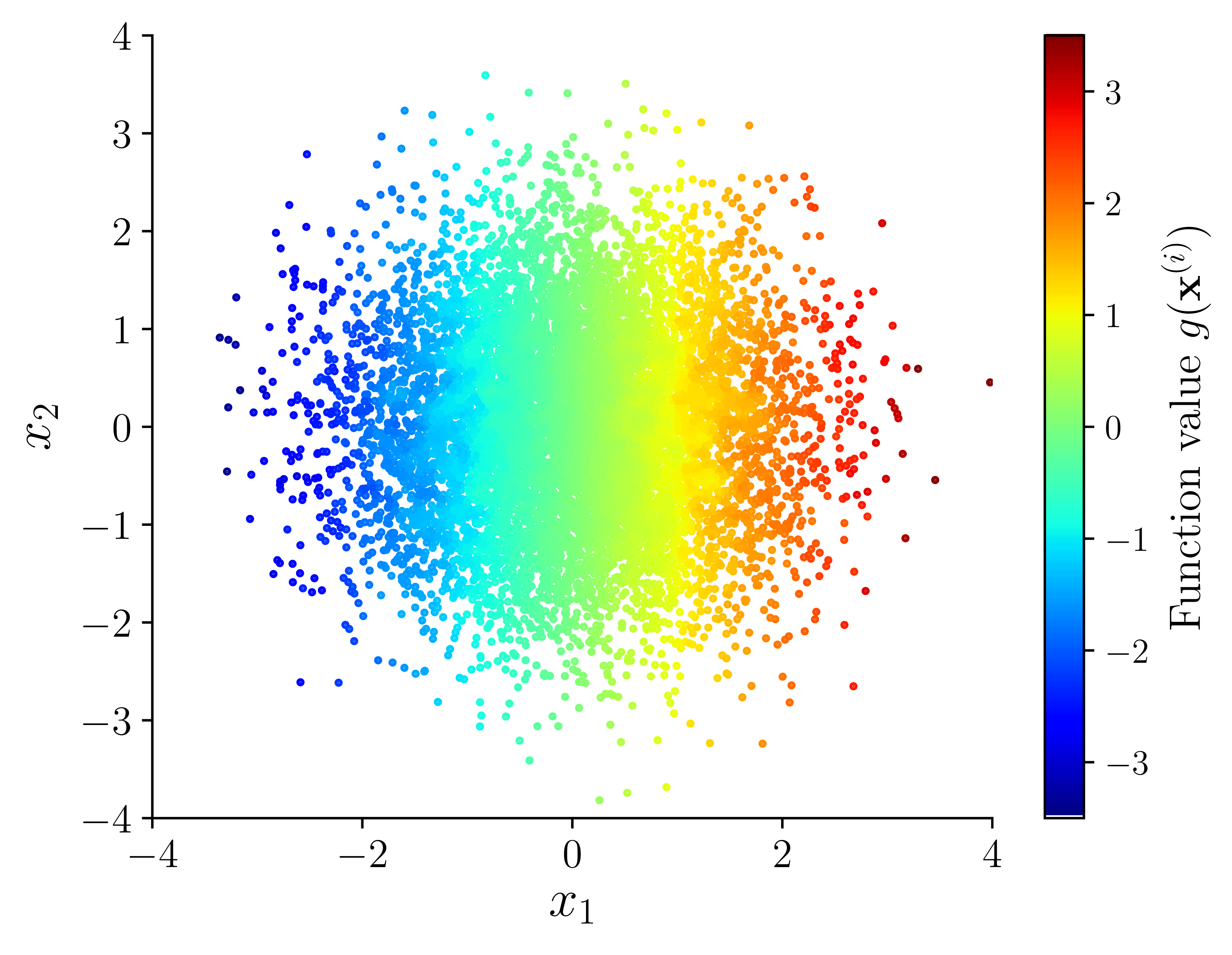}
    \caption{Spectral expansion $\boldsymbol{g}_\ell = \boldsymbol{Q}_\ell \boldsymbol{\widetilde{g}}$}
  \end{subfigure}
  \begin{subfigure}{0.45\textwidth}
    \includegraphics[width=1.0\textwidth]{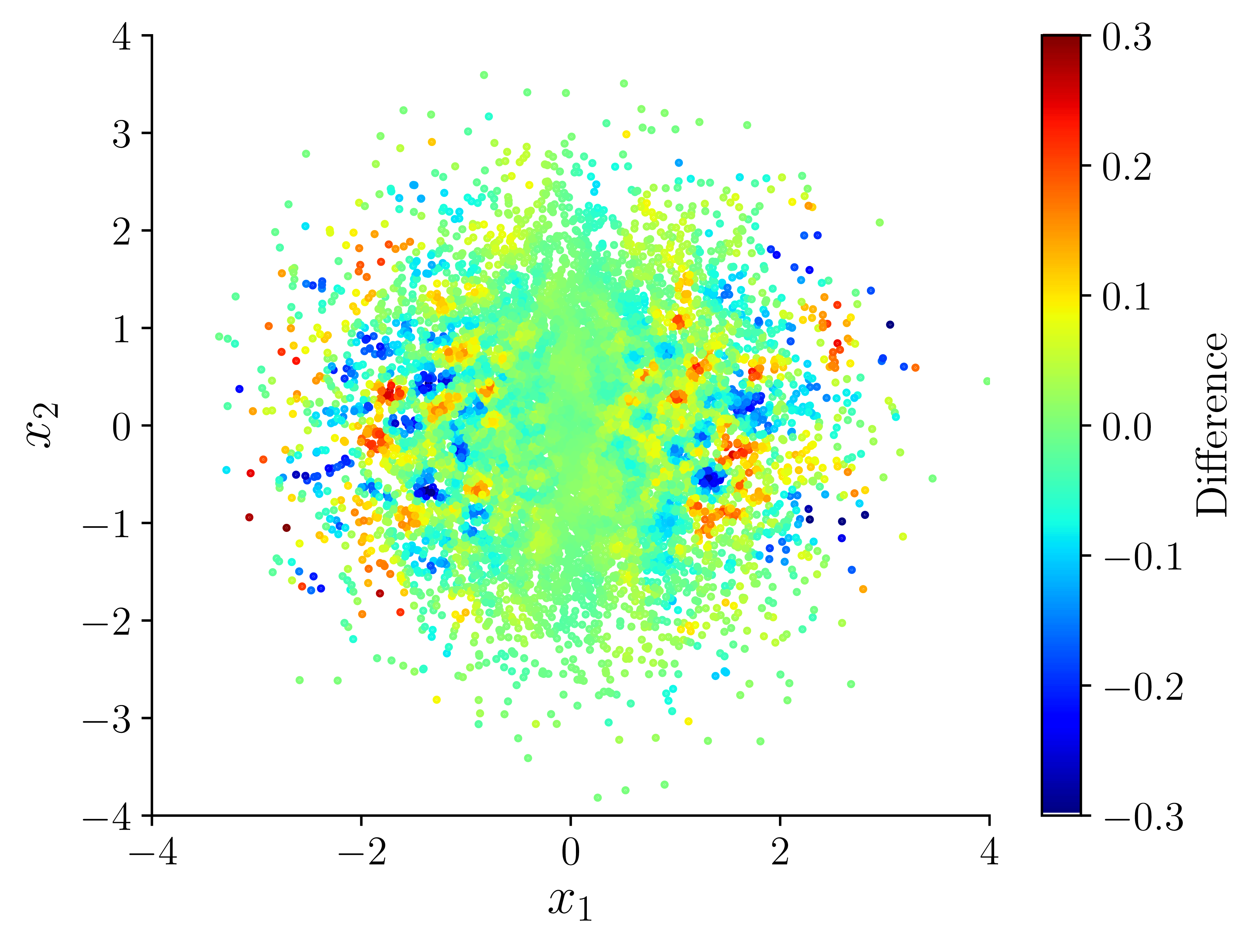}
    \caption{Difference $g(\boldsymbol{x}^{(i)})-g^{(i)}$}
  \end{subfigure}
  \caption{(a) The function $g(\boldsymbol{x}) = x_1$ expanded using the $100$ eigenvectors of $\boldsymbol{L}_{\phi, c}$ computed using the same parameters as Fig.~\ref{fig:laplace-operator-eigenvalues}. (b) The difference between the estimated function at each sample and the true function values, $g(\boldsymbol{x}^{(i)})-g^{(i)}$, where $g^{(i)}$ is the $i^{th}$ component of the vector $\boldsymbol{g}_\ell = \boldsymbol{Q}_\ell \boldsymbol{\widetilde{g}}_\ell$.}
  \label{fig:function-expansion}
\end{figure}

%

\subsection{\label{sec:gradients}Gradient representation using the eigenbasis}

Given the discrete solution $\bm f$ in~\eqref{eqFLS}, our next task is to approximate the gradient vector field $\nabla f$. For that, we employ an analogous approach to that developed in the context of the spectral exterior calculus \cite{BerryGiannakis20}, which is based on the so-called Carr\'e du Champ identity (product rule) \cite{bakryetal2013} for the Laplace-Beltrami operator on a Riemannian manifold. Lemma~\ref{lemma:product-rule} states a simple generalization of this identity for the class of Kolmogorov operators $\mathcal L_{\psi,c}$ in~\eqref{eq:Kolmogorov}.
\begin{lemma}[Carr\'e du Champ identity for $\mathcal{L}_{\psi, c}$]
    For any two functions $u,v \in C^2(\Omega)$ the Carr\'e du Champ identity 
    \begin{displaymath}
        \mathcal{L}_{\psi, c} (uv) = u \mathcal{L}_{\psi, c} v + v \mathcal{L}_{\psi, c} u + 2 \nabla u \cdot \nabla v
    \end{displaymath}
holds.
\label{lemma:product-rule}
\end{lemma}
\begin{proof}
 The identity follows from the definition 
\begin{eqnarray*}
    \mathcal{L}_{\psi, c} (uv) &=& \Delta (uv) + c \nabla (uv) \cdot \frac{\nabla \psi}{\psi} \\
    &=& u \nabla^2 v + v \nabla^2 u + 2 \nabla u \cdot \nabla v + c u \nabla v \cdot \frac{\nabla \psi}{\psi} + c v \nabla u \cdot \frac{\nabla \psi}{\psi} \\
    &=& u \mathcal{L}_{\psi, c} v + v \mathcal{L}_{\psi, c} u + 2 \nabla u \cdot \nabla v.\hspace{4mm} \qed
\end{eqnarray*}
\end{proof}

Lemma~\ref{lemma:product-rule} allows us to express the Riemannian inner product $\nabla u \cdot \nabla v$, which is independent of the density $\psi$, in terms of the action of the Kolmogorov operator on the functions $u$ and $v$. Choosing $v(\boldsymbol{x}) = x_{s}$---where $x_{s}$ denotes the $s^{th}$ coordinate of $\boldsymbol{x}$---for example, isolates the derivative of $u$ with respect to the $s^{th}$ coordinate, i.e., 
\begin{equation}
    \label{eqGradU}
    \frac{\partial u}{\partial x_s} = \nabla u \cdot \nabla v = \frac{\mathcal{L}_{\psi, c} (uv) - u \mathcal{L}_{\psi, c} v - v \mathcal{L}_{\psi, c} u}{2}.
\end{equation}

Given $u,v \in C^2(\Omega) \cap \mathcal H_{\psi,c}$, we can represent $\nabla u \cdot \nabla v$ using the eigenbasis of $\mathcal{L}_{\psi, c}$ \cite{BerryGiannakis20}. Recall that $(\lambda_j, \phi_j)$ are eigenvalue/eigenfunction pairs such that $\mathcal{L}_{\psi, c} \phi_j = \lambda_j \phi_j$. The identity in Lemma~\ref{lemma:product-rule} implies that 
\begin{equation}
    \nabla \phi_j \cdot \nabla \phi_k = \frac{1}{2} ( \mathcal{L}_{\psi, c} (\phi_j \phi_k) - (\lambda_j + \lambda_k) \phi_j \phi_k ).
    \label{eq:eigenfunction-gradient-dot-product}
\end{equation}
Moreover, defining the coefficients
\begin{equation}
    C_{ljk} = \langle \phi_l, \phi_j \phi_k \rangle_{\psi,c} = \int_\Omega \phi_l(\bm x) \phi_j(\bm x) \phi_k(\bm x) \psi^c(\bm x) \, d\boldsymbol{x}, 
    \label{eqCIJK}
\end{equation}
we have
\begin{displaymath}
    \phi_j \phi_k = \sum_{l=0}^{\infty} C_{ljk} \phi_l,
\end{displaymath}
where the sum over $l$ in the right-hand side converges in the $\mathcal H_{\psi,c}$ norm. Substituting this into~\eqref{eq:eigenfunction-gradient-dot-product} derives
\begin{displaymath}
    \nabla \phi_j \cdot \nabla \phi_k = \sum_{l=0}^{\infty} \frac{C_{ljk}}{2} (\lambda_l - \lambda_j - \lambda_k) \phi_l.
\end{displaymath}
Decomposing $ u = \sum_{j=0}^\infty \tilde u_j \phi_j$ and $ v = \sum_{k=0}^\infty \tilde v_k \phi_k$, where $ \tilde u_j = \langle \phi_j, u \rangle_{\psi,c}$ and $\tilde v_k = \langle \phi_k, v \rangle_{\psi,c}$, it then follows that
\begin{align*}
    \nabla u \cdot \nabla v &= \sum_{j,k=0}^\infty \tilde u_j \tilde v_k  \nabla \phi_j \cdot \nabla \phi_k \\
    &= \sum_{j,k,l=0}^\infty \tilde u_j \tilde v_k \frac{C_{ljk}}{2} (\lambda_l - \lambda_j - \lambda_k) \phi_l.
\end{align*}
Choosing, as in~\eqref{eqGradU}, the function $v(\boldsymbol{x}) = x_s$ that picks the $s^{th}$ coordinate index, and computing  the expansion coefficients
\begin{displaymath}
    \tilde{s}_k \equiv \tilde{v}_k = \langle \phi_k, v \rangle_{\psi,c} = \int_\Omega \phi_k(\bm x) x_s \psi^c(\bm x) \, d\bm x, 
\end{displaymath}
we arrive at the relation
\begin{equation}
    \label{eqGradUPhi}
    \frac{\partial u}{\partial x_s} = \sum_{j,k,l=0}^\infty \tilde u_j \tilde s_k \frac{C_{ljk}}{2} (\lambda_l - \lambda_j - \lambda_k) \phi_l.
\end{equation}

Equation~\eqref{eqGradUPhi} provides a spectral representation of the components of gradients of functions $u$ with respect to the Riemannian metric on $\Omega$. Importantly, each of the terms in the right-hand-side has a direct data-driven analog obtained from the eigendecomposition of the discrete Kolmogorov operator $\bm L_{\psi,c}$. To effect that approximation, we replace the eigenpairs $(\lambda_j,\phi_j)$ with their discrete counterparts $(\hat \lambda_j, \bm \phi_j)$, compute the corresponding coefficients $ \hat C_{ljk} := \langle \bm \phi_l, \bm \phi_j * \bm \phi_k \rangle_{\bm S} $ (cf.~\eqref{eqCIJK}),  where $*$ is the component-wise product operator between vectors, and approximate the expansion coefficients $ \tilde u_j$ and $\tilde v_k^{(s)}$ by $\hat u_j = \langle \bm u, \bm \phi_j \rangle_{\bm S} $ and $\hat v_k^{(s)} = \langle \bm v^{(s)}, \bm \phi_k \rangle_{\bm S}$, respectively. We also truncate the infinite sums over $l,j,k$ after $\ell$ eigenvalue--eigenvector pairs. 

Figure~\ref{fig:gradient-example} shows the result of this procedure for the function $g(\boldsymbol{x}) = \boldsymbol{x} \cdot \boldsymbol{r}$ for a constant vector $\boldsymbol{r} \in \mathbb{R}^{2}$. We know $\nabla g = \boldsymbol{r}$ analytically, so in Fig.~\ref{fig:gradient-example} we compare the analytical and data-driven gradients for the prescribed function $g$. In this paper, our primary interest is to apply the discrete analog of~\eqref{eqGradUPhi} to approximate the gradient $\nabla g$, where $\bm g$ is given by~\eqref{eqFLS}. Figure~\ref{fig:weighted-laplace-solution} shows the approximate solution $\bm f$ for the weighted Kolmogorov problem with right-hand-side function $g(\boldsymbol{x})$. We present solutions given parameter choices $c=0.1$ and $c=1$. The latter more strongly biases the curvature of the distribution, which is reflected in Fig.~\ref{fig:weighted-laplace-solution}. We investigate the effect of varying $c$ in more detail in Section~\ref{sec:spherical-manifold-example}.


\begin{figure}
  \centering
  \begin{subfigure}{0.45\textwidth}
    \includegraphics[width=1.0\textwidth]{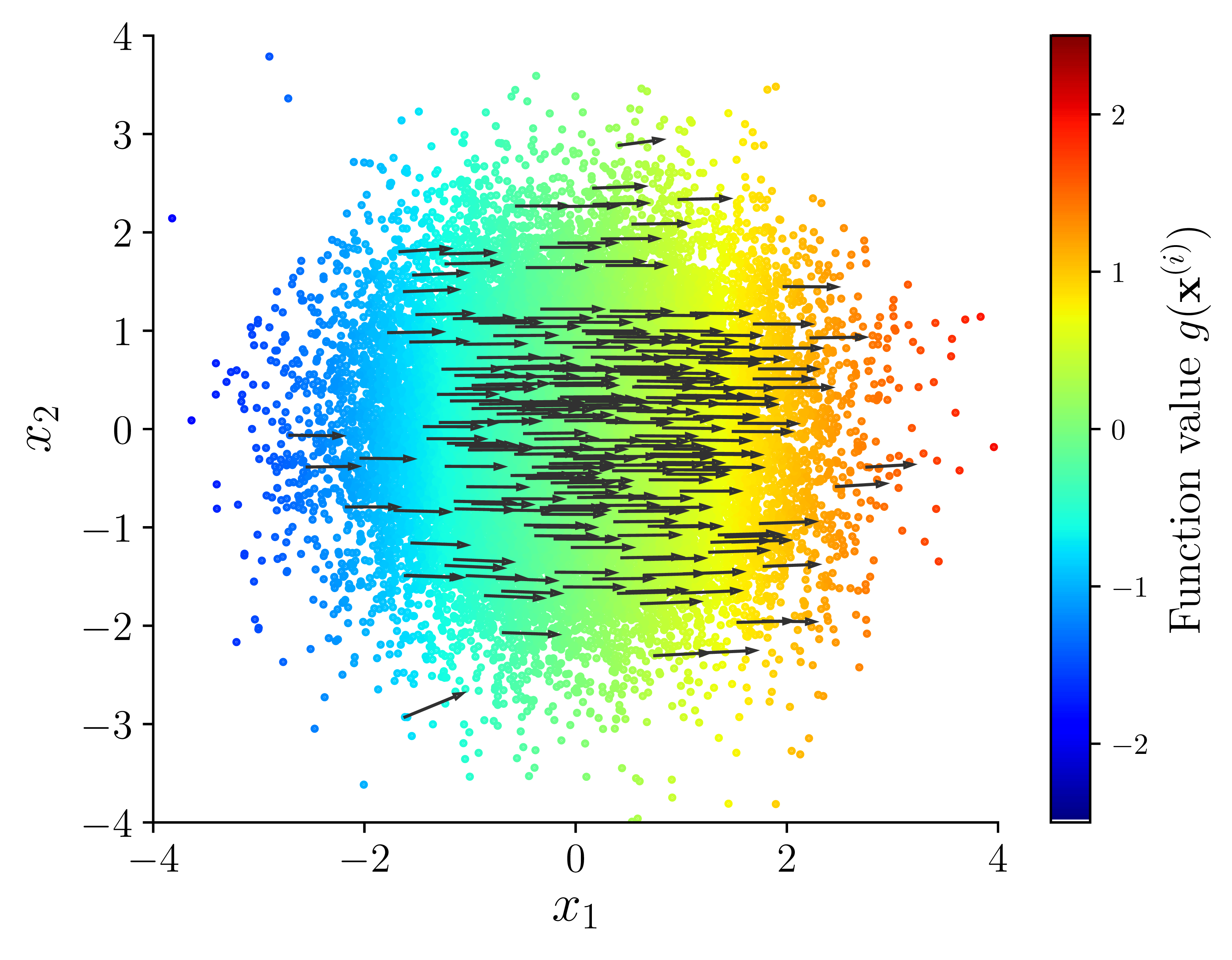}
    \caption{$\boldsymbol{r} = (1, 0)$, $c=0.5$}
  \end{subfigure}
  \begin{subfigure}{0.45\textwidth}
    \includegraphics[width=1.0\textwidth]{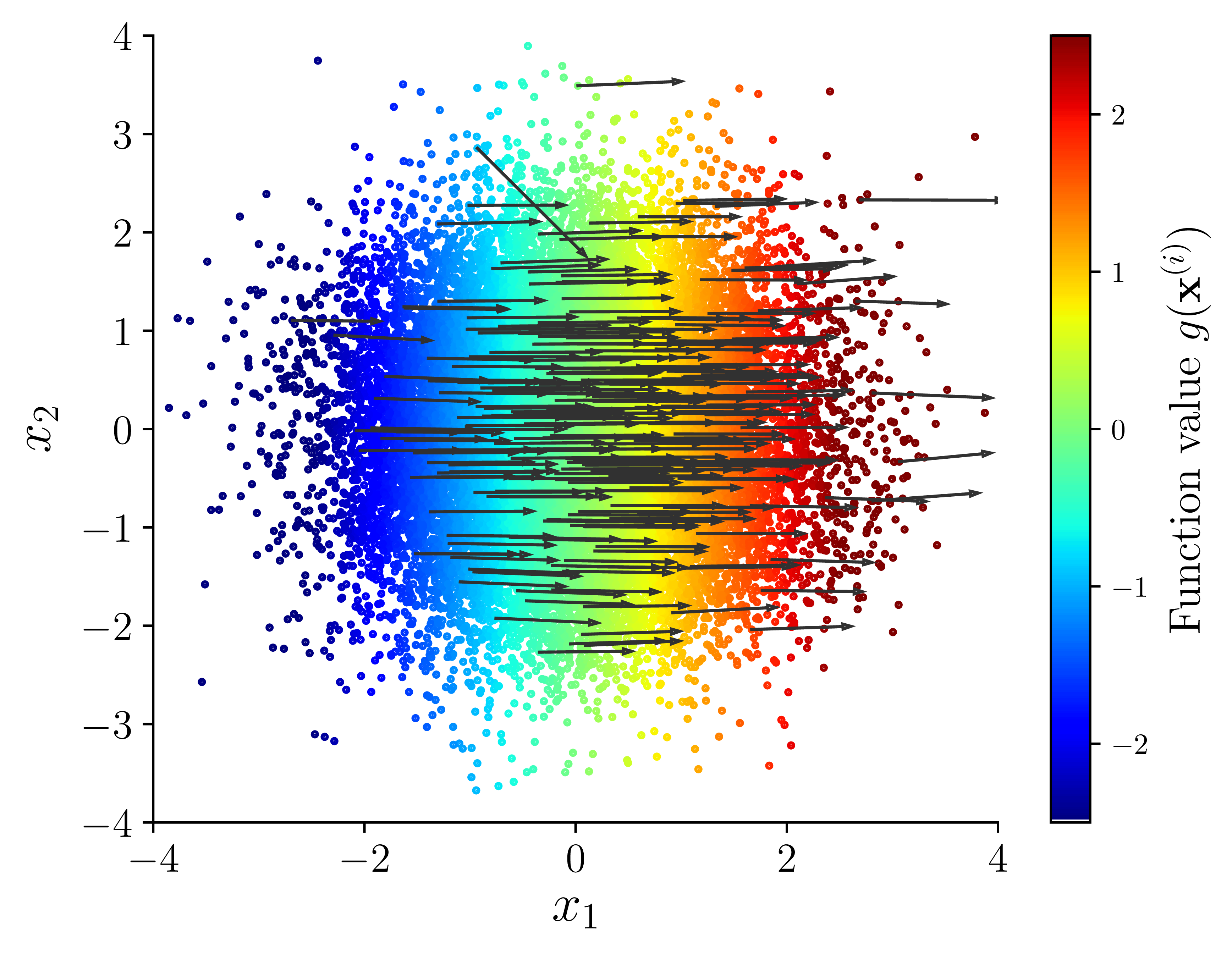}
    \caption{$\boldsymbol{r} = (1, 0)$, $c=1$}
  \end{subfigure}
  \caption{Estimated gradient vector field (arrows) for the function $g(\boldsymbol{x}) = c \boldsymbol{x} \cdot \boldsymbol{r}$ with $\boldsymbol{r} = (1, 0)$ (colors) using the eigendecomposition of the discrete Kolmogorov operator $\boldsymbol{L}_{\psi, c}$. The operator $\boldsymbol{L}_{\psi, c}$ is constructed using $n=10^{4}$ samples from a standard Gaussian distribution on $\mathbb{R}^{2}$, setting (a) $c=0.5$ and (b) $c=1$ (see also Fig.~\ref{fig:laplace-operator-eigenvalues}). We represent $g$ and its gradient using $\ell = 100$ eigenvalue/eigenvector pairs. The exact gradient is $\nabla g = c \boldsymbol{r}$.}
  \label{fig:gradient-example}
\end{figure}
\begin{figure}[h!]
  \centering
  \begin{subfigure}{0.45\textwidth}
    \includegraphics[width=1.0\textwidth]{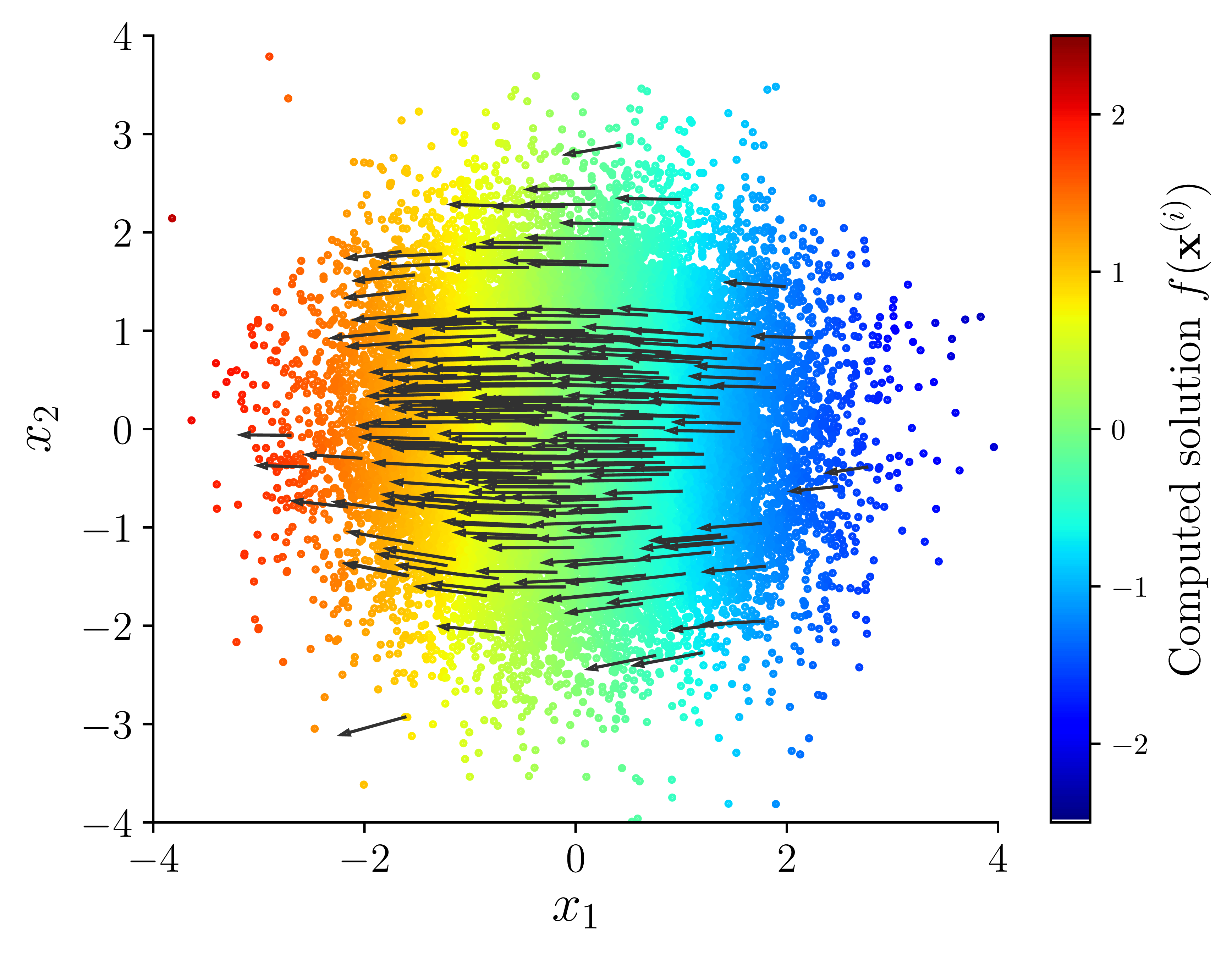}
    \caption{$f$ and $\nabla f$, $c=0.5$}
  \end{subfigure}
  \begin{subfigure}{0.45\textwidth}
    \includegraphics[width=1.0\textwidth]{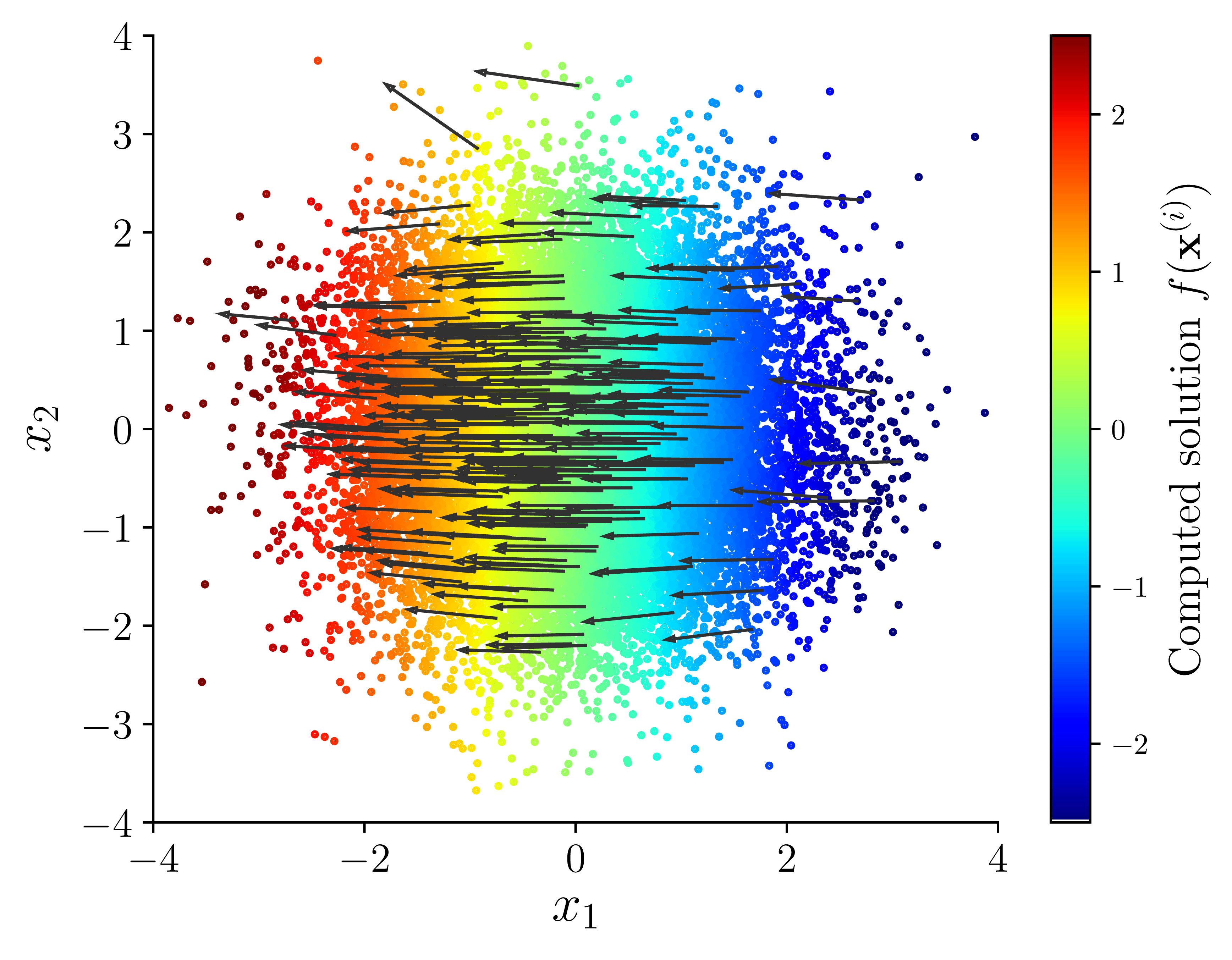}
    \caption{$f$ and $\nabla f$, $c=1$}
  \end{subfigure}
  \caption{The least-squares solution $\bm f$ (colors) and its gradient (arrows) for the Kolmogorov problem $\boldsymbol{L}_{\psi, c} \bm f = \bm g$ for (a) $c=0.5$ and (b) $c=1$ with $n=10^{4}$ samples from a standard Gaussian distribution on $\mathbb{R}^{2}$. The right-hand side function $g(\boldsymbol{x}) = c x_1$ is approximated in the Kolmogorov basis as $\boldsymbol{g}_\ell = \boldsymbol{Q} \boldsymbol{\widetilde{g}}$ (see Fig.~\ref{fig:function-expansion}~and~\ref{fig:gradient-example}). As in Fig.~\ref{fig:gradient-example}, we represent $\bm f$ and its gradient using $\ell = 3 \log n$ eigenvalue/eigenvector pairs. \edits{In both cases, the exact solution is $f(\bm{x}) = -x_1$.}}
  \label{fig:weighted-laplace-solution}
\end{figure}


\section{Implementation} \label{sec:implementation}

Given $n$ samples, we divide the computation into four major steps: 
\begin{enumerate}
    \item Compute the kernel matrix $\boldsymbol{K}_{\epsilon}$ in~\eqref{eq:density-kernel-matrix} and the estimated density function $\boldsymbol{\psi} = \boldsymbol{W}^{-1} \boldsymbol{K}_{\epsilon} \boldsymbol{1}$ in \eqref{eq:density-estimation}.
    \item Compute the variable-bandwidth kernel matrices $\boldsymbol{K}_{\epsilon, \beta}$ (unnormalized) and $\boldsymbol{K}_{\epsilon, \beta, \alpha}$ (normalized) from \eqref{eq:variable-bandwidth-kernel-matrix} and~\eqref{eq:normalized-variable-bandwidth-kernel-matrix}, respectively.
    \item Using $\boldsymbol{K}_{\epsilon, \beta, \alpha}$, form the sparse symmetric matrix $\boldsymbol{\hat{L}}_{\psi, c}$ in~\eqref{eq:symmetric-discrete-Kolmogorov-operator}, and compute the leading $\ell$ eigenvalue/eigenvector pairs $(\hat \lambda_j, \hat{\bm \phi_j})$. Apply the similarity transformation $\bm S $ to obtain the leading $\ell$ eigenpairs $(\hat \lambda_j, \bm \phi_j)$ of the discrete Kolmogorov operator $\boldsymbol{L}_{\psi, c}$. 
    \item Use the eigendecomposition to solve the Kolmogorov problem and/or represent gradient vector fields, as described in Sections~\ref{sec:solution} and \ref{sec:gradients}.
\end{enumerate}

Within this procedure, we spend a bulk of our computational resources (i) computing the eigendecomposition of $\hat{\boldsymbol{L}}_{\psi,c}$, and (ii) computing the density kernel matrix $\boldsymbol{K}_{\epsilon}$ and the variable-bandwidth kernel matrix $\boldsymbol{K}_{\epsilon, \beta}$. In principle, computing the eigendecomposition is the most expensive component of the algorithm. In practice, existing tools, such as \texttt{Spectra}: Sparse Eigenvalue Computation Toolkit as a Redesigned ARPACK \cite{spectra}, efficiently compute the sparse eigendecompositions. The next most expensive step is assembling the matrices $\boldsymbol{K}_{\epsilon}$ and $\boldsymbol{K}_{\epsilon, \beta}$. Algorithm~\ref{alg:brute-force-kernel-computation} defines the ``brute-force'' approach of looping over each pair of samples and computing the kernel matrix entry, which is $\mathcal{O}(m n^2)$ complexity. We must evaluate the kernel $\mathcal{O}(n^2)$ times and evaluating the kernel itself scales linearly with the dimension of the ambient space $m$. Therefore, we reduce the complexity of the kernel matrix construction by leveraging a $k$-$d$ tree nearest neighbor search. 

\begin{algorithm}
 \begin{algorithmic}
    \State Given: the samples $\{\boldsymbol{x}^{(i)}\}_{i=1}^{n}$ and a bandwidth vector $\boldsymbol{b} \in \mathbb{R}^{n}$
    \State Initialize the sparse kernel matrix $\boldsymbol{K} \in \mathbb{R}^{n \times n}$ (assume unset entries are $0$)
    \For{$i \gets 1$ to $n$} \Comment{Loop through the rows of the kernel matrix}
    \For{$j \gets i$ to $n$} \Comment{Loop through the columns of the kernel matrix}
    \State Compute the entry $k = \exp{\left(-\frac{\|\boldsymbol{x}^{(i)} - \boldsymbol{x}^{(j)}\|^2}{\epsilon^2 b_i b_j}\right)}$
    \If{$k>\delta_{tol}$}
    Set the entries $K^{(ij)} = k$ and $K^{(ji)} = k$
    \EndIf
    \EndFor
    \EndFor
 \end{algorithmic}
 \caption{The ``brute force'' algorithm to compute the sparse kernel matrices---either $\boldsymbol{K}_{\epsilon}$ or $\boldsymbol{K}_{\epsilon, \beta}$. The algorithm computes $\boldsymbol{K}_{\epsilon}$ by setting $\boldsymbol{\rho} = (\rho^{(1)}, \ldots, \rho^{(n)})$ to the values $\rho^{(i)} = b(\bm x^{(i)})$ of the bandwidth function defined in~\eqref{eq:bandwdith-parameter} and $\boldsymbol{K}_{\epsilon, \beta}$ by setting $ \rho^{(i)} = 2 \hat\psi^\beta(\bm x^{(i)})$, where $\hat \psi$ is the density estimate from~\eqref{eqDensEst}.  This algorithm requires two parameters: (i) the bandwidth parameter $\epsilon>0$ and (ii) a tolerance $\delta_{tol}$ to render the resulting matrices sparse. In the $\boldsymbol{K}_{\epsilon, \beta}$ case we also define the variable bandwidth parameter $\beta$. The complexity is $\mathcal{O}(m n^2)$.}
 \label{alg:brute-force-kernel-computation}
\end{algorithm}

\subsection{Binary search trees}

A $k$-$d$ tree stores each sample as leaf nodes of a binary tree. Each non-leaf node implicitly defines a hyper-plane such that the left sub-tree contains only points on the ``left'' side of the plane and the right sub-tree contains only points on the ``right'' side \cite{bentley1975}. Efficiently finding nearest neighbors using $k$-$d$ trees is a well-studied problem \cite{aryaetal1996,aryaetal1998,bentley1975,cormenetal2009,friedmanetal1977,sproull1991}. Constructing a $k$-$d$ tree is $\mathcal{O}(m n \log n)$ complexity and after construction finding the closest $s$ neighbors in the set of samples is a $\mathcal{O}(c(m) s \log n)$ operation, where $c(m)$ is a dimension-dependent function that grows at least as fast as $2^{m}$. Empirically, $c(m)$ grows rapidly with dimension $m$ and nearest-neighbor computation becomes difficult in high dimensions \cite{aryaetal1996,sproull1991}. 

Heuristically, the $k$-$d$ tree is more efficient than an exhaustive search when $2^{m} \ll n$ \cite{tothetal2017}. As $m$ grows, we resort to approximate nearest neighbor searches. Rather than finding exactly the $s$ nearest neighbors, we find $s$ points such that the distance between $\boldsymbol{x}^{(i)}$ and $\boldsymbol{x}^{(j)}$ is within $1+\delta$ of the distance between $\boldsymbol{x}^{(i)}$ and its true nearest neighbors. Finding the approximate nearest neighbors makes the constant $c(m) = m (1+6m/\delta)^{m}$ \cite{aryaetal1998}. For our examples, which are meant to demonstrate the overall method, we will use exact nearest neighbor searches using the software package \texttt{nanoflann}, which is remarkably efficient \cite{nanoflann}.

\subsection{Computing the kernel matrix with a $k$-$d$ tree}

We include a version of our algorithm to compute the kernel matrix for the sake of completeness and to explicitly highlight the computational cost of estimating eigenfunctions of the Kolmogorov operator. Although this cost is dominated by computing the eigendecomposition itself, nearest-neighbor searching, especially in high dimensions, is a non-trivial component of this algorithm. In general, diffusion maps and other graph-based algorithms become more expensive---both in computational run time and storage---as the dimension of the ambient space grows (i.e., $m \gg 1$). The samples $\boldsymbol{x} \in \Omega$ are on a $d$-dimensional manifold and leveraging this structure may further improve efficiency. However, such efforts are beyond the scope of this paper and we leave them to future work. In this paper, we focus on improving efficiency with respect to the number of samples $n$ and assume the ambient dimension $m$ is sufficiently small that it does not significantly affect performance.

For each sample, we first find the $s_i$ points within a radius $R$ ball centered at $\boldsymbol{x}^{(i)}$ that corresponds to the sample pairs that result in non-zero kernel matrix entries; this is a $\mathcal{O}(s_i \log n)$ operation. We choose $R$ such that the kernel function satisfies 
\begin{equation*} 
\exp{\left(-\frac{\|\boldsymbol{x}^{(i)} - \boldsymbol{x}^{(j)}\|^2}{\epsilon^2 \rho^{(i)} \rho^{(j)}}\right)} > \delta_{tol},
\end{equation*}
where $\rho^{(i)}$ is the bandwidth function that defines the kernel (e.g., the bandwidth in~\eqref{eq:bandwdith-parameter}) evaluated at the $i^{th}$ sample and $\delta_{tol}$ is a sparsity tolerance, implying that $R = \epsilon^2 b_{max}^2 \log\delta_{tol}$, where $b_{max} = \max_{i}\{ \rho^{(i)} \}$. We then loop over each of the $s_i$ nearest neighbors and compute the \edits{$(i, j)$ and $(j, i)$ kernel entries (since the kernel matrix is symmetric)}. Letting $s = \max_{i}{(s_i)}$, the total complexity is roughly $\mathcal{O}(m(s+1) n \log n)$, where the $s+1$ term accounts for the initial construction cost of the $k$-$d$ tree. Typically, $s \ll n$ and this complexity is approximately $\mathcal{O}(m n \log n)$. 

We leverage symmetry to further reduce the computation cost. \edits{After computing the first row/column we only need to compute the lower right $(n-1, n-1)$ corner of the kernel matrix. When we compute row $i$, we only need to find the nearest neighbors in the subset of samples $\{\boldsymbol{x}\}_{j=i}^{n}$. Therefore, we build a sequence of $k$-$d$ trees, each using a smaller subset of samples.} At row $i$, we \edits{use the $k$-$d$ tree with the fewest number of samples} to find the nearest neighbors $\boldsymbol{x}^{(j)}$ that are within the critical radius $R$ and such that $i \leq j$. We introduce a lag parameter $L$ that determines the number of $k$-$d$ trees to construct and build a series of $k$-$d$ trees $\{T_{r}\}_{r=0}^{t}$ such that $T_{r}$ only includes the points $\{ \boldsymbol{x}^{(i)} \}_{i=1+r L}^{n}$. \edits{We prescribe the lag parameter $L$ by prescribing the number of $k$-$d$ trees $t = t(n)$ and setting $L =  \lfloor n/t \rfloor$} For each row $i$, we use the $k$-$d$ tree such that $i>r L$---using a tree that contains fewer samples improves performance. The trade-off for this improved efficiency is a larger cost for the initial $k$-$d$ tree construction. An upper bound for the cost of constructing the $t$ $k$-$d$ trees is $\mathcal{O}(t n \log n)$. We find that choosing constant $t$ or $t \propto \log n$ improves the overall efficiency of the algorithm. However, choosing $t \propto n$ makes the overall cost roughly $\mathcal{O}(m n^2 \log n)$, which is slightly worse than the brute force algorithm. Algorithm~\ref{alg:kd-tree-kernel-computation} provides the pseudo-code for the $k$-$d$ tree based kernel matrix construction, and Fig.~\ref{fig:wall-clock-time} shows the performance of this algorithm compared to the brute force approach. To our knowledge, the sequential $k$-$d$ tree construction used in Algorithm~\ref{alg:kd-tree-kernel-computation} has not been previously discussed in the literature. Yet, as Fig.~\ref{fig:wall-clock-time} shows, this technique can significantly reduce the practical runtime of the kernel matrix construction. 

\begin{algorithm}
 \begin{algorithmic}
    \State Given: the samples $\{\boldsymbol{x}^{(i)}\}_{i=1}^{n}$ and a bandwidth vector $\boldsymbol{b} \in \mathbb{R}^{n}$
    \State Initialize the sparse kernel matrix $\boldsymbol{K} \in \mathbb{R}^{n \times n}$ (assume unset entries are $0$)
    \State Let $b_{max} = \max_{i \in [1,n]}{(b_i)}$ be the maximum bandwidth 
    \State Construct the $k$-$d$ trees $\{T_{r}\}$ = \Call{Construct $k$-$d$ trees}{$\{\boldsymbol{x}^{(i)}\}_{i=1}^{n}$, $L$}
    \For{$i \gets 1$ to $n$} \Comment{Loop through the rows of the kernel matrix}
    \State Find the largest $r$ such that $i > r L$ and use $T_{r}$ to find the nearest neighbors $$\mathcal{X} = \{\boldsymbol{x}^{(j)}: i \geq j \geq n \mbox{ and } \|\boldsymbol{x}^{(i)} - \boldsymbol{x}^{(j)} \|^{2} < -\epsilon^2 b_{max}^2 \log\delta_{tol} \}$$
    \For{$J \gets 1$ to $\vert \mathcal{X} \vert$} 
    \State Let $j = j(J)$ be the index of the $J^{th}$ nearest neighbor $\boldsymbol{x}^{(j)} \in \mathcal{X}$
    \State Compute the entry $k = \exp{\left(-\frac{\|\boldsymbol{x}^{(i)} - \boldsymbol{x}^{(j)}\|^2}{\epsilon^2 b_i b_j}\right)}$ 
    \If{$k>\delta_{tol}$}
    Set the entries $K^{(ij)} = k$ and $K^{(ji)} = k$
    \EndIf
    \EndFor
    \EndFor
    \Procedure{Construct $k$-$d$ trees}{$\{\boldsymbol{x}^{(i)}\}_{i=1}^{n}$, $L$}
    \State Find the largest $t$ such that $t L < n$ and initialize the $k$-$d$ trees $\mathcal{T} = \emptyset$
    \For{$r \gets 0$ to $t$}
    \State Compute the $k$-$d$ tree $T_{r}$ using the last $n-r L$ samples $\{\boldsymbol{x}^{(i)}\}_{i=1+r L}^{n}$
    \State Update $\mathcal{T} \gets \mathcal{T} \cup \{T_{r}\}$
    \EndFor
    \State \Return The set of $k$-$d$ trees $\mathcal{T}$ 
    \EndProcedure
 \end{algorithmic}
 \caption{The $k$-$d$ tree based algorithm that, similar to Algorithm~ \ref{alg:brute-force-kernel-computation}, computes either $\boldsymbol{K}_{\epsilon}$ or $\boldsymbol{K}_{\epsilon, \beta}$. This algorithm requires three parameters: (i) the bandwidth parameter $\epsilon>0$, (ii) the sparsity tolerance $\delta_{tol}$, (iii) the lag parameter  $L$ (in the $\boldsymbol{K}_{\epsilon, \beta}$ case we additionally need to define the variable bandwidth parameter $\beta$). The complexity is typically $\mathcal{O}(n \log n)$.}
 \label{alg:kd-tree-kernel-computation}
\end{algorithm}

\begin{figure}[h!]
  \centering
  \begin{subfigure}{0.75\textwidth}
    \includegraphics[width=1.0\textwidth]{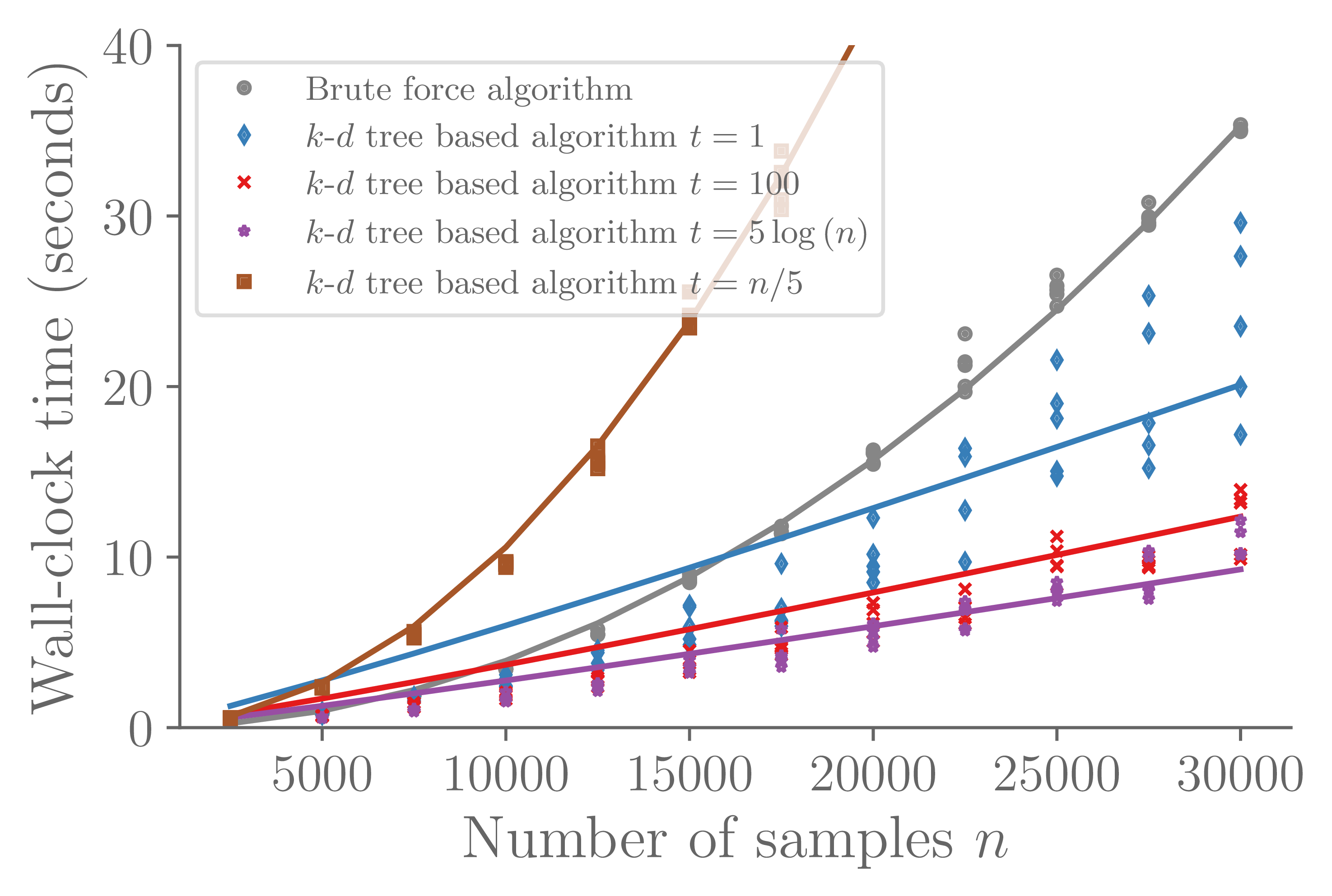}
  \end{subfigure}
  \caption{The wall-clock time as a function of the number of samples from a Gaussian distribution to construct the kernel matrix using the brute force approach in Algorithm~\ref{alg:brute-force-kernel-computation} (grey points) and the $k$-$d$ tree based algorithm in Algorithm~\ref{alg:kd-tree-kernel-computation} (colored points). We ran this on an Intel Core i7-4770 CPU at 3.40GHz. The blue and red markers represent a series of $t=1$ and $t=100$ $k$-$d$ trees---equivalently, the lag parameter is $L = n/t$. The purple and brown markers use $t = 5\log n$ $t=n/5$ and $k$-$d$ trees with lag parameters $L = n/(5\log n)$ and $L = 5$, respectively. For reference, the grey and brown lines shown are proportional to $n^2$, and the blue, red, and purple lines are proportional to $n \log n$.}
  \label{fig:wall-clock-time}
\end{figure}

The outer loop over the samples in both Algorithms~\ref{alg:brute-force-kernel-computation}~and~\ref{alg:kd-tree-kernel-computation} is trivially parallelizable. Given $n_{proc}$ processors, the brute-force algorithm has approximately $\mathcal{O}(m n^2 / n_{proc})$ complexity. Initially constructing the $k$-$d$ tree still costs $\mathcal{O}(mn \log n)$, and this task is not easily parallelizable. After construction, the $k$-$d$ tree based algorithm has approximately $\mathcal{O}(m n/n_{proc} \log n)$ complexity. Therefore, for sufficiently large $n_{proc}$ the ``brute force'' algorithm may in principle outperform the $k$-$d$ tree based algorithm. In most applications, however, we choose $n \gg n_{proc}$ and find the $k$-$d$ approach to be more efficient. Additionally, efficiently computing the bandwidth parameter in~\eqref{eq:bandwdith-parameter} typically requires computing the $k$-$d$ trees anyway.

The parameter tuning described in Section~\ref{sec:optimal-bandwidth} is also computationally expensive. The optimization problem---maximizing~\eqref{eq:bandwidth-sensitivity}---is relatively straightforward. However, computing the cost function requires repeatedly constructing the kernel matrix. We use Algorithm~\ref{alg:kd-tree-kernel-computation} to efficiently compute the cost function and leverage gradient-free algorithms in \texttt{NLopt} \cite{nlopt} to numerically solve the one dimensional optimization problem.

\section{Examples} \label{sec:examples}

We illustrate our approach using three examples: (i) a toy problem using Gaussian samples (Section~\ref{sec:Gaussian-toy-example}), (ii) solving the Kolmogorov problem on a spherical manifold (Section~\ref{sec:spherical-manifold-example}), and (iii) an application that uses our method to evolve samples from a probability distribution according to an advection equation (Section \ref{sec:advection-example}).

\subsection{Gaussian examples} \label{sec:Gaussian-toy-example}

We use a two-dimensional Gaussian toy problem to assess the convergence of the eigenfunction estimates, \edits{and then solve the Kolmogorov problem using samples from a four-dimensional Gaussian distribution to demonstrate our method in higher dimensions}. We sample $n$ points $\{\boldsymbol{x}^{(i)}\}_{i=1}^{n}$ from a Gaussian distribution $\mathcal{N}(\boldsymbol{0}, \edits{\boldsymbol{\Sigma}})$ on $\Omega = \mathbb R^{m}$ \edits{with $m=2$ or $m=4$.} \edits{In $m=2$ dimensions, we use $\boldsymbol{\Sigma} = \boldsymbol{I}$. In $m=4$ dimensions, we set $\bm\Sigma=\diag(\sqrt{2}, \sqrt{2}, \sqrt{3}, \sqrt{3})$. Note that in these examples we have $\Omega = \mathbb R^m$ since the Gaussian distribution is supported on the entire Euclidean space, and thus the intrinsic and extrinsic dimensions are equal, $m=d$.}  

\edits{Using the samples $\boldsymbol{x}^{(i)}$,} we estimate the density function (Figure~\ref{fig:density-estimation}) and the eigendecomposition of the Kolmogorov operator (\edits{Figs.~\ref{fig:laplace-operator-eigenvalues} and~\ref{fig:standard-Gaussian-error-analysis}}) as described in Section~\ref{sec:Kolmogorov-problem}. We then use these samples to estimate solutions to~\eqref{eq:Kolmogorov-problem} (Figure~\ref{fig:weighted-laplace-solution} \edits{and~\ref{fig:4d-example}}) using the method described in Section~\ref{sec:solution-gradients}. \edits{In the two-dimensional case where we can readily perform visualizations, we also show the estimated gradient vector field (Figure~\ref{fig:weighted-laplace-solution}).} 

\begin{figure}[h!]
  \centering
  \begin{subfigure}{0.45\textwidth}
    \includegraphics[width=1.0\textwidth]{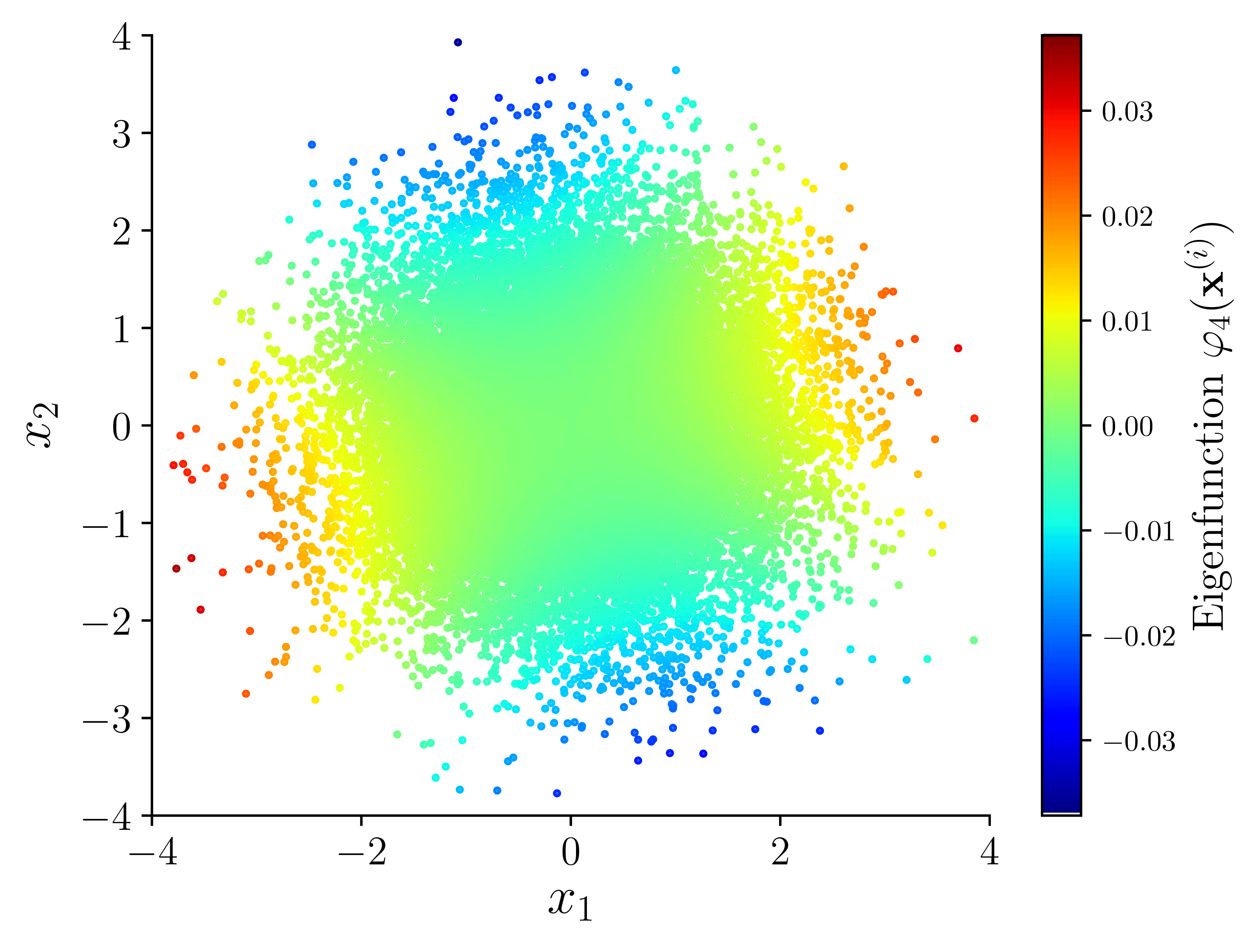}
    \caption{Exact eigenfunction $\varphi_4$ of $\mathcal L_{\psi, c}$}
    \label{fig:fourth-eigenvector-analytic}
  \end{subfigure}
  \begin{subfigure}{0.45\textwidth}
    \includegraphics[width=1.0\textwidth]{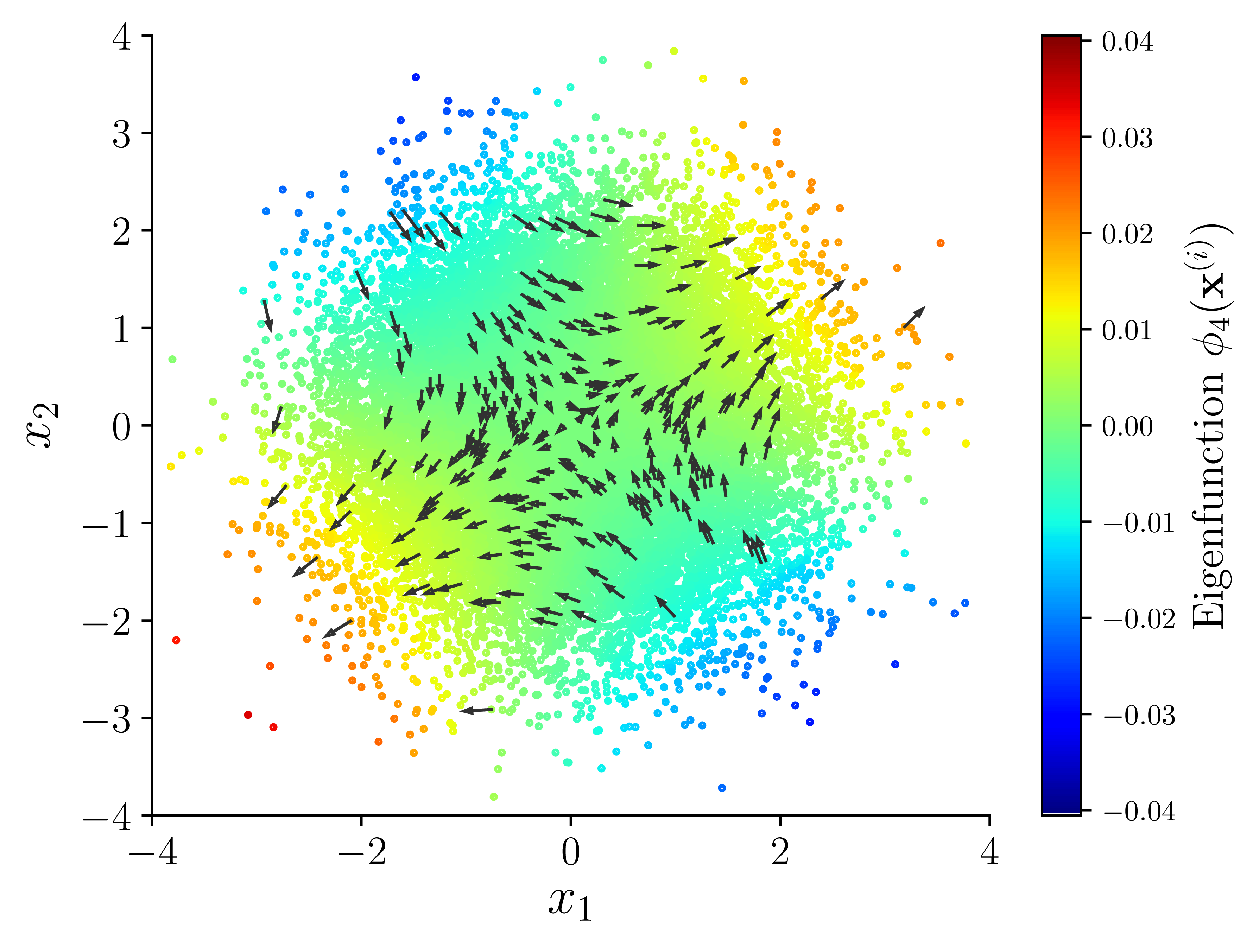}
    \caption{Computed eigenfunction $\phi_4$ of $\mathcal L_{\psi, c}$}
    \label{fig:fourth-eigenvector-computed}
  \end{subfigure}
  \begin{subfigure}{0.45\textwidth}
    \includegraphics[width=1.0\textwidth]{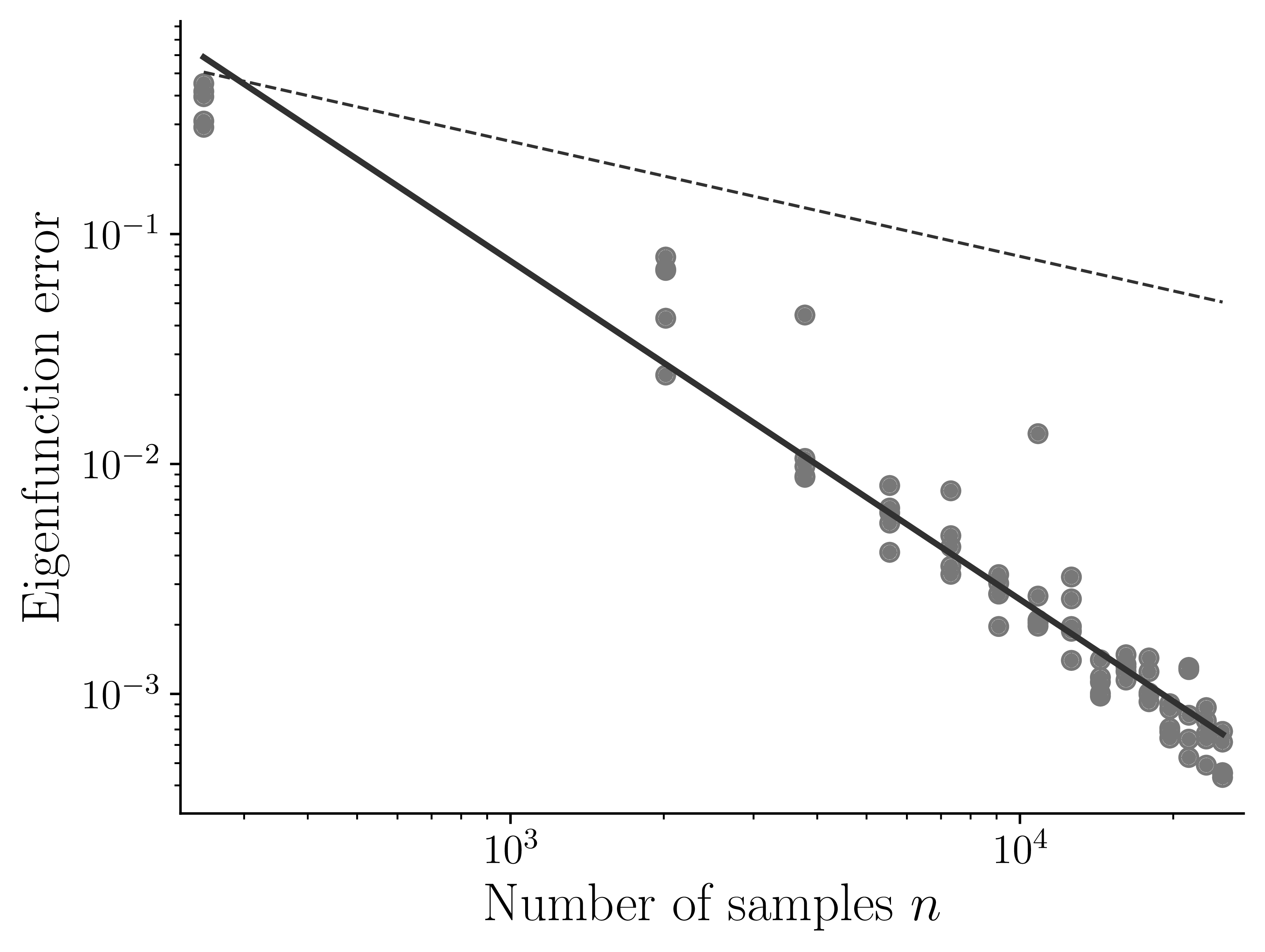}
    \caption{Expected eigenfunction error}
    \label{fig:fourth-eigenvector-expected-error}
  \end{subfigure}
  \begin{subfigure}{0.45\textwidth}
    \includegraphics[width=1.0\textwidth]{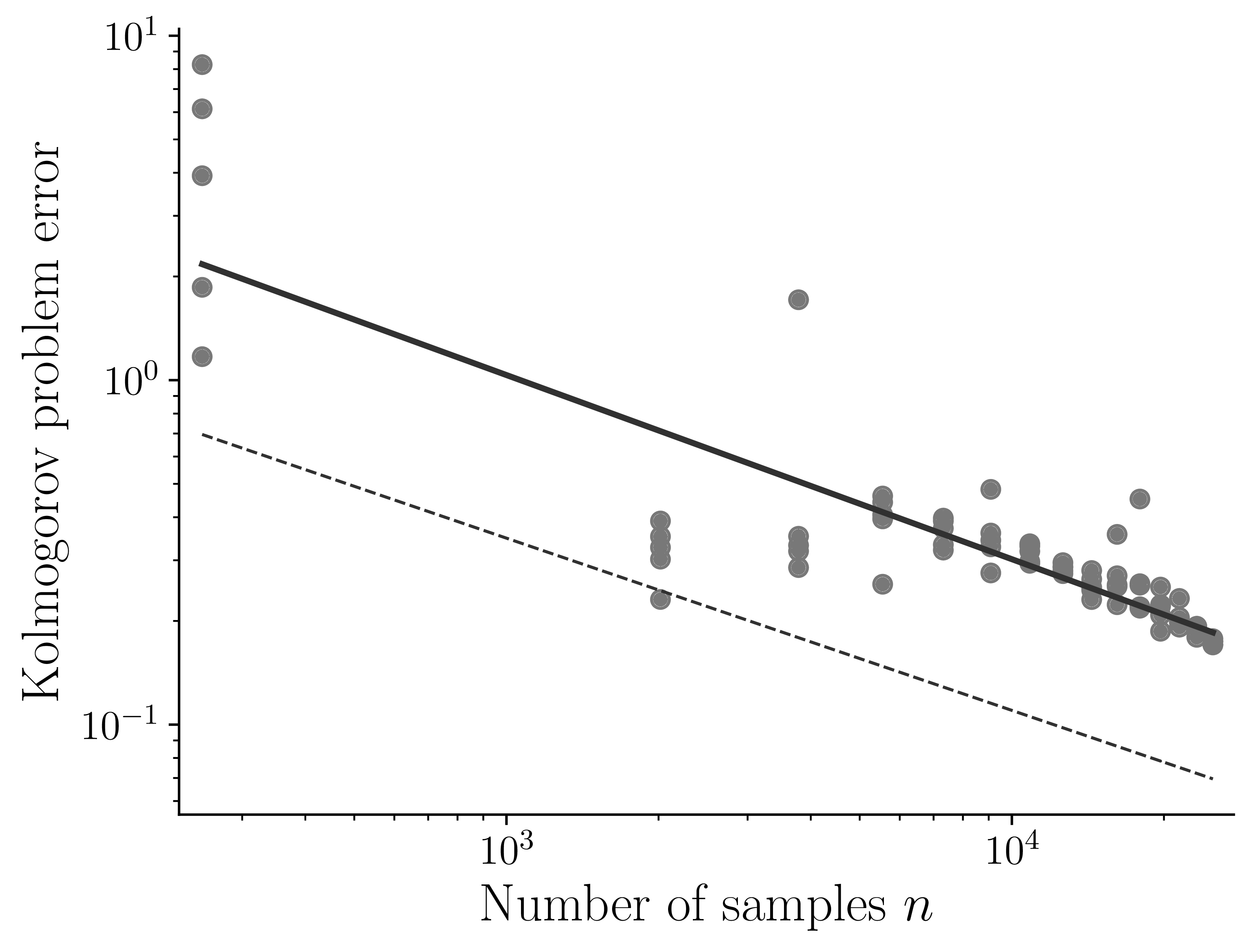}
    \caption{Expected solution error}
    \label{fig:solution-expected-error}
  \end{subfigure}
  \caption{\edits{Analytical and numerical eigenfunctions of the Kolmogorov operator for a Gaussian distribution $\psi=\mathcal{N}(\boldsymbol{0}, \boldsymbol{I})$ in $m=2$ dimensions.} \edits{Panel~(a) shows an analytical eigenfunction $\varphi_4(\boldsymbol{x})$ of $\mathcal L_{1,\psi}$ corresponding to the twofold-degenerate eigenvalue $\lambda_j = -2$, evaluated at each of the $n=2.5 \times 10^{4}$ samples $\bm x^{(i)}$. The analytical eigenfunction is expressed as a linear combination $ \varphi_4 = R_{11} \phi_4 + R_{12} \phi_5$, where $R_{ij}$ are elements of the rotation matrix $\bm R(\theta^*)$ that minimizes the empirical error in~\eqref{eq:expected-eigenfunction-error}. Panel~(b) shows the numerical eigenvector $\bm \phi_4 $ of $\bm L_{\psi, c}$ (colors) and the corresponding gradient (arrows). The gradient is computed using the procedure in Section~\ref{sec:solution}. Panel~(c) shows the empirical error between $\bm\varphi_4$ (the vector containing the values $\varphi_4(\bm x^{(i)})$) and $\bm \phi_4$, computed via~\eqref{eq:expected-eigenfunction-error}, as a function on the number of samples $n$. In Panel~(d), we prescribe the non-rotated analytical eigenfunction $ \phi_4$ as the right hand side $g$ of the Poisson problem $\mathcal L_{\psi,1} f = g$, and compute the empirical error of the numerical solution $ \bm f$ relative to the analytical solution $ f = g / \lambda_4 $ as a function of $n$, using~\eqref{eqNormalizedError}. In Panels~(c) and~(d) we compute each error value using independent datasets of $n$ samples $\{\boldsymbol{x}^{(i)}\}_{i=1}^{n}$ drawn from $\mathcal{N}(\boldsymbol{0}, \boldsymbol{I})$. Solid lines show the best-fit empirical error---the empirical error averaged over $5$ experiments with $n$ samples. Dashed lines show an $n^{-1/2}$ convergence rate for reference.}}  
  \label{fig:standard-Gaussian-error-analysis}
\end{figure}

\subsubsection{\label{sec2DGauss}Dimension $m=2$}

\edits{Starting with the two-dimensional example,} we assess the validity of our algorithms by comparing our numerically computed eigenfunctions to an analytical solution. Figures~\ref{fig:computed-fourth-eigenvector}~and~\ref{fig:fourth-eigenvector-computed} show the computed $4^{th}$ eigenvector $\bm \phi_4$ \edits{using $n=2.5 \times 10^4$ samples}. Recall that the $i$-th component of $ \bm \phi_4$, i.e., $\phi_4^{(i)}$, is an estimate of the eigenfunction value $\phi_4(\boldsymbol{x}^{(i)})$. In this case, standard results show that the $4^{th}$ and $5^{th}$ eigenfunctions are $\phi_4(\boldsymbol{x}) = x_1 x_2 / \sqrt{2 \pi}$ and $\phi_5(\boldsymbol{x}) = (x_1^2-x_2^2) / \sqrt{2 \pi}$ and that they correspond to an eigenvalue with multiplicity $2$ \cite{couranthilbert2008}. Let $\boldsymbol{\varphi^{\prime}}_4$ and $\boldsymbol{\varphi^{\prime}}_5$ be the vector whose $i^{th}$ component are the true eigenfunctions evaluated at the $i^{th}$ sample. We compare the eigenspace spanned by these vectors to the space spanned by the $4^{th}$ and $5^{th}$ eigenvectors computed by our algorithm ($\boldsymbol{\phi}_4$ and $\boldsymbol{\phi}_5$, respectively). We normalize the eigenvectors $\boldsymbol{\varphi^{\prime}}_4$, $\boldsymbol{\varphi^{\prime}}_5$, $\boldsymbol{\phi}_4$, and $\boldsymbol{\phi}_5$ \edits{to be unit vectors in the norm $\|\boldsymbol{\phi}\|^2_{\boldsymbol{S}} = \langle \boldsymbol{\phi}, \boldsymbol{\phi} \rangle_{\boldsymbol{S}}$ (see~\eqref{eqInnerProd})}. Let $\boldsymbol{R}(\theta)$ be a $2 \times 2$ rotation matrix that rotates vectors by $\theta \in [0, 2\pi]$. Let 
\begin{equation*}
    \left[ \begin{array}{ccc}
    \horzbar & \boldsymbol{\varphi}_4^\top & \horzbar \\
    \horzbar & \boldsymbol{\varphi}_5^\top & \horzbar 
    \end{array} \right] = \boldsymbol{R}(\theta^{*}) \left[ \begin{array}{ccc}
    \horzbar & \boldsymbol{\varphi^{\prime}}_4^\top & \horzbar \\
    \horzbar & \boldsymbol{\varphi^{\prime}}_5^\top & \horzbar 
    \end{array} \right],
\end{equation*}
where $\theta^{*}$ is the angle that minimizes the empirical squared error 
\begin{equation}
    e^2(\{\boldsymbol{x}\}_{i=1}^{n}) = \frac{1}{n} (\boldsymbol{\phi}_4 - \boldsymbol{\varphi}_4) \cdot (\boldsymbol{\phi}_4 - \boldsymbol{\varphi}_4). 
    \label{eq:expected-eigenfunction-error}
\end{equation}
Figure~\ref{fig:fourth-eigenvector-analytic} shows the analytical and numerical $4^{th}$ eigenvector. Figure~\ref{fig:fourth-eigenvector-expected-error} shows the expected squared error as a function of $n$. Even though, to our knowledge, there are no spectral convergence results in the literature for the class of variable-bandwidth kernels in \cite{berryharlim2016}, we observe a convergence rate that significantly exceeds the typical $1/\sqrt{n}$ Monte Carlo convergence rate.

\edits{Next, observe that given a diagonal covariance matrix $\bm\Sigma = [\Sigma_{ij}]$ we have
\begin{equation}
    \label{eqLXi}
    \mathcal{L}_{1\psi,1} x_i = x_i \nabla_i \log{\psi} = - \Sigma_{ii}^{-1} x_i
\end{equation}
and, therefore, $\phi(\bm{x}) = x_i$ is an eigenfunction with eigenvalue $\lambda= \Sigma_{ii}^{-1}$. Moreover, the multiplicity of $\lambda$ is equal to the number of times that the value $\Sigma_{ii}$ appears in the diagonal entries of $\bm\Sigma$. Prescribing $ g = \phi$ as the right hand side of the Poisson problem $\mathcal L_{\psi,1} f = g$, it follows that the solution $ f= g / \lambda$ is expressible as a linear combination of eigenfunctions $ \phi_j$ corresponding to eigenvalues $\lambda_j = \lambda$. In the discrete case, we expect that the solution vector $ \bm f$ of the Poisson problem associated with $\bm L_{\psi,c} $ for the source vector $\bm g = [g(\bm x^{(i)})]$ (see~\eqref{eqFLS}) to well-approximate the analytical solution $ f$ for modest values of the spectral truncation parameter $\ell$.}

\edits{Let $ \bm{\check f} = [f(\bm x^{(i)}) ]$ be the vector whose entries are given by the true solution values. In Figure~\ref{fig:standard-Gaussian-error-analysis}(d), we compute the normalized empirical squared error of $\bm f $ relative to $ \bm{\check f} $, 
\begin{equation}
    \label{eqNormalizedError}
    \tilde e^2(\{\boldsymbol{x}\}_{i=1}^{n}) = \frac{(\boldsymbol f - \boldsymbol{\check f}) \cdot (\boldsymbol f - \boldsymbol{\check f})}{ \bm{\check f}\cdot \bm{\check f}}, 
\end{equation}
as a function of the number of samples $n$. We observe an approximate $O(n^{-1/2})$ convergence rate. 
}

\subsubsection{Dimension $m=4$}

\edits{We now demonstrate our method in $m=4$ dimensions using $n = 10^4$ samples $\{\boldsymbol{x}^{(i)}\}_{i=1}^{n}$ from a Gaussian distribution $\psi(\boldsymbol{x}) = \mathcal{N}(\boldsymbol{x}; \boldsymbol{0}, \boldsymbol{\Sigma})$. We recall that $\bm \Sigma = \diag(\sqrt{2}, \sqrt{2}, \sqrt{3}, \sqrt{3})$. Thus, according to~\eqref{eqLXi}, the leading two nonzero eigenvalues of $\mathcal L_{\psi,1}$ are $-1/\sqrt 3$ and $-1/\sqrt 2$, and both of these eigenvalues have multiplicity 2.}

\edits{First, we use the procedure in Section~\ref{sec:density-estimation} to estimate the $4$-dimensional Gaussian density (see the diagonal panels of Figure~\ref{fig:4d-solution}). Note that we estimate the $4$-dimensional joint density $\psi(\bm x)$ and these plots project this field onto each coordinate; we are visualizing the joint density, not the marginal. We then prescribe the source function $g(\boldsymbol{x}) = x_1 + x_3$ and solve the Kolmogorov problem $\mathcal{L}_{\psi, 1} f = g$ for the solution $f$. In this case, $f(\bm{x}) = -\sqrt{2} x_1 - \sqrt{3} x_3$ is the exact solution since $x_1$ and $x_3$ are eigenfunctions with eigenvalues $\lambda_1 = - 1/\sqrt 3$ and $\lambda_3 = - 1/\sqrt 2$, respectively.}  

\edits{In the upper-triangular panels of Figure~\ref{fig:4d-solution} we plot various two-dimensional projections of the solution $f( \bm x^{(i)} )$ on the samples. Specifically, the upper-triangular panel in the $j^{th}$ row and $k^{th}$ column (measured from the top left) shows a scatterplot of the $j^{th}$ and $k^{th}$ components of $\bm x^{(i)}$ (i.e., the points $(x^{(i)}_j, x^{(i)}_k) \in \mathbb R^2 $) colored by the value of the solution (i.e., the $i$-th component of the exact solution vector $ \bm{\check f} = [f(\bm x^{(i)})]$). Note that since $ f$ is a function of $x_1$ and $x_3$ only, the panel corresponding to $j=1$ and $k=3$ shows a scatterplot of a well-defined function, $(x_1^{(i)}, x_3^{(i)}) \mapsto f(\bm x) $; thus, we see a smooth progression of colors. The projections in the other upper-triangular panels do not correspond to functions of the coordinates, so the colors appear noisy, but nonetheless convey useful information about the structure of the solution in space.} 

\edits{In our numerical experiments, we use our $k$-$d$ tree implementation from Section~\ref{sec:implementation} to efficiently compute the discrete approximation of the Kolmogorov operator $\boldsymbol{L}_{\psi, 1}$. We compute the empirical solution $\bm f$ via~\eqref{eqFLS} using $\ell = 100$ eigenfunctions. Since we do not compute gradients of the solution in this case, we only need to 
solve the sparse least-squares system $\boldsymbol{L}_{\psi, 1} \boldsymbol{f} = \boldsymbol{g}$ to estimate solutions to the differential equation in $4$ dimensions. However, in this example we solve the least-squares problem using the eigendecomposition of $\boldsymbol{L}_{\psi, 1}$. In the lower-triangular panels of Figure~\ref{fig:4d-solution} we plot projections of the empirical solution $\bm f$ analogously to the exact solution $\bm{\check f}$ in the upper-triangular panels. Despite the modest number of samples for the dimensionality of the problem, we qualitatively observe that our computed solution matches the expected result. Still, some differences are present, notably a bias of the empirical solution towards larger absolute values (as evidenced by the stronger colors in the lower-triangular panels). As we will see momentarily, this bias can be attributed to errors in the eigenvalues of the discrete Kolmogorov operator.   
}

\begin{figure}
  \centering
  \begin{subfigure}{0.8\textwidth}
    \includegraphics[width=1.0\textwidth]{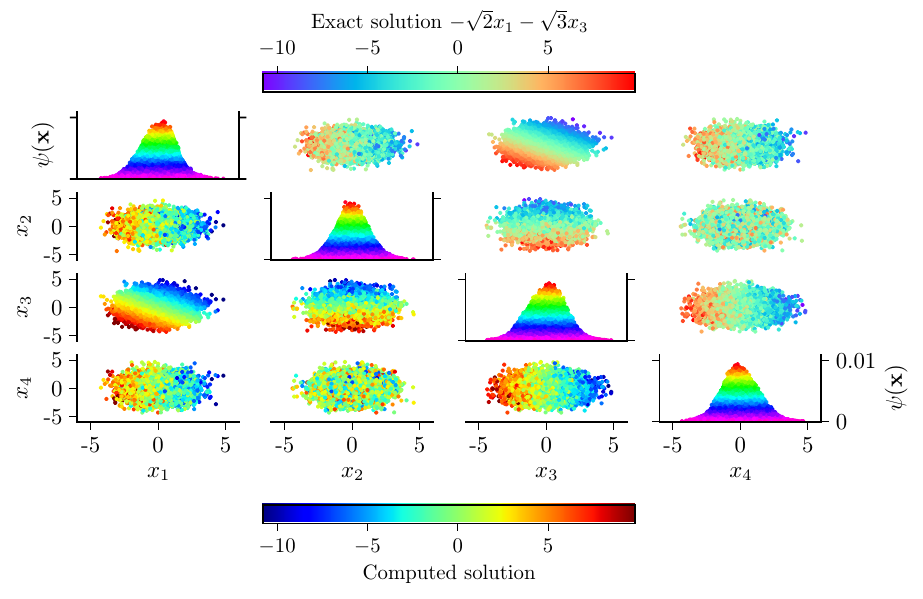}
    \caption{Density estimate and solution}
    \label{fig:4d-solution}
  \end{subfigure}
  \begin{subfigure}{0.475\textwidth}
    \includegraphics[width=1.0\textwidth]{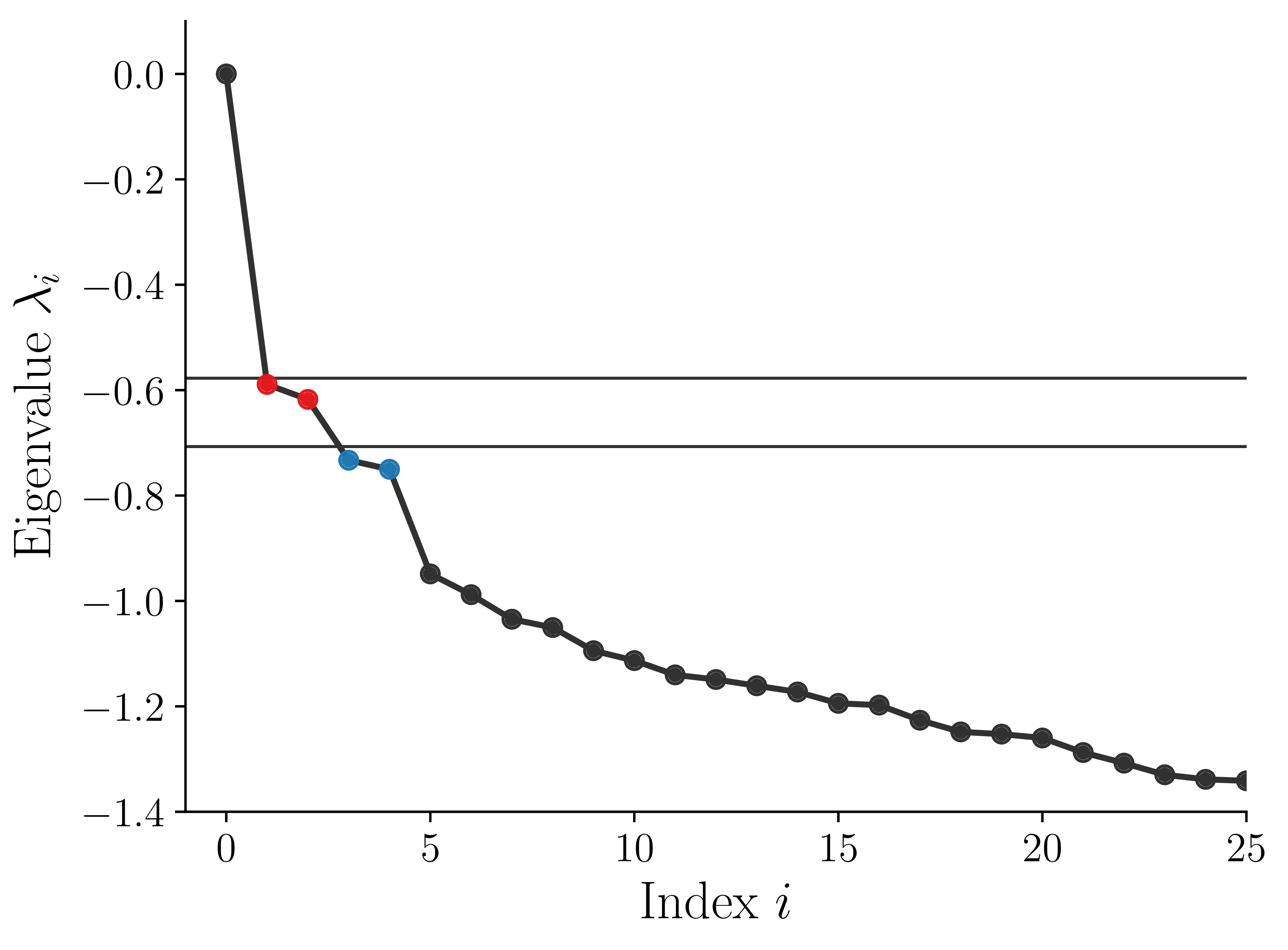}
    \caption{Eigenvalues of $\bm{L}_{\psi, 1} = \boldsymbol{Q}_{\ell} \bm{\Lambda} \boldsymbol{Q}_{\ell}^{-1}$}
    \label{fig:4d-eigenvalues}
  \end{subfigure}
  \begin{subfigure}{0.475\textwidth}
    \includegraphics[width=1.0\textwidth]{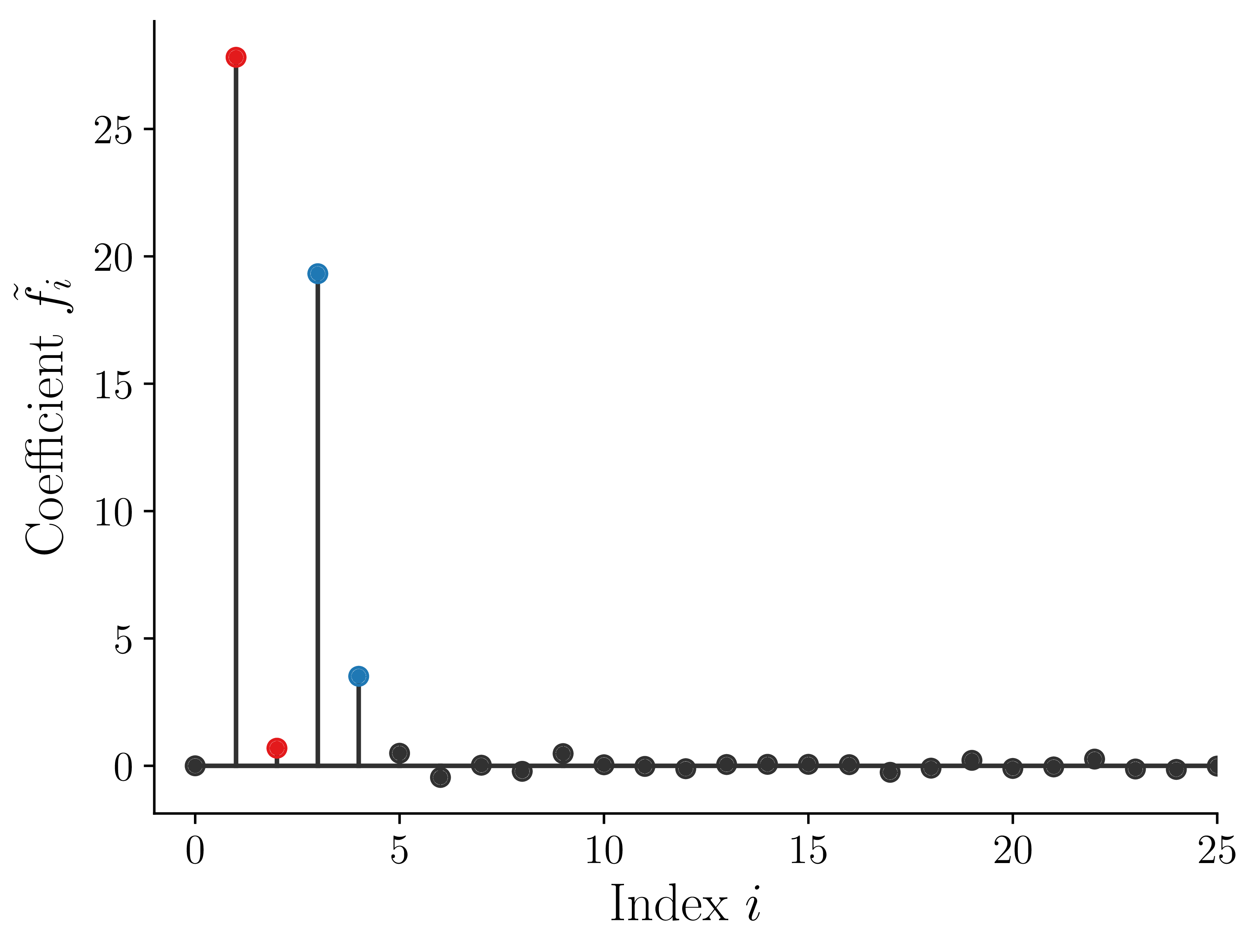}
    \caption{Solution coefficients $\bm{\tilde{f}} = \boldsymbol{Q}_{\ell}^{-1} \bm{f}$}
    \label{fig:4d-coefficients}
  \end{subfigure}
  \caption{\edits{Solution of the Kolmogorov equation in $\mathbb R^4$ using $n=10^4$ samples from a Gaussian distribution $\psi$ and linear right hand side $g$, using $n=10^4$. In Panel~(a), the diagonal subfigures show $1$-dimensional projections of our estimate of the Gaussian density $\psi(\boldsymbol{x}) = \mathcal{N}(\boldsymbol{x}; \boldsymbol{0}, \boldsymbol{\Sigma})$. The $j^{th}$ diagonal entry shows a scatterplot of $(\boldsymbol{x}_j^{(i)}, \psi(\boldsymbol{x}^{(i)}))$; the color scheme of the diagonal plots corresponds to the magnitude of the $y$-axis. Note that $\psi(\boldsymbol{x}^{(i)})$ is an estimate of the joint density at the $i^{th}$ sample, not the marginal. The off-diagonal subfigures show two-dimensional projections of the solution. The $(j,k)^{th}$ panel in the upper-triangular portion (measured from the top left) shows a scatterplot of the coordinates $(x_k^{(i)},x_j^{(i)})$ of the samples, colored by the value of the solution,  $f(\bm x^{(i)}) =-\sqrt{2} x_1^{(i)} - \sqrt{3} x_3^{(i)}$. The subfigures in the lower-triangular portion show corresponding scatterplots for the numerical solution $\bm f$ obtained via~\eqref{eqFLS}. Panel~(b) shows the first $25$ eigenvalues of the discrete Kolmogorov operator $\bm{L}_{\psi, 1}$. We expect $-1/\sqrt{2}$ and $-1/\sqrt{3}$ to be eigenvalues with multiplicity $2$; the two horizontal lines in Panel~(b) indicate these values. The blue and red dots indicate the two closest eigenvalues to $-1/\sqrt{2}$ and $-1/\sqrt{3}$, respectively. Panel~(c) shows the coefficients $\bm{\tilde{f}} = \boldsymbol{Q}_{\ell}^{-1} \bm{f}$ representing the solution as a linear combination of eigenvectors of the discrete Kolmogorov operator. We only expect the coefficients corresponding to eigenvalues $-1/\sqrt{2}$ and $-1/\sqrt{3}$ to be nonzero. The blue and red dots indicate the coefficients corresponding to the two closest eigenvalues to $-1/\sqrt{2}$ and $-1/\sqrt{3}$, respectively.}}
  \label{fig:4d-example}
\end{figure}

\edits{
We further assess the quality of our numerical solution by investigating the computed eigenspectrum of the discrete Kolmogorov operator $\bm{L}_{\psi, 1}$. Figure~\ref{fig:4d-eigenvalues} shows the eigenvalues $\hat \lambda_i $ of $\bm{L}_{\psi, 1} = \boldsymbol{Q}_{\ell} \bm{\Lambda} \boldsymbol{Q}_{\ell}^{-1}$ and Figure~\ref{fig:4d-coefficients} shows the solution coefficients $\bm{\tilde{f}} = \boldsymbol{Q}_{\ell}^{-1} \bm{f}$. Since the covariance matrix is diagonal with entries $\diag{(\boldsymbol{\Sigma})} = (\sqrt{2}, \sqrt{2}, \sqrt{3}, \sqrt{3})$, we expect $-1/\sqrt{2}$ and $-1/\sqrt{3}$ to be eigenvalues with multiplicity $2$. Figure~\ref{fig:4d-eigenvalues} clearly shows the two pairs of eigenvalues, $(\hat \lambda_1, \hat \lambda_2)$ and $(\hat \lambda_3, \hat \lambda_4)$ near each of these expected eigenvalues, respectively. This is also reflected in the coefficients representing the linear expansion of the solution in the eigenbasis, which are displayed in Figure~\ref{fig:4d-coefficients}. The nonzero coefficients clearly correspond to the $2$-dimensional eigenspaces with eigenvalues $-1/\sqrt{2}$ and $-1/\sqrt{3}$. However, both pairs of numerical eigenvalues have somewhat larger values than the true values (compare the colored dots with the horizontal lines in Figure~\ref{fig:4d-eigenvalues}). Upon solution of the least-squares problem to obtain $\bm{\tilde f}$ (which employs the reciprocals of the eigenvalues $\hat \lambda_i$; see~\eqref{eqLSMinimizer}), these errors lead to the amplitude bias of the solution mentioned above.} 

\edits{As one might expect from the higher dimensionality of this problem, we observed higher sensitivity of our results to the bandwidth parameter $\epsilon$ and the number of data points $n$ compared to the two-dimensional case in Section~\ref{sec2DGauss}. We use the tuning procedure in Section~\ref{sec:optimal-bandwidth} to choose the bandwidth parameter. In the example shown in Figure~\ref{fig:4d-example}, the estimated dimension was approximately $3.94 \approx 2 \max{(\chi^{\prime}_{\epsilon})}$ and $3.91 \approx 2 \max{(\chi^{\prime}_{\epsilon, \beta})}$ for the density and Kolmogorov operator estimation, respectively. We expect both values to be $4$, which suggests that the tuning procedure is performing adequately.}

\edits{Although the discrete Kolmogorov operator $\bm{L}_{\psi, 1}$ is sparse, storing its eigendecomposition becomes increasingly memory-intensive as $n $ increases. With our available resources, we found 
this is the limiting factor that prevents us from using more samples to estimate the solution to the Kolmogorov problem. However, if we do not need to estimate gradients $\nabla f$, we could solve $\bm{L}_{\psi, 1} \bm{f} = \bm{g}$ using a sparse matrix decomposition. Optimistically, we find the pointwise values of the solution in Figure~\ref{fig:4d-solution} are less sensitive to the number of samples $n$ than the estimate of the eigendecomposition. This suggests that the $L^2$ error in the computed solution may be less affected by the curse of dimensionality than the spectral convergence required by the eigenfunctions.
}

\subsection{Spherical manifold} \label{sec:spherical-manifold-example}

\begin{figure}
  \centering
  \begin{subfigure}{0.45\textwidth}
    \includegraphics[width=1.0\textwidth]{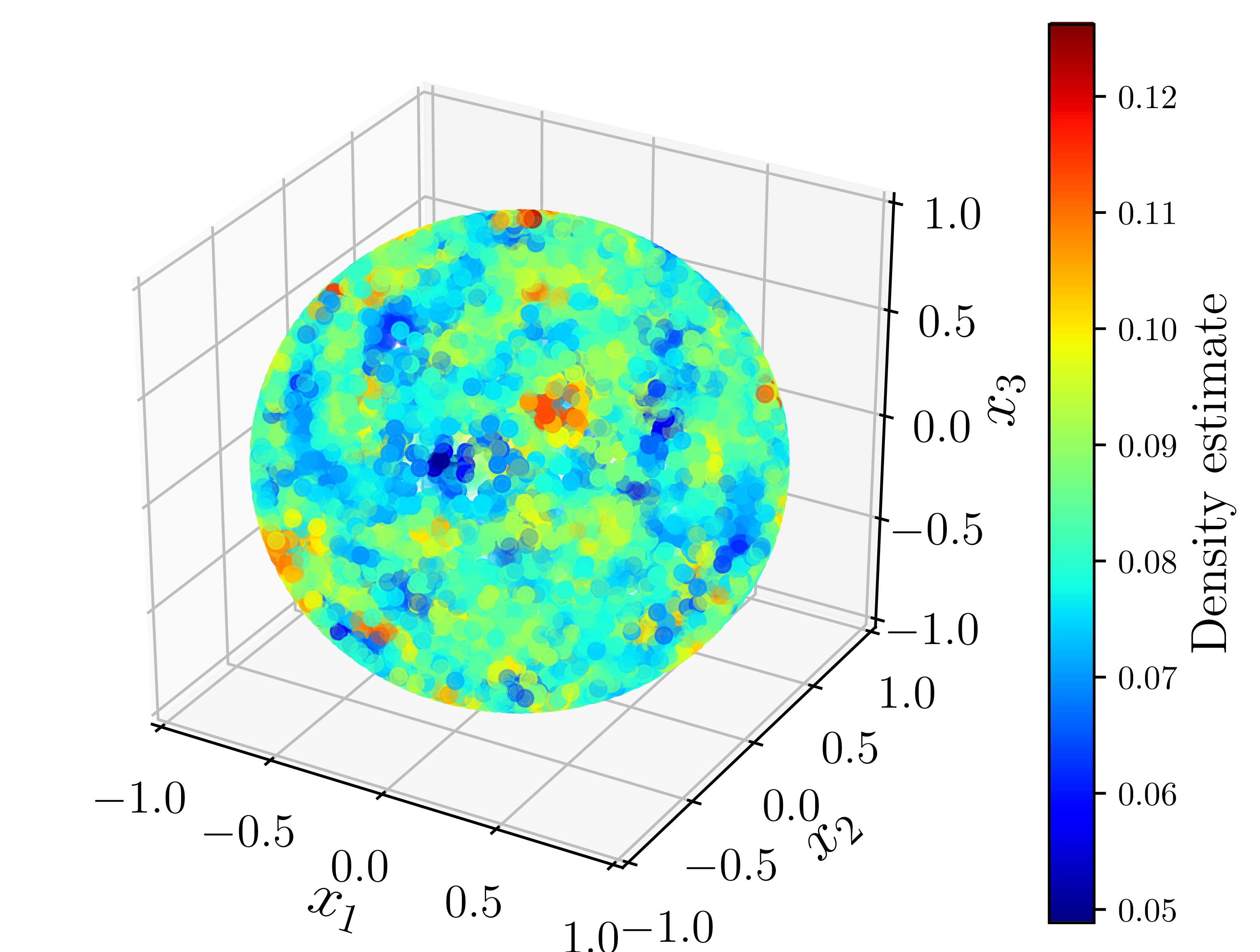}
    \caption{Estimated uniform density}
    \label{fig:SphericalManifold-UniformDensityEstimate}
  \end{subfigure}
  \begin{subfigure}{0.45\textwidth}
    \includegraphics[width=1.0\textwidth]{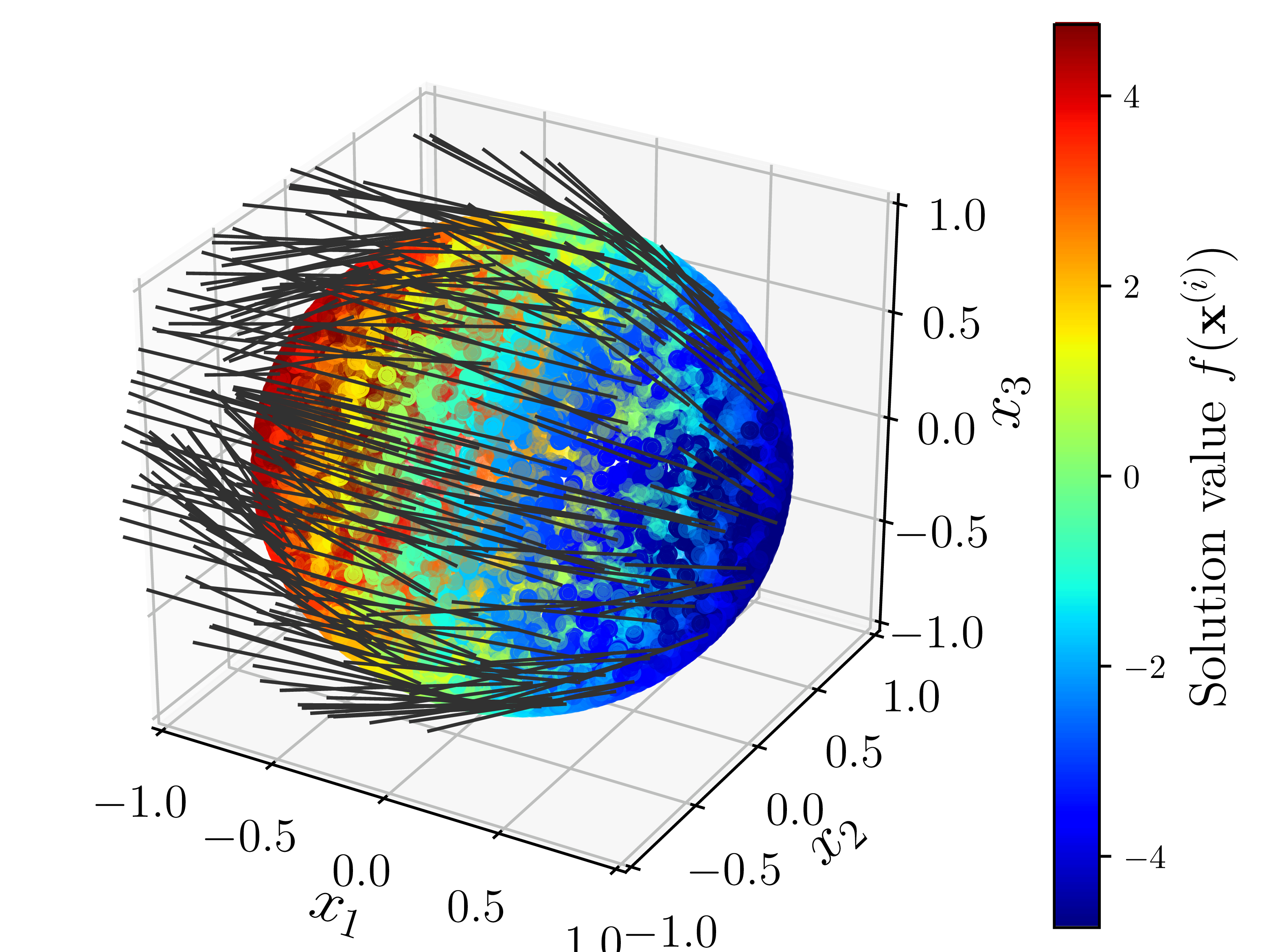}
    \caption{$c=1$}
    \label{fig:SphericalManifold-UniformDensitySolution_c1}
  \end{subfigure}
  \begin{subfigure}{0.45\textwidth}
    \includegraphics[width=1.0\textwidth]{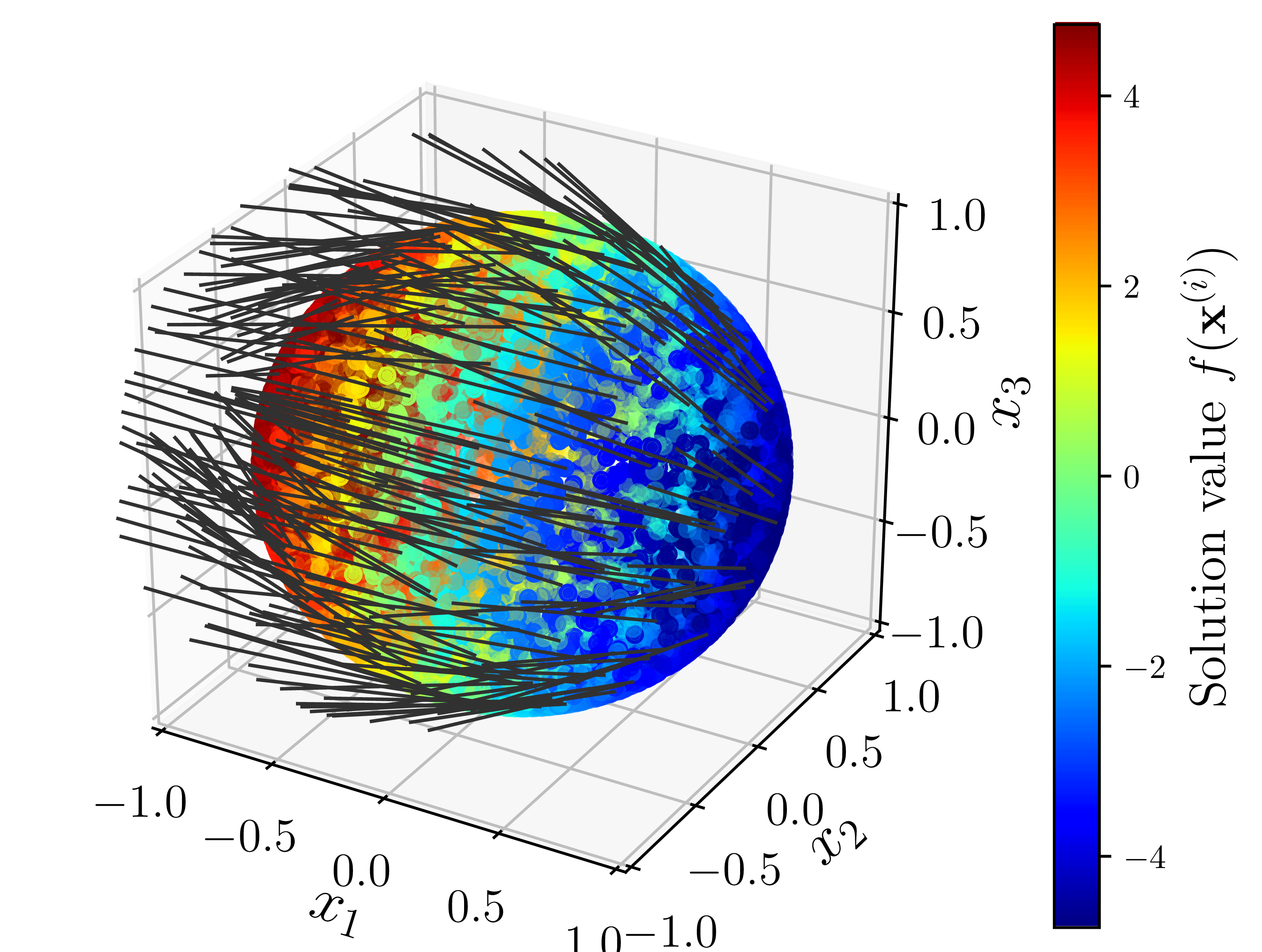}
    \caption{$c=2$}
    \label{fig:SphericalManifold-UniformDensitySolution_c2}
  \end{subfigure}
  \begin{subfigure}{0.45\textwidth}
    \includegraphics[width=1.0\textwidth]{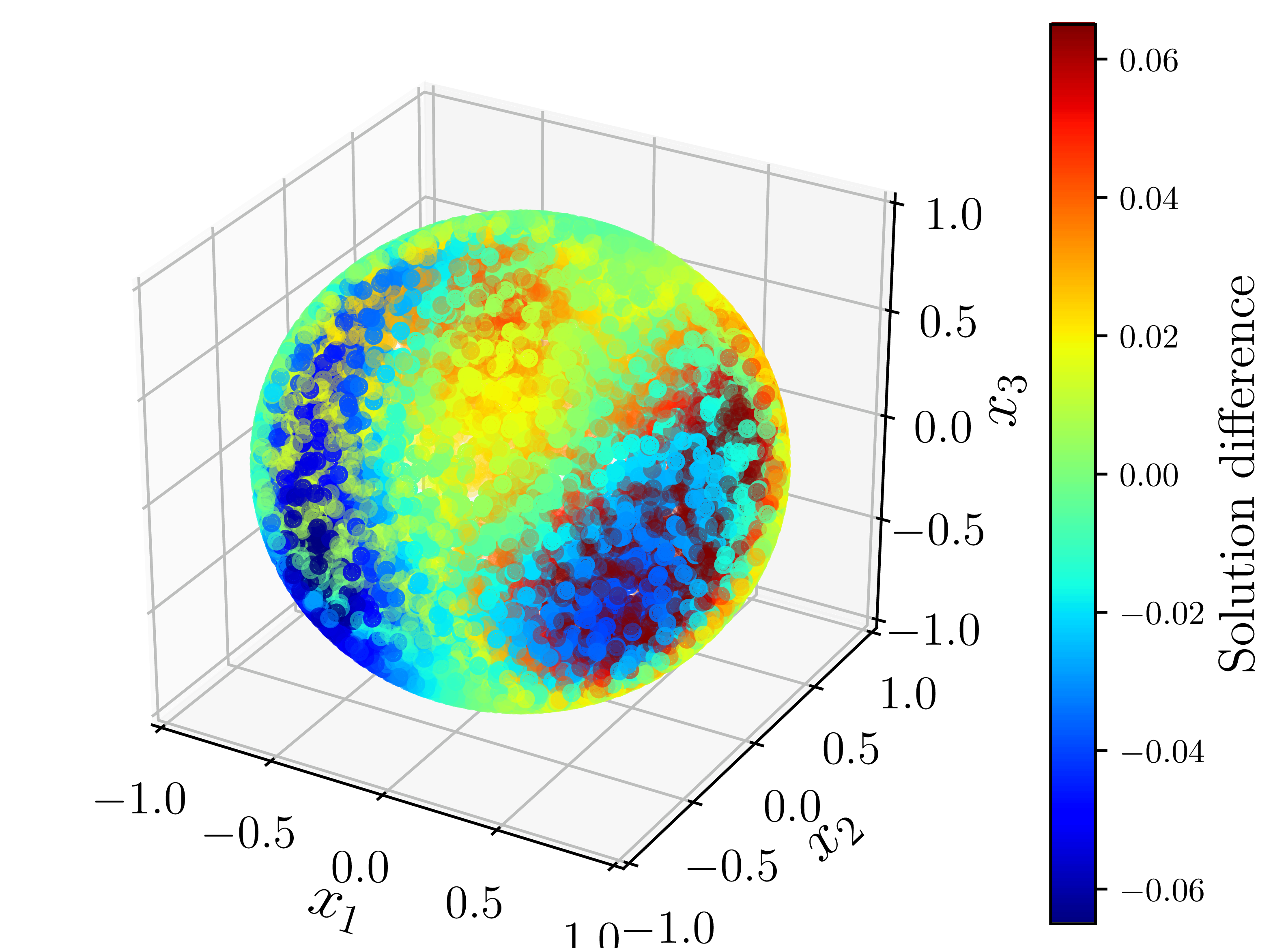}
    \caption{Difference ($c=1$ and $c=2$)}
    \label{fig:SphericalManifold-UniformDensityDifference_c0c1}
  \end{subfigure}
  \begin{subfigure}{0.45\textwidth}
    \includegraphics[width=1.0\textwidth]{fig_SphericalManifold_UniformSolution_c2.png}
    \caption{$c=0$}
    \label{fig:SphericalManifold-UniformDensitySolution_c0}
  \end{subfigure}
  \begin{subfigure}{0.45\textwidth}
    \includegraphics[width=1.0\textwidth]{fig_SphericalManifold_UniformDifference_c2.png}
    \caption{Difference ($c=1$ and $c=0$)}
    \label{fig:SphericalManifold-UniformDensityDifference_c0c10}
  \end{subfigure}
  \caption{Estimated uniform density $\psi_{uni}$ (a), numerical solutions to the Kolmogorov problem  with $c=0,\, 1,\, 2$ (e, b, c), and difference between the solutions (d, f).}
\end{figure}

In this example, we solve the Kolmogorov problem on the 2-sphere, $\Omega = S^2$. We use this example to explore the effects of varying the density-biasing parameter $c$. We first generate $n = 10^{4}$ samples from the uniform density on $S^2$ with respect to the standard (round) metric. Figure~\ref{fig:SphericalManifold-UniformDensityEstimate} shows the estimated density, which should be constant, $\psi_{uni}(\boldsymbol{x}) = (4 \pi )^{-1} \approx 0.08$. Since the density is constant on the manifold, the Kolmogorov operator simplifies so that it is independent of $c$: $\mathcal{L}_{\psi_{uni}, c} f = \Delta f$. Figures~\ref{fig:SphericalManifold-UniformDensitySolution_c1},~\ref{fig:SphericalManifold-UniformDensitySolution_c2}, and~\ref{fig:SphericalManifold-UniformDensitySolution_c0} show the solution to the Kolmogorov problem in~\eqref{eq:Kolmogorov} with source term $g(\boldsymbol{x}) = \boldsymbol{x} \cdot \boldsymbol{r}$ and \edits{$\boldsymbol{r} = (1,0,0)^\top$ and $c=1$, $c=2$, and $c=0$.} Figures~\ref{fig:SphericalManifold-UniformDensityDifference_c0c1}~and~\ref{fig:SphericalManifold-UniformDensityDifference_c0c10} show the difference between the solution \edits{with $c=1$ and $c=2$ and $c=0$,} respectively. We compute the bandwidth function~\eqref{eq:bandwdith-parameter} using $k_{nn} = 25$ nearest neighbors and set the entries of the density kernel matrix $\boldsymbol{K}_{\epsilon}$ to zero if they are below the threshold $10^{-2}$. We use the optimal bandwidth parameter $\epsilon$---see Section~\ref{sec:optimal-bandwidth}. These results show that the solutions are similar, as expected. However, the error is larger with larger $c$ because errors in the density estimate are more exaggerated. We see that regions with the largest error are the same in both cases, and are larger when $c=2$.

\begin{figure}[h!]
  \centering
  \begin{subfigure}{0.45\textwidth}
    \includegraphics[width=1.0\textwidth]{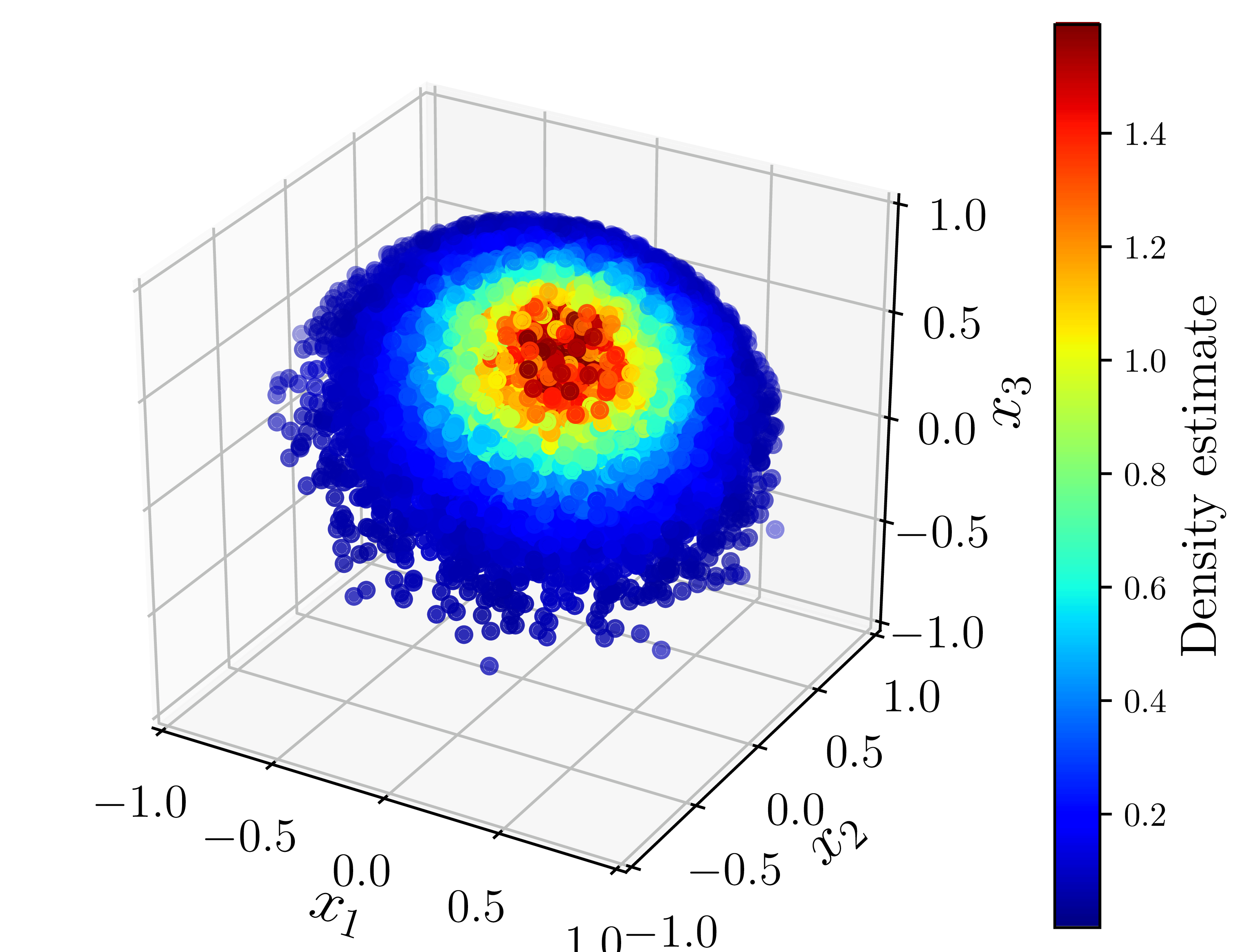}
    \caption{Estimated density}
    \label{fig:vonMises-distribution-estimated-density}
  \end{subfigure}
  \begin{subfigure}{0.45\textwidth}
    \includegraphics[width=1.0\textwidth]{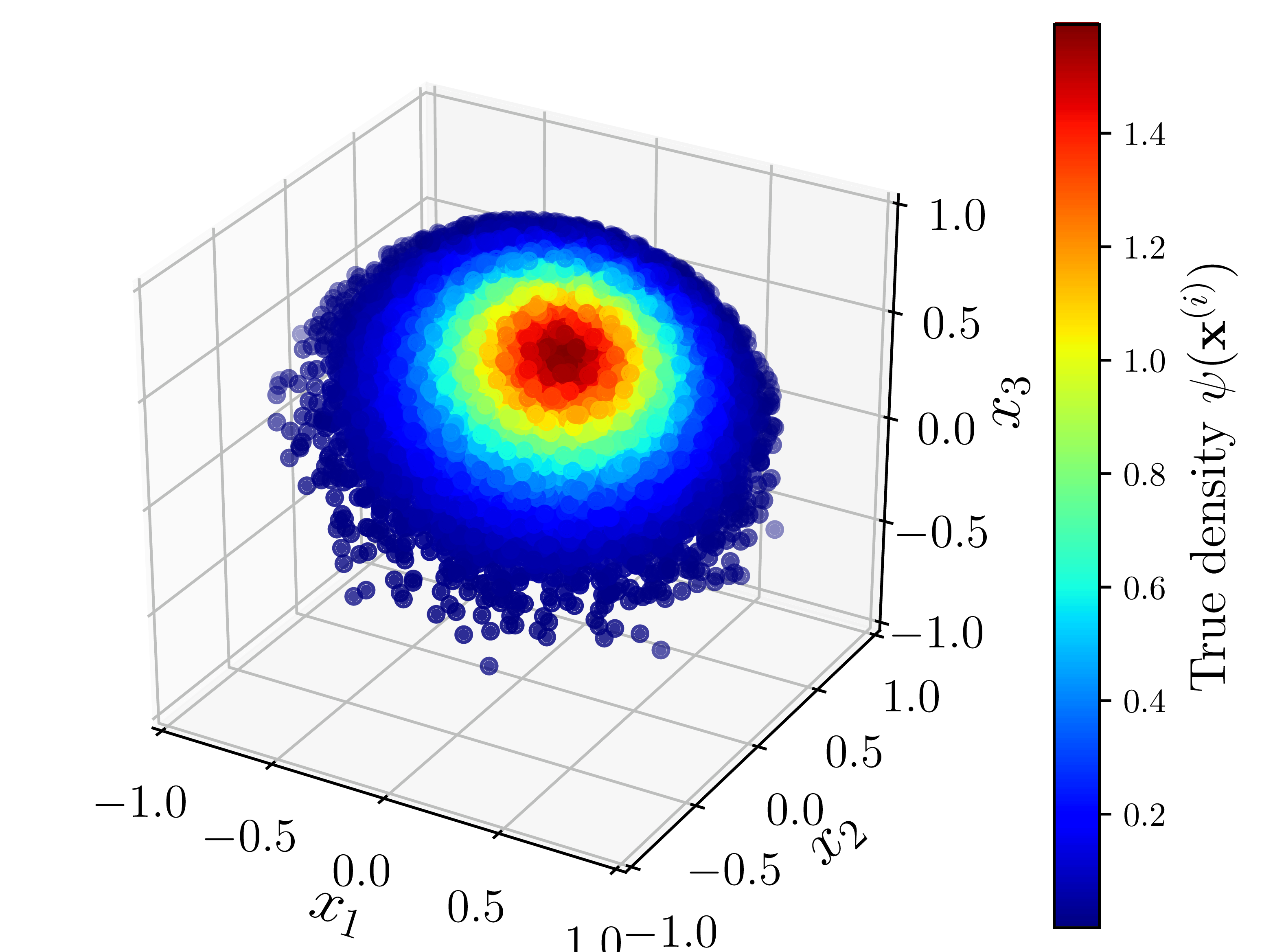}
    \caption{True density}
    \label{fig:vonMises-distribution-true-density}
  \end{subfigure}
  \caption{The density estimated using $n=10^4$ samples from the von Mises-Fisher distribution in~\eqref{eq:vonMises-distribution-density} with $\kappa = 10$ and $\boldsymbol{u} = \boldsymbol{\widetilde{u}} / \| \boldsymbol{\widetilde{u}} \|$, $\boldsymbol{\widetilde{u}} = (\frac{1}{2},\, -\frac{1}{2},\, 1)^\top$.}
  \label{fig:vonMises-distribution}
\end{figure}

\begin{figure}[h!]
  \centering
  \begin{subfigure}{0.45\textwidth}
    \includegraphics[width=1.0\textwidth]{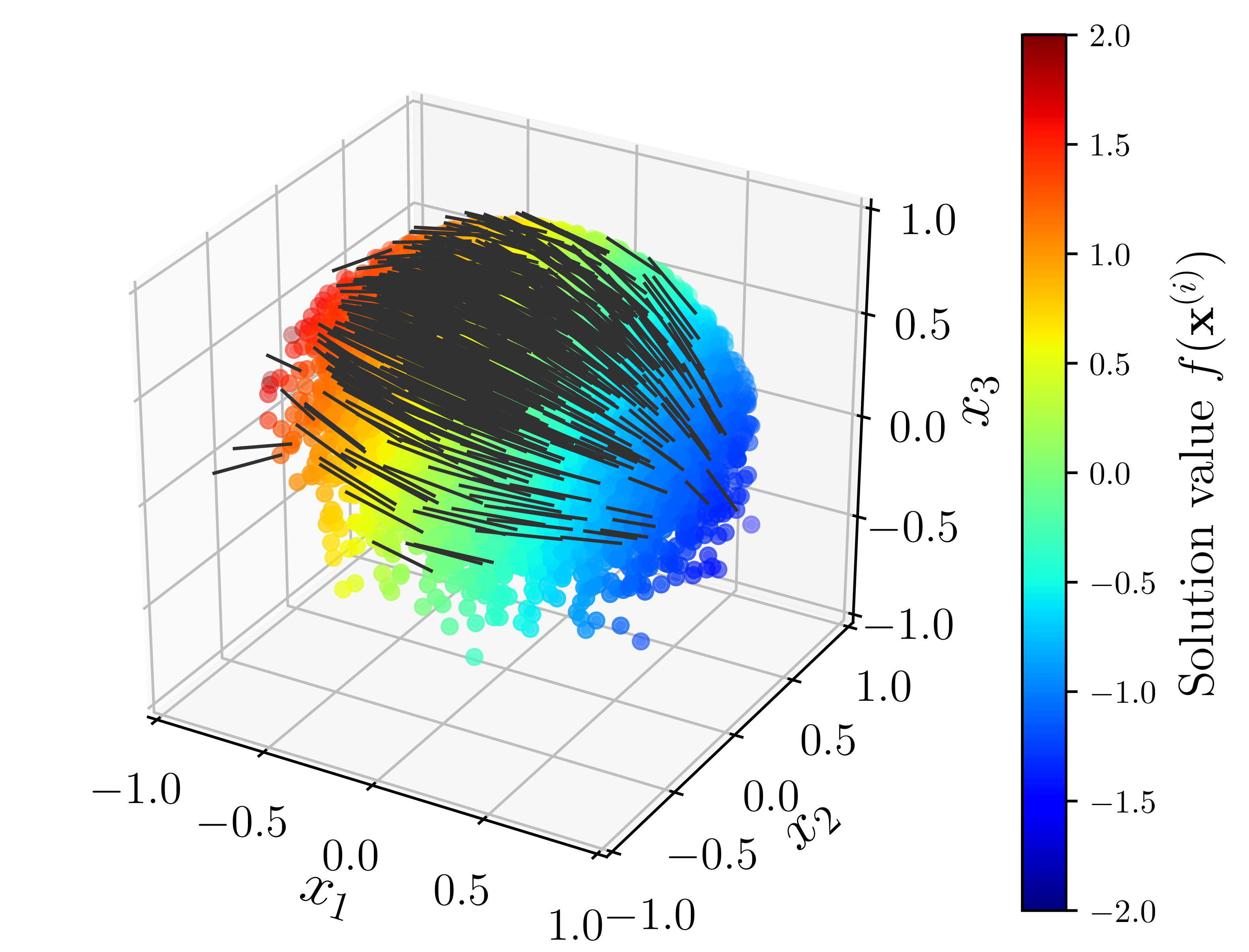}
    \caption{$c=0$}
  \end{subfigure}
  \begin{subfigure}{0.45\textwidth}
    \includegraphics[width=1.0\textwidth]{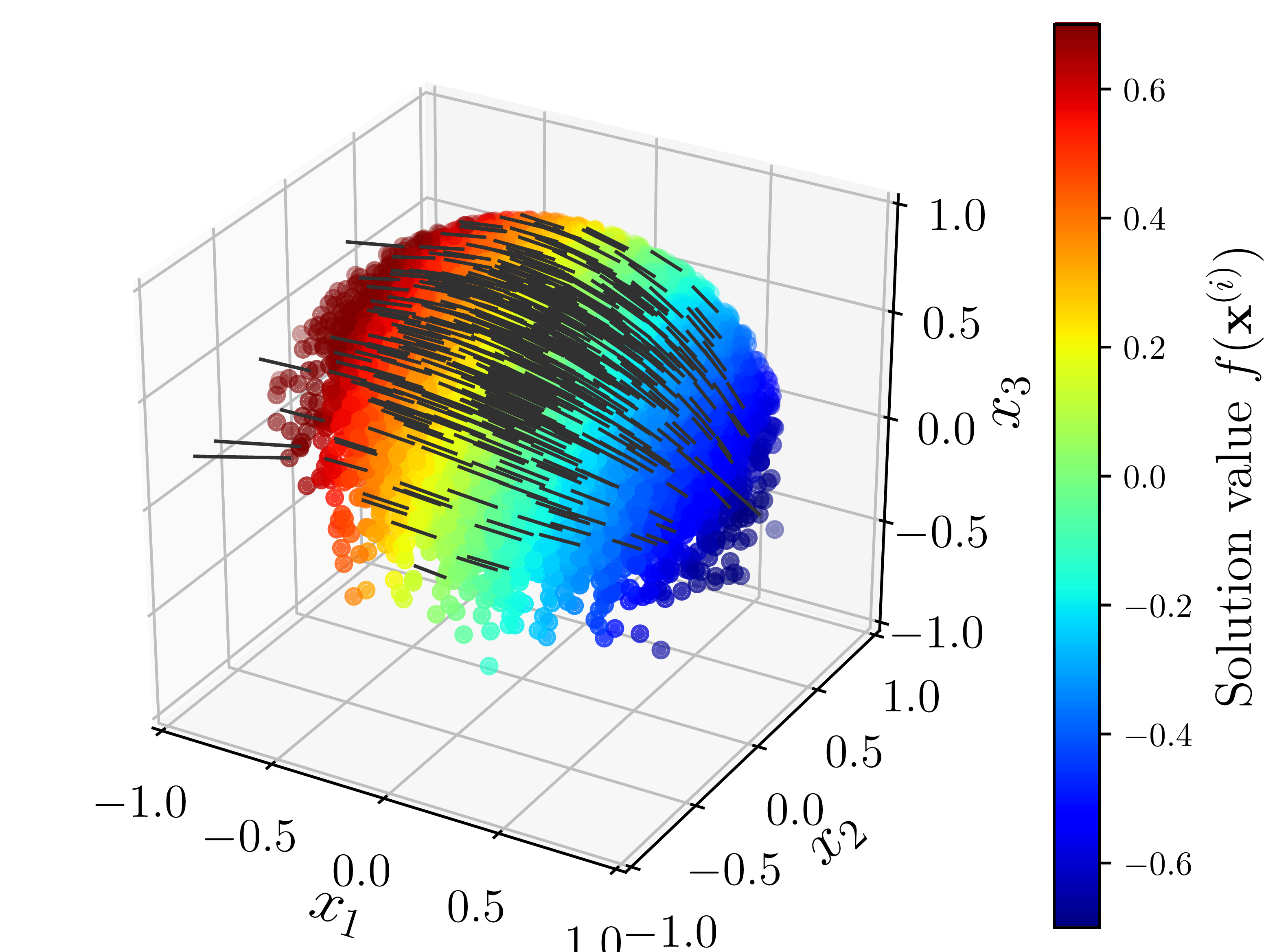}
    \caption{$c=1$}
    \label{fig:vonMises-Kolmogorov-solution-c1}
  \end{subfigure}
  \begin{subfigure}{0.45\textwidth}
    \includegraphics[width=1.0\textwidth]{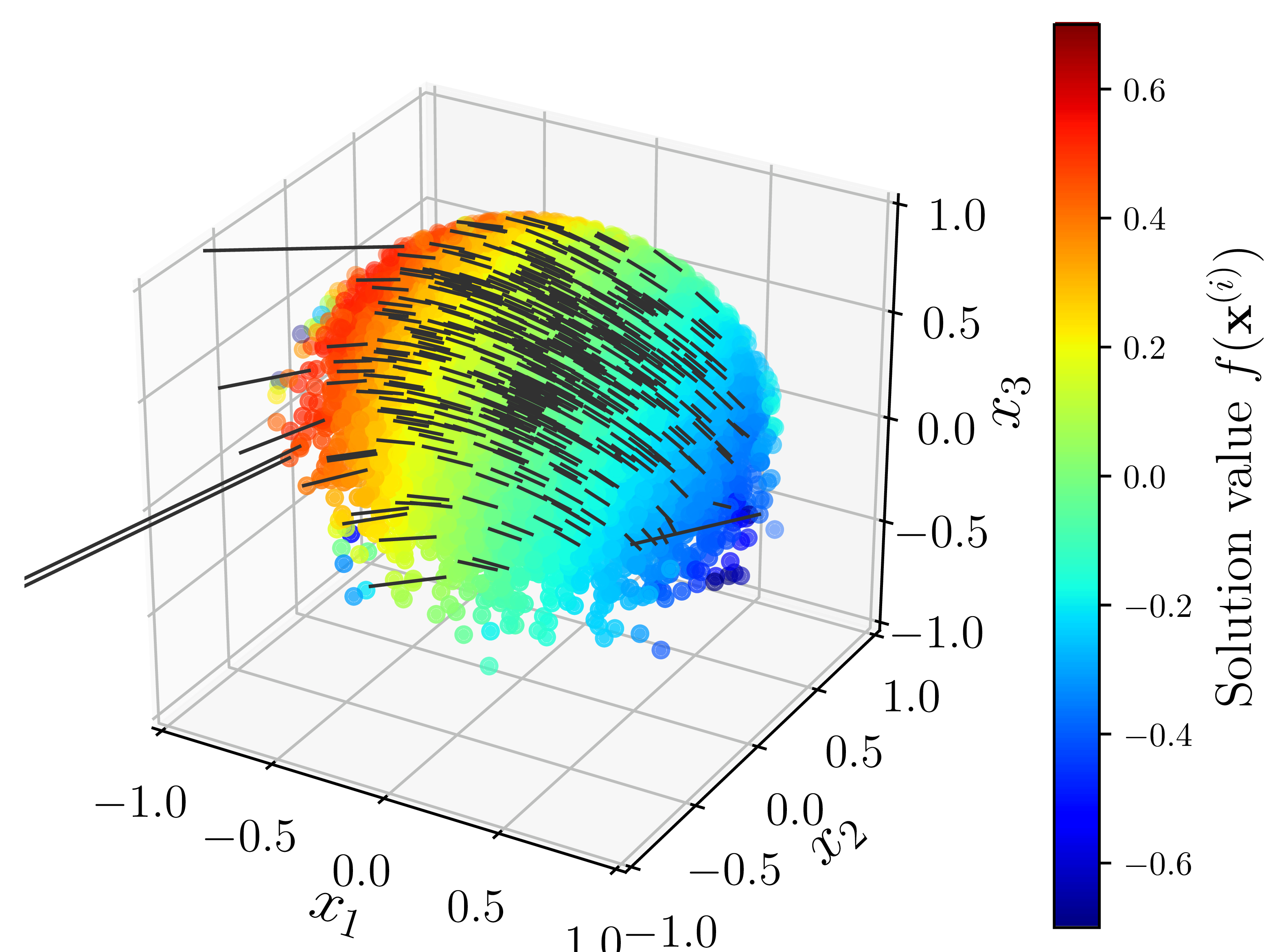}
    \caption{$c=2$}
    \label{fig:vonMises-Kolmogorov-solution-c2}
  \end{subfigure}
  \begin{subfigure}{0.45\textwidth}
    \includegraphics[width=1.0\textwidth]{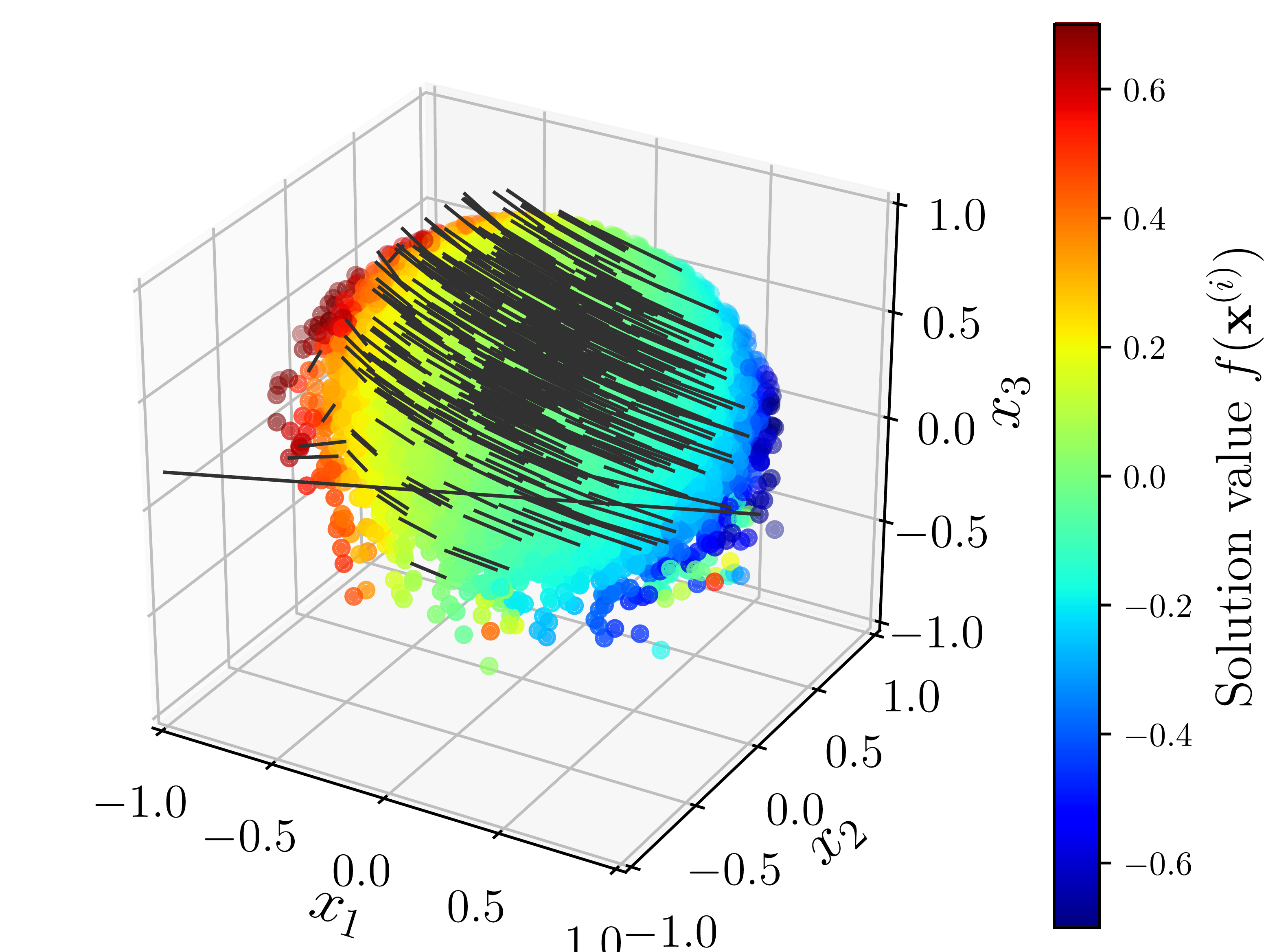}
    \caption{$c=4$}
    \label{fig:vonMises-Kolmogorov-solution-c4}
  \end{subfigure}
  \caption{Solution to the Kolmogorov problem in~\eqref{eq:Kolmogorov-problem} for the von Mises-Fisher distribution with $\kappa = 10$ and $\boldsymbol{u} = \boldsymbol{\widetilde{u}} / \| \boldsymbol{\widetilde{u}} \|$, $\boldsymbol{\widetilde{u}} = (\frac{1}{2},\, -\frac{1}{2},\, 1)^\top$ (Equation~\eqref{eq:vonMises-distribution-density}) with (a) $c=0$, (b) $c=1$, (c) $c=2$, and (d) $c=4$. The arrows are the gradient of the solution, which more strongly follow the curvature of the underlying distribution for larger values of $c$.}
  \label{fig:vonMises-Kolmogorov-solution}
\end{figure}

We also generate $n=2.5 \times 10^{4}$ samples on the unit sphere that are distributed according to the von Mises-Fisher distribution. In three dimensions, the density function is parameterized by $\kappa > 0$ and a unit vector $\boldsymbol{u}$ such that (Figure~\ref{fig:vonMises-distribution-true-density})
\begin{equation}
    \psi(\boldsymbol{x}) = \frac{\kappa}{2 \pi (\exp{(\kappa)}-\exp{(-\kappa)})} \exp{(\kappa \boldsymbol{u} \cdot \boldsymbol{x})}.
    \label{eq:vonMises-distribution-density}
\end{equation}
We emphasize that this distribution is defined on the unit sphere ($\| \boldsymbol{x} \| = 1$) and, therefore, the manifold dimension is $d=2$. We sample this distribution using rejection sampling: first sampling from the uniform distribution over the unit sphere, then accepting that sample with probability $\psi_{uni}(\boldsymbol{x})/(M \psi(\boldsymbol{x}))$, where $M = \max{(\psi_{uni}(\boldsymbol{x})/\psi(\boldsymbol{x}))}$. We use $n = 10^4$ samples to estimate the density with $\kappa = 10$ and $\boldsymbol{u} = \boldsymbol{\widetilde{u}} / \| \boldsymbol{\widetilde{u}} \|$ and $\boldsymbol{\widetilde{u}} = (\frac{1}{2},\, -\frac{1}{2},\, 1)^\top$---as shown in Figure~\ref{fig:vonMises-distribution-estimated-density}. As in the uniform case, we compute the bandwidth function~\eqref{eq:bandwdith-parameter} using $k_{nn} = 25$ nearest neighbors and set the entries of the density kernel matrix $\boldsymbol{K}_{\epsilon}$ to zero if they are below the threshold $10^{-2}$. We use the optimal bandwidth parameter $\epsilon$---see Section~\ref{sec:optimal-bandwidth}.

We solve the Kolmogorov problem in~\eqref{eq:Kolmogorov-problem} for the right hand side function $g(\boldsymbol{x}) = \boldsymbol{x} \cdot \boldsymbol{r}$ with \edits{$\boldsymbol{r} = (1,0,0)^\top$}. Figure~\ref{fig:vonMises-Kolmogorov-solution} compares numerical solutions $\bm f$ and their associated gradient vector fields corresponding to the von Mises-Fisher distribution with the operator parameter $c=0$, $c=1$, $c=2$, and $c=4$. When $c=0$, the Kolmogorov operator simplifies to the Laplace-Beltrami operator $\mathcal{L}_{0, \psi} = \Delta$, which is independent of the underlying density. In~\eqref{eq:Kolmogorov-operator-similarity} we show that the operator $\mathcal{L}_{\psi, c}$ is equivalent to $\mathcal{L}_{\widetilde{\psi}, 1}$ with $\widetilde{{\psi}} \propto \psi^c$. Therefore, larger values of $c$ more strongly bias the operator according to the density function $\psi$. \edits{For large $c$, $\tilde{\psi} \ll 1$ in the tails of the distribution, meaning samples $\bm{x} \sim \psi$ are disproportionately sampled in low-probability regions of $\tilde{\psi}$. We see in Figure~\ref{fig:vonMises-Kolmogorov-solution} that the solution is noisy in the tails of the distribution for $c=2$ and $c=4$. Heuristically, we obtain more accurate solutions setting $c=1$ and sampling from $\tilde{\psi}$ directly.}

\subsection{Evolution of a time-dependent density} \label{sec:advection-example}

In this example, we evolve a $d$-dimensional random variable $\boldsymbol{X}_t$ that is distributed according to a time-dependent density $\psi(\boldsymbol{x}; t)$. For simplicity, we assume that the density is defined over $\mathbb{R}^{d}$ (rather than some nonlinear manifold). However, the same procedure can evolve the random variables on a manifold $\Omega \subset \mathbb R^d$. We evolve an unnormalized density function $\mu(\boldsymbol{x}, t)$ according to 
\begin{equation}
    \partial_t \mu + \nabla \cdot (\boldsymbol{u} \mu) - \boldsymbol{D}_{\boldsymbol{\sigma}} \mu = \mu g^{\prime},
    \label{eq:unnormalized-Fokker-Planck}
\end{equation}
where $\boldsymbol{u}(\boldsymbol{x}, t)$ is a prescribed velocity field, $\boldsymbol{D}_{\boldsymbol{\sigma}}$ is a diffusion operator defined by the matrix $\boldsymbol{\sigma}(\boldsymbol{x}, t) \in \mathbb{R}^{d \times d}$
\begin{equation*}
    \boldsymbol{D}_{\boldsymbol{\sigma}} \mu = \frac{1}{2} \sum_{s=1}^{d} \sum_{i=1}^{d} \frac{\partial}{\partial x_i} \left( \sigma_{is} \sum_{j=1}^{d} \frac{\partial}{\partial x_j} \left( \sigma_{js} \mu \right) \right),
\end{equation*}
and $g^{\prime}(\boldsymbol{x}, t)$ is an external source function. \edits{We find problems with this form in kinetic theory, where $\mu$ is the density function over material space and velocity. The material moves around the domain according to an advection-diffusion model and $g$ represents sources/sinks of material. We leave the specifics of this application to future work and solve the generic version of this problem here.} Given the unnormalized density $\mu$, we define the time-dependent  probability density function $\psi(\bm x; t) = \mu(\bm x; t)/M(t)$, where $M(t) = \int_{\mathbb{R}^{d}} \mu(\boldsymbol{x}, t) \, d\boldsymbol{x} >0$ is a normalization factor. Letting $\bar{g} = \int_{\mathbb{R}^{d}} g^{\prime} \psi \, d\boldsymbol{x}$, the normalization factor evolves according to 
\begin{equation*}
    \partial_t M = M \bar{g},
\end{equation*}
and the normalized density evolves according to
\begin{equation}
    \partial_t \psi + \nabla \cdot (\boldsymbol{u} \psi) - \boldsymbol{D}_{\boldsymbol{\sigma}} \psi = \psi g,
    \label{eq:Fokker-Planck}
\end{equation}
with $g = g^{\prime} - \bar{g}$. In the following example, our goal is to compute particle-based approximations to the solution of~\eqref{eq:Fokker-Planck}. 

First, note that when $g=0$, \eqref{eq:Fokker-Planck} is the Fokker-Planck equation and a random variable $\boldsymbol{X}_t$ that is distributed according to $\psi$ evolves according to
\begin{equation*}
    d\boldsymbol{X}_t = \boldsymbol{u} \, dt + \boldsymbol{\sigma} \, d\boldsymbol{W}_t,
\end{equation*}
where $\boldsymbol{W}_t$ is a $d$-dimensional Wiener process. When $g \neq 0$, we can construct solutions to~\eqref{eq:Fokker-Planck} by solving the Kolmogorov problem 
\begin{equation}
    \mathcal{L}_{1,\psi} f = \psi^{-1} \nabla \cdot (\psi \nabla f) = g, \quad \mathbb E_{\psi(\cdot;t)} f(\cdot; t) = 0,
    \label{eq:advection-Diffusion-Laplace}
\end{equation}
which allows us to rewrite~\eqref{eq:Fokker-Planck} as
\begin{equation}
    \partial_t \psi + \nabla \cdot ((\boldsymbol{u} - \nabla f) \psi) - \boldsymbol{D}_{\boldsymbol{\sigma}} \psi = 0.
    \label{eq:modified-Fokker-Planck}
\end{equation}
The Kolmogorov problem in~\eqref{eq:advection-Diffusion-Laplace} and the corresponding gradient field $\nabla f$ define an effective velocity $\boldsymbol{u^{\prime}} = \boldsymbol{u} - \nabla f$. The corresponding random variable evolves according to 
\begin{equation}
\label{eq:SDE}
    d \boldsymbol{X}_t = \boldsymbol{u^{\prime}} \, dt + \boldsymbol{\sigma} \, d\boldsymbol{W}_t.
\end{equation}
Therefore, we can construct our particle-based approximation to~\eqref{eq:Fokker-Planck} by (i) computing the effective velocity $\bm u'$; and (ii) evolving realizations of $\bm X_t$ according to~\eqref{eq:SDE}. Notationally, we distinguish between the random variable $\boldsymbol{X}_t$ and a specific realization $\boldsymbol{x}_t$ using upper- versus lower-case letters. The $t$ subscript further distinguishes time-dependent samples from the independent coordinate $\boldsymbol{x}$.

We begin by drawing $n$ samples that are initially distributed according to the density $\psi(\cdot; 0)$ (i.e., $\boldsymbol{x}_1^{(i)} \sim \psi(\cdot; 0)$). We evolve the $n$ samples so that 
\begin{equation*}
    \mathbb E_{\psi(\cdot;t)} h  = \int_{\mathbb{R}^{d}} h(\bm x) \psi(\bm x, t ) \, d\boldsymbol{x} = \lim_{n \rightarrow \infty} \frac{1}{n} \sum_{i=1}^{n} h(\boldsymbol{x}_t^{(i)})
\end{equation*}
for all integrable $h$. Each timestep consists of two steps: (i) use the procedures in Sections~\ref{sec:discrete-kolmogorov} and~\ref{sec:solution-gradients} to estimate the gradient $\nabla f$ at each sample, and (ii) update each sampling using an explicit time-stepping algorithm. Letting $\boldsymbol{u}^{(i)}$ and $\nabla f^{(i)}$ be the velocity and gradient estimate at the $i^{th}$ sample, we update each sample $\boldsymbol{x}_t^{(i)}$ over a timestep $\Delta t$ using
\begin{equation}
    \boldsymbol{x}_{t+\Delta t}^{(i)} = \boldsymbol{x}_{t}^{(i)} + \Delta t (\boldsymbol{u}^{(i)} - \nabla f^{(i)}) + \sqrt{\Delta t} \boldsymbol{\sigma} \boldsymbol{W},
    \label{eq:discrete-sample-evolution}
\end{equation}
where $\boldsymbol{W} \sim \mathcal{N}(\boldsymbol{0}, \boldsymbol{I})$. 

\edits{As a basic validation of our method, we consider the case with $g(\bm{x})=(\bm{x}-\bm{\bar{x}}) \cdot \bm{r}$, $\bm{r} = (1, 0)$, where $\bm{\bar{x}} = \mathbb{E}_{\psi(\cdot, t)} \bm{x}$.} We also set $\boldsymbol{u} = \boldsymbol{0}$ and $\boldsymbol{\sigma} = \boldsymbol{I}$. In this case, the solution should diffuse such that $\boldsymbol{x}_{t+\Delta t}^{(i)} = \boldsymbol{x}_{t}^{(i)} + \sqrt{\Delta t} \boldsymbol{W}$. \edits{We compute the effective velocity $\boldsymbol{u^{\prime}} = -\nabla f$, where $f$ solves~\eqref{eq:advection-Diffusion-Laplace}, i.e.,  $\mathcal{L}_{1,\psi} f = g$ and $\mathbb E_{\psi(\cdot;t)} f(\cdot; t) = 0$. If we initially sample from the standard Gaussian distribution $\psi(\bm{x}; 0) = \mathcal{N}(\bm{x}; \bm{0}, \bm{I})$, then the exact solution is $-\nabla f^{(i)} = (1, 0)$} Figure~\ref{fig:density-evolution-diffusion-only} shows the solution; as expected, the particles diffuse isotropically and the effective velocity \edits{translates the samples in the direction $\bm{r} = (1, 0)$}.

\begin{figure}[h!]
  \centering
  \begin{subfigure}{0.45\textwidth}
    \includegraphics[width=1.0\textwidth]{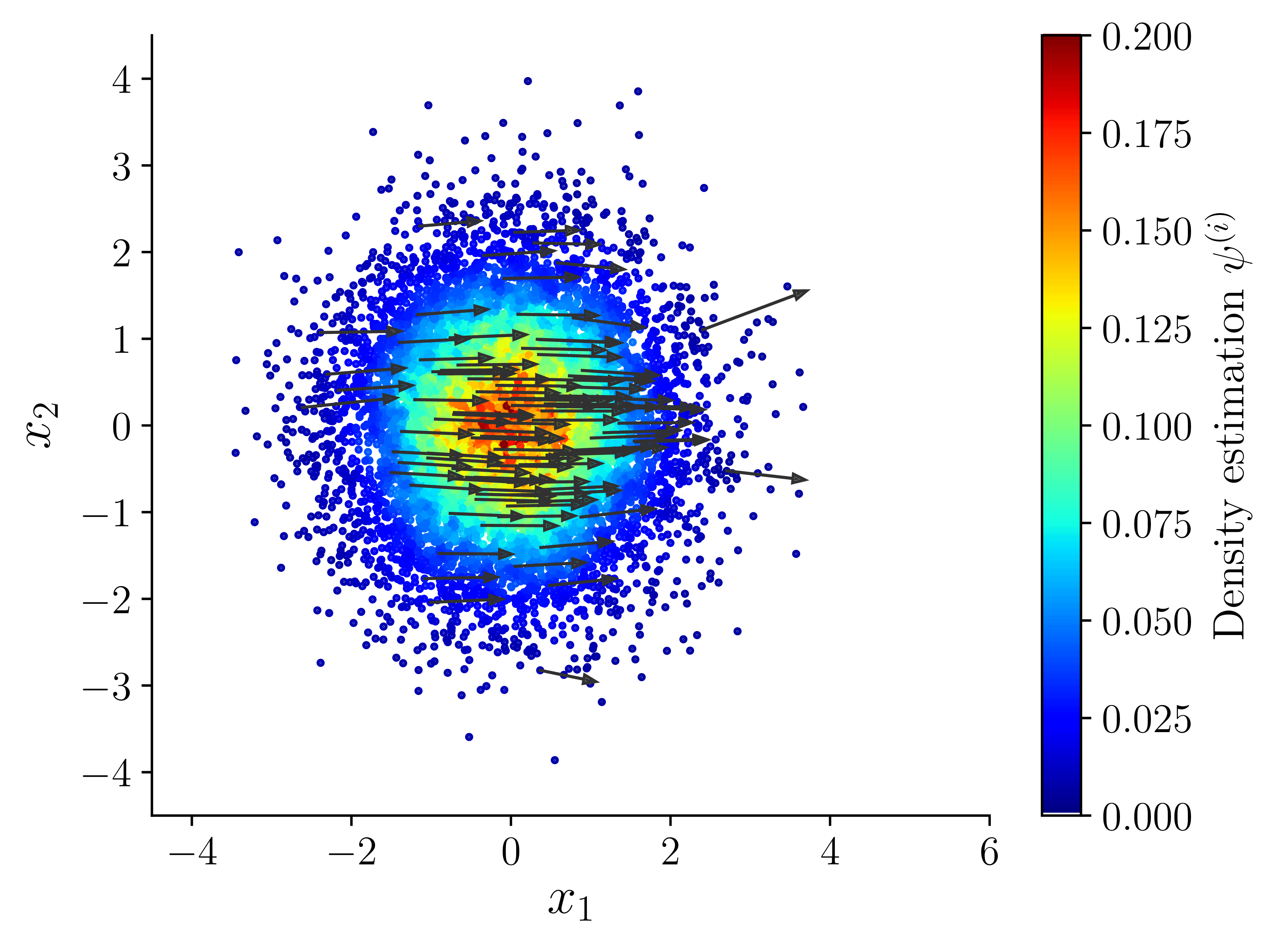}
    \caption{$t=0$}
  \end{subfigure}
  \begin{subfigure}{0.45\textwidth}
    \includegraphics[width=1.0\textwidth]{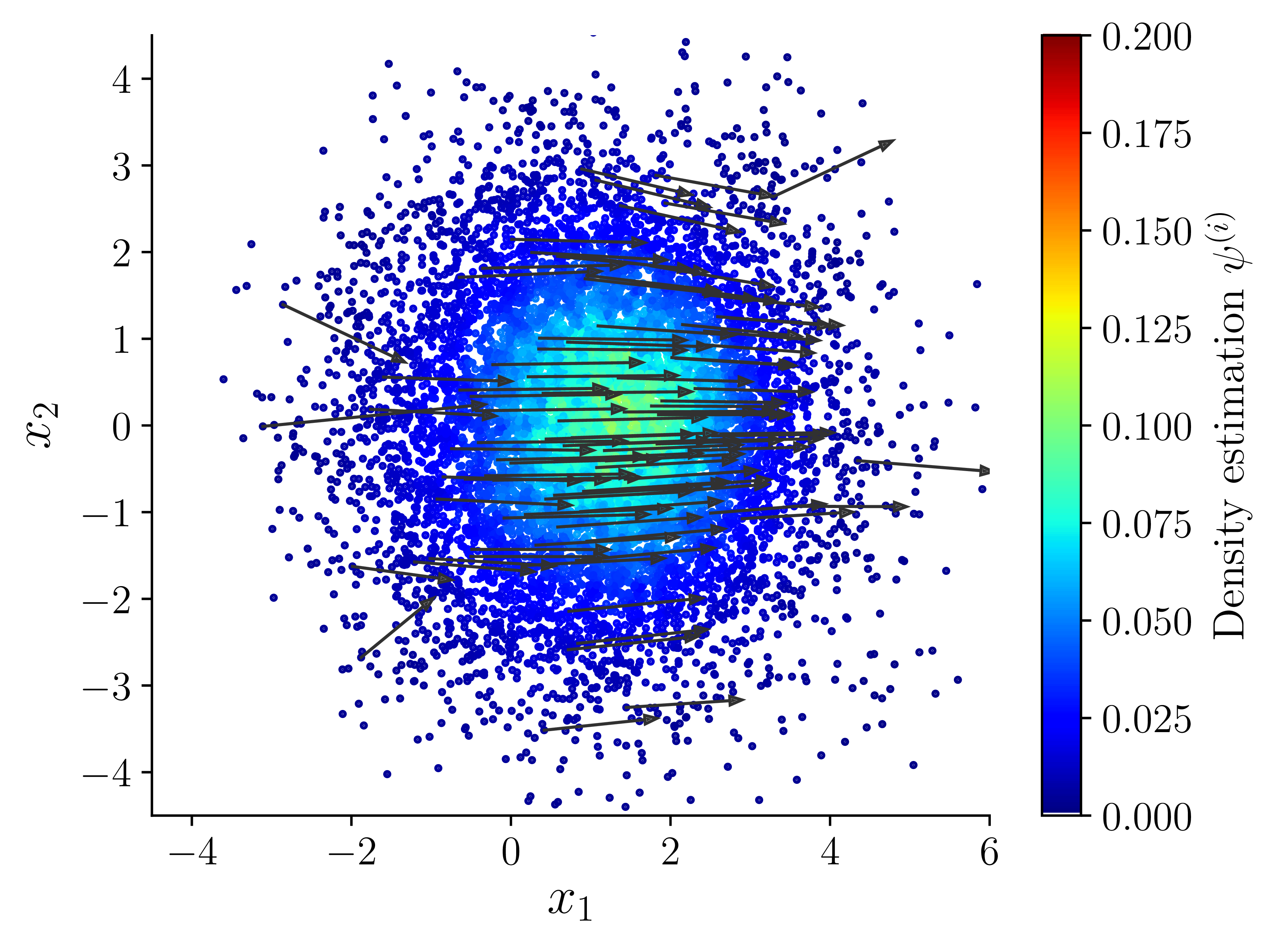}
    \caption{$t=1$}
  \end{subfigure}
  \caption{The solution to the density evolution problem in~\eqref{eq:Fokker-Planck} with $\boldsymbol{u} = \boldsymbol{0}$, $\boldsymbol{\sigma} = \boldsymbol{I}$, and $g = 0$. To verify our method, we solve the homogeneous problem \edits{$\mathcal{L}_{1,\psi} f = (\bm{x}-\bm{\bar{x}}) \cdot \bm{r}$, with $\bm{\bar{x}} = \mathbb{E}_{\psi(\cdot; t)} \bm{x}$ and $\bm{r} = (1, 0)$} such that $\mathbb E_{\psi(\cdot;t)}f = 0$, and define the effective velocity \edits{$\boldsymbol{u^{\prime}} = -\nabla f = (1, 0)$}. As expected, the effective velocity \edits{translates the samples in the direction $\bm{r}$} and the solution diffuses according to $\boldsymbol{D}_{\boldsymbol{\sigma}}$.}
  \label{fig:density-evolution-diffusion-only}
\end{figure}

We now consider a physically motivated example. We could, for example, interpret $\mu(\boldsymbol{x}, t)$ as the concentration of a non-reactive tracer and let $M(t) = \int_{\mathbb{R}^{d}} \mu(\boldsymbol{x}, t) \, d\boldsymbol{x}$ be the total amount of the tracer in the domain. Suppose that $\mu$ evolves according to~\eqref{eq:unnormalized-Fokker-Planck} with known steady state velocity $\boldsymbol{u}$. For example, the velocity may satisfy Darcy's law. The model~\eqref{eq:unnormalized-Fokker-Planck} simulates how the tracer flows through the domain.

\begin{figure}[h!]
  \centering
  \begin{subfigure}{0.45\textwidth}
    \includegraphics[width=1.0\textwidth]{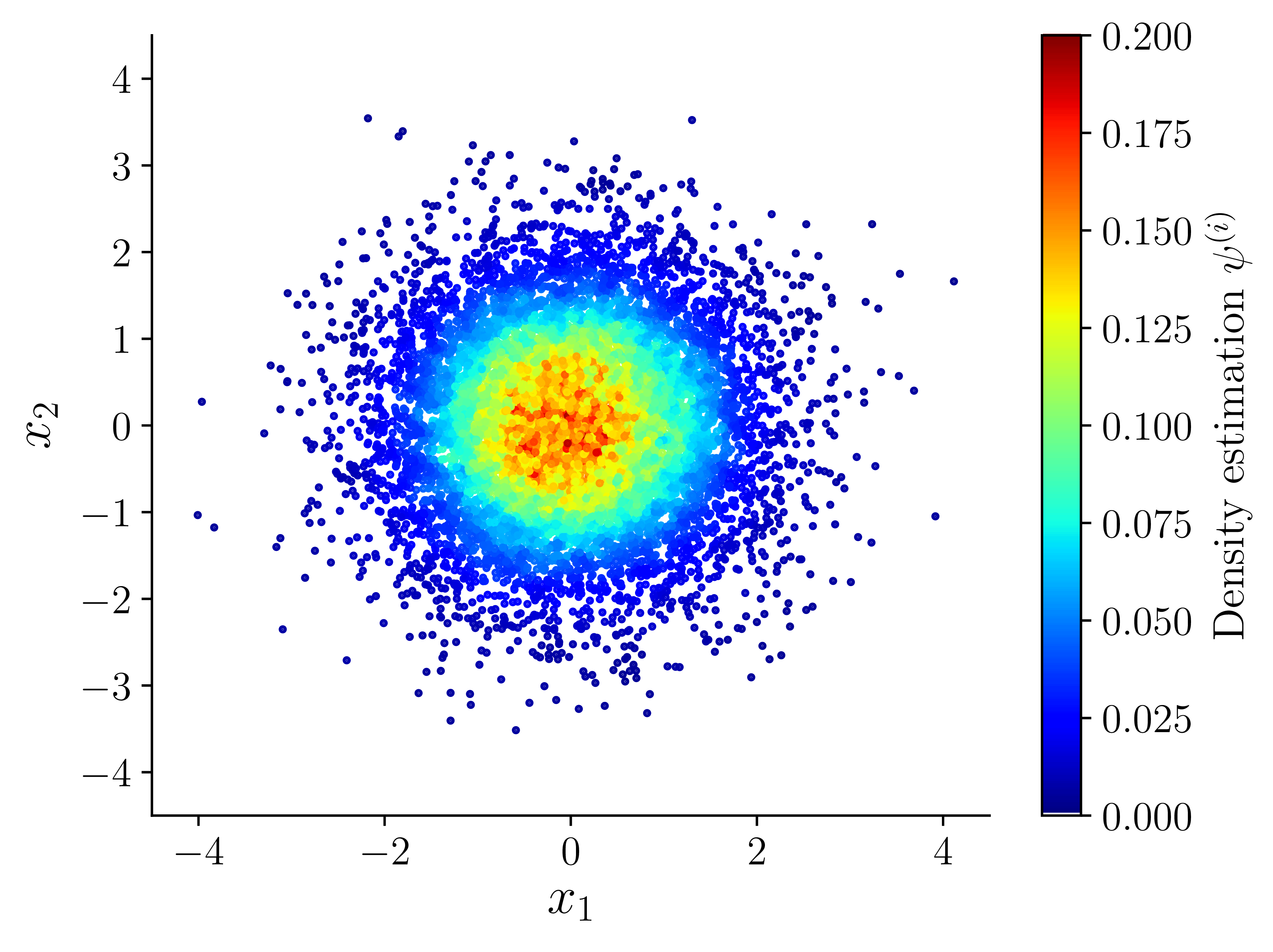}
    \caption{$t=0$}
    \label{fig:density-evolution-with-source-000}
  \end{subfigure}
  \begin{subfigure}{0.45\textwidth}
    \includegraphics[width=1.0\textwidth]{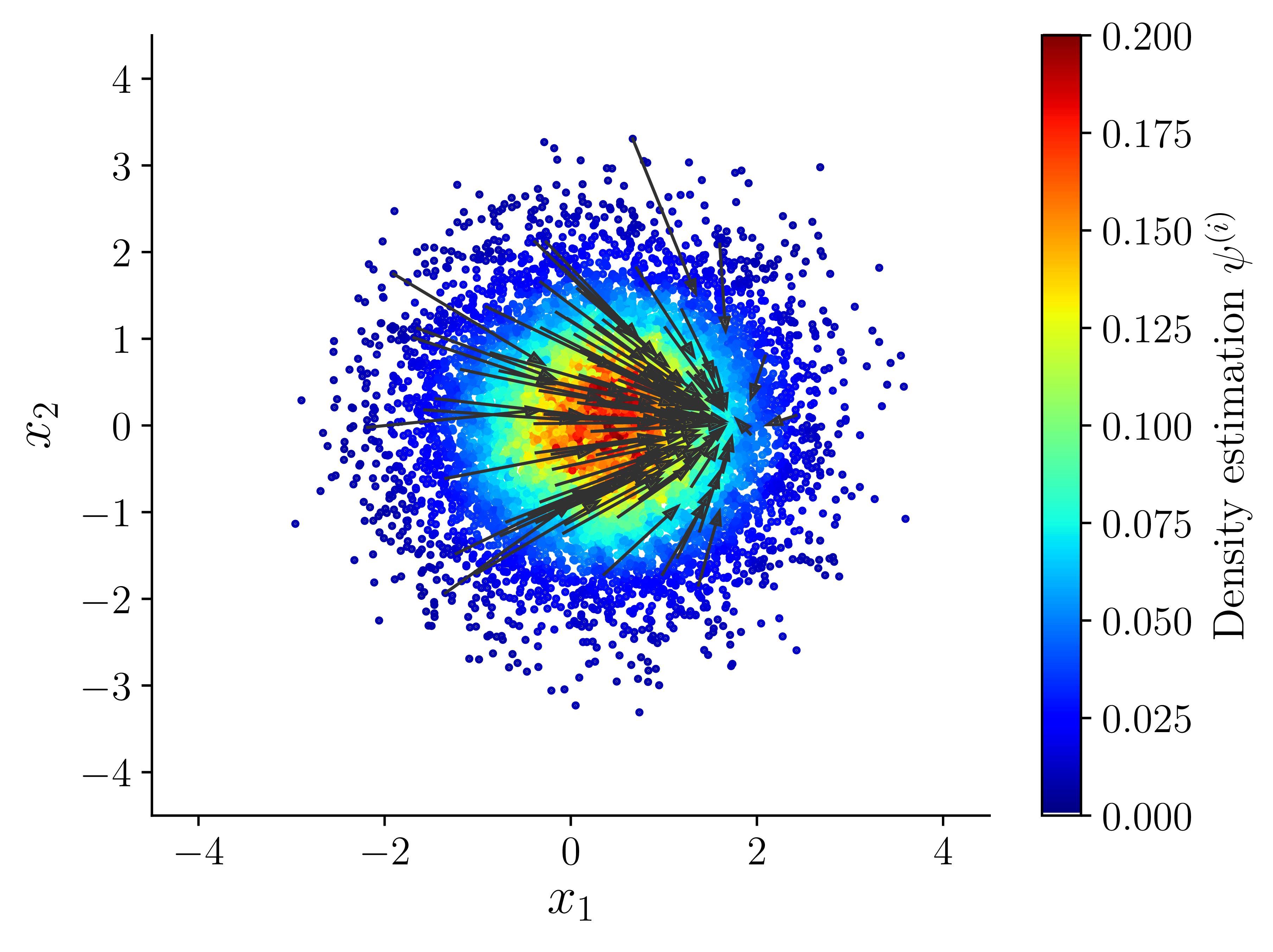}
    \caption{$t=0.3$}
    \label{fig:density-evolution-with-source-333}
  \end{subfigure}
  \begin{subfigure}{0.45\textwidth}
    \includegraphics[width=1.0\textwidth]{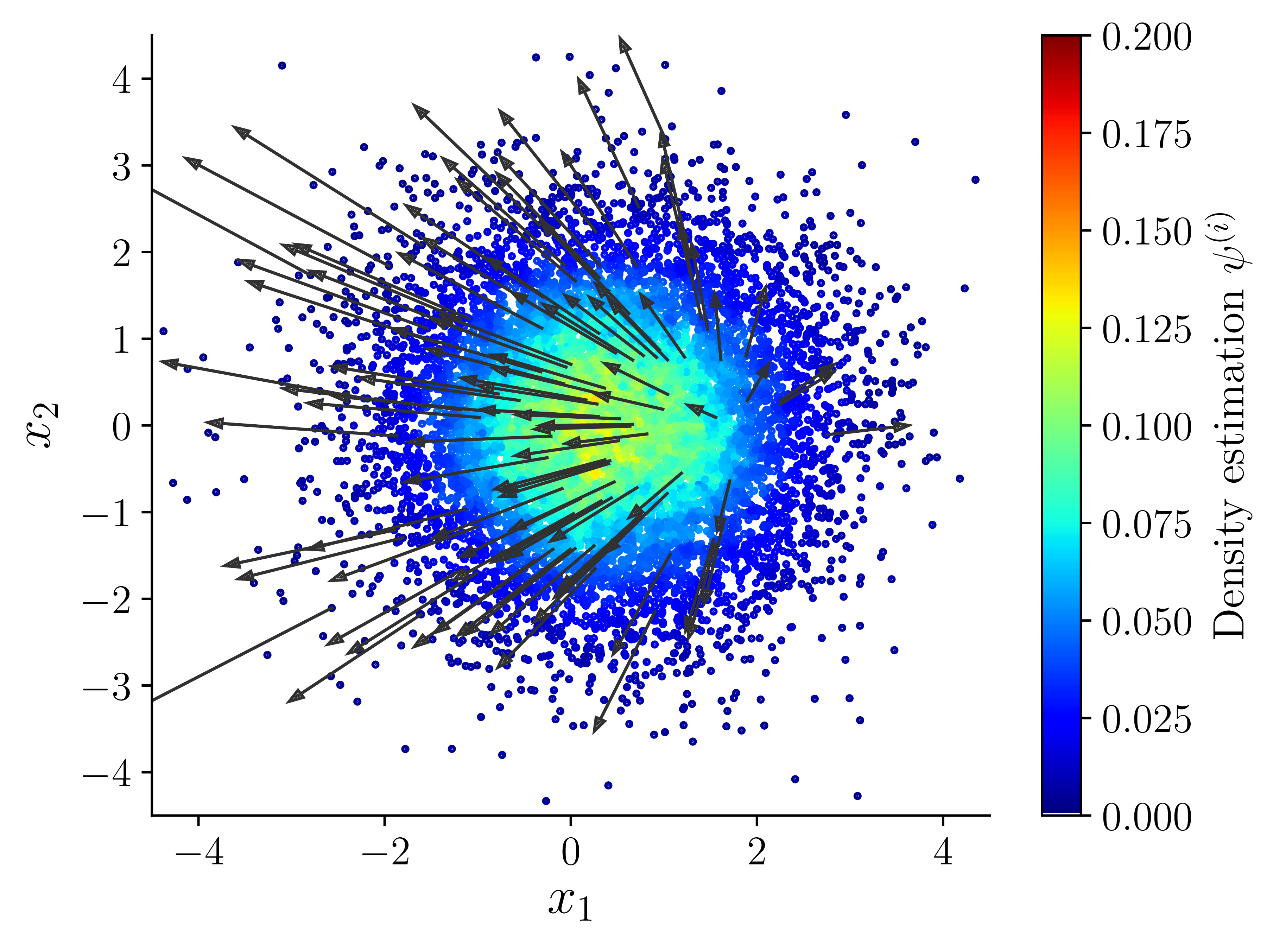}
    \caption{$t=0.7$}
    \label{fig:density-evolution-with-source-666}
  \end{subfigure}
  \begin{subfigure}{0.45\textwidth}
    \includegraphics[width=1.0\textwidth]{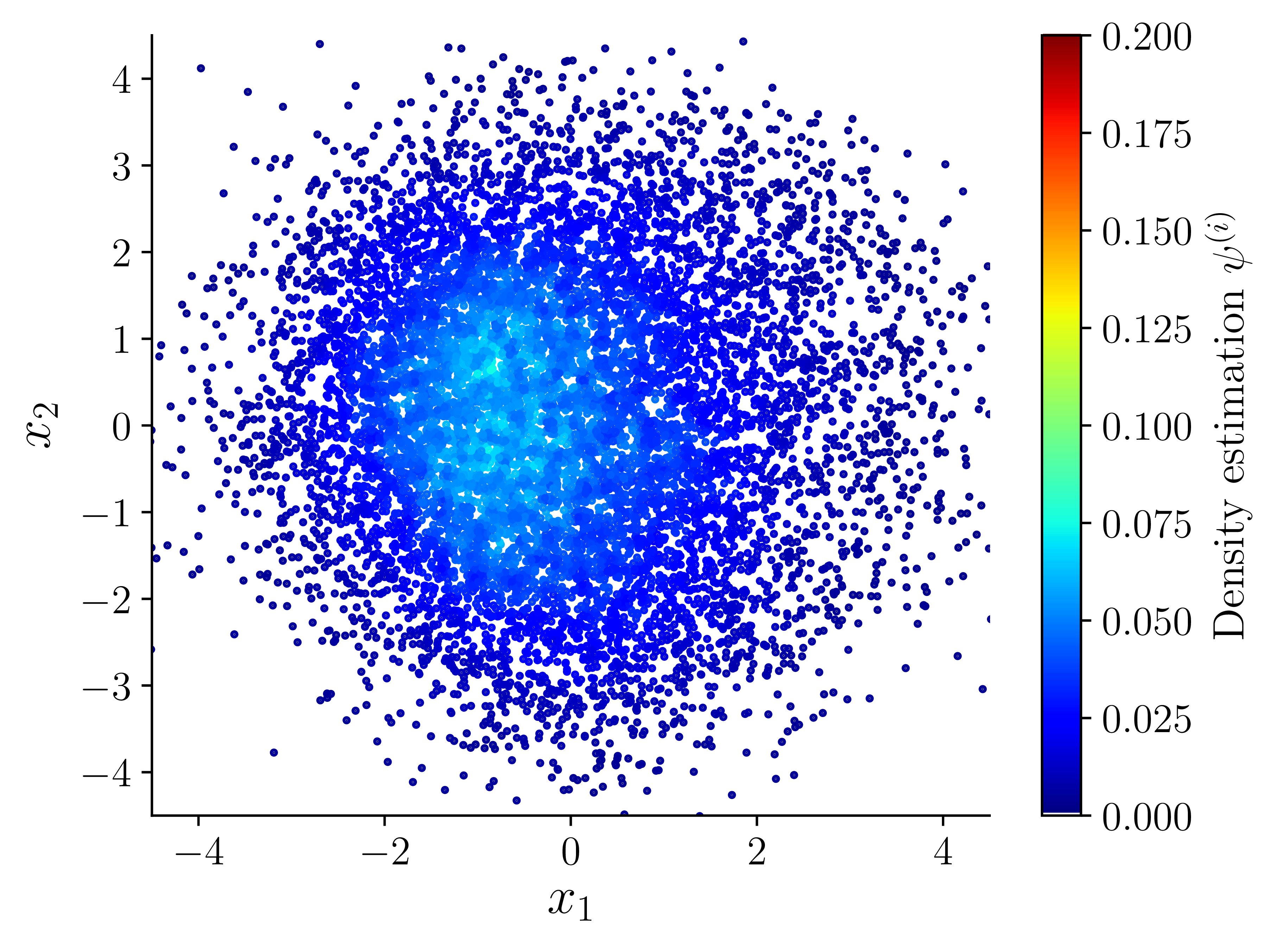}
    \caption{$t=1$}
    \label{fig:density-evolution-with-source-1000}
  \end{subfigure}
  \caption{The solution to the density evolution problem in~\eqref{eq:Fokker-Planck} with $\boldsymbol{u} = \boldsymbol{0}$, $\boldsymbol{\sigma} = \boldsymbol{I}$, and $g = g^{\prime}-\bar{g}$ with $g^{\prime}$ defined in~\eqref{eq:advection-diffusion-forcing} and $\bar{g} = \frac{1}{n} \sum_{i=1}^{n} f^{\prime}(\boldsymbol{x}^{(i)}, t)$. Colors show the solution $f$ to the weighted Laplace problem $\mathcal{L}_{\psi, 1} f = g$ with $\mathbb E_{\psi(\cdot;t)} f(\cdot;t) = 0 $, and the arrows show the effective velocity $\bm{u^{\prime}} = -\nabla f$. Panels~(b) and~(c) show how sources/sinks result in an effective velocity that redistributes mass density.}
  \label{fig:density-evolution-with-source}
\end{figure}

We add/remove mass through a ``well'' to demonstrate how a source term affects the effective velocity. In our physically motivated example, the addition or removal of the tracer may occur to to pumping or leakage at an actual well. Suppose mass is injected at specified location $\boldsymbol{\bar{x}}$ such that 
\begin{equation}
    g^{\prime}(\boldsymbol{x}, t) = 35 \sin{(2 \pi t)} \exp{\left( -\frac{1}{20}(\boldsymbol{x}-\boldsymbol{\bar{x}})^T (\boldsymbol{x}-\boldsymbol{\bar{x}}) \right)}.
    \label{eq:advection-diffusion-forcing}
\end{equation}
When $\sin{(2 \pi t)} > 0$, tracer is injected near the point $\boldsymbol{\bar{x}}$, and when $\sin{(2 \pi t)} < 0$ tracer is removed. We estimate $\bar{g}(t) \approx n^{-1}   \sum_{i=1}^{n} g^{\prime}(\boldsymbol{x}^{(i)}, t)$ to evolve the normalizing constant $M(t)$ and prescribe $g = g^{\prime} - \bar{g}$.
Figure~\ref{fig:density-evolution-with-source} shows the evolution of samples that are initially distributed according to a standard Gaussian distribution and evolve according to~\eqref{eq:discrete-sample-evolution} with $\boldsymbol{u} = \boldsymbol{0}$, $\boldsymbol{\sigma} = \boldsymbol{I}$, and $\boldsymbol{\bar{x}} = (1, \, 0)$. 

The scatterplot in Figure~\ref{fig:density-evolution-with-source} shows the solution $f$ to the Kolmogorov problem at each timestep. However, the gradient of the solution $\nabla f$ is the important result since it determines the effective velocity $\bm{u^{\prime}} = -\nabla f$. In Figs.~\ref{fig:density-evolution-with-source-000} and~\ref{fig:density-evolution-with-source-1000} the solution is zero since there is no injection or removal of mass from the source term ($\sin{(2 \pi t)} = 0$ at these times). The effective velocity is, therefore, also zero. However, we see in Figure~\ref{fig:density-evolution-with-source-333}, when $\sin{(2 \pi t)} > 0$, that the effective velocity moves mass toward the source location $\boldsymbol{\bar{x}}$ and in Figure~\ref{fig:density-evolution-with-source-666}, when $\sin{(2 \pi t)} < 0$ mass moves away from the source location. Recall, the mass is distributed according to $\psi$, which always integrates to 1. The effective velocity moves mass toward a source when $g>0$ since the mass influx results in a larger fraction of the mass being located in the injection region. The opposite is true when the source $g<0$. Thus, the qualitative behavior of the particle-based approximation is consistent with our physical intuition for the behavior of the density function.

\section{Conclusion} \label{sec:conclusions}

This work used the spectral decomposition of a discrete Kolmogorov operator approximating a gradient flow on a manifold to solve the associated Poisson-type problem (Equation~\ref{eq:Kolmogorov-problem}), and represent the gradient vector field of the solution. Our kernel-based approach performs these computations given samples from the equilibrium distribution of the gradient flow, without having to resort to explicit discretizations of the state space (e.g., as in finite-difference schemes). Therefore, our framework is well-suited for function and gradient-field approximation from unstructured data sampled on nonlinear  manifolds, embedded in a potentially high-dimensional ambient space. The method presented here extends previous work \cite{berryharlim2016}, which is primarily concerned with density estimation or discretizing the Kolmogorov operator, by using ideas from the exterior calculus on manifolds \cite{BerryGiannakis20} to represent gradient vector fields.

We also develop computationally efficient software that, given samples, (i) uses kernel density estimation to estimate the probability density function, (ii) computes the discrete Kolmogorov operator, and (iii) provides tools to represent functions and their gradients using the eigenbasis of the Kolmogorov operator. The two computational bottlenecks are (i) computing the eigendecomposition of the discrete Kolmogorov operator and (ii) assembling large kernel matrices. For the former, we use existing software: \texttt{Spectra} \cite{spectra}. To efficiently compute the kernel matrices, we employ an algorithm with $\mathcal{O}(n\log{(n)})$ complexity. While there is extensive literature and software addressing specific aspects of these problems, we believe that efficient implementations targeting the full algorithmic pipeline presented in this paper are not readily available in the public domain. Our implementation is available as part of the MIT Uncertainty Quantification software package (\texttt{muq.mit.edu}). Thus, we hope that the software and computational guidance presented here will be useful to the broad range of disciplines where data-driven approximation of differential operators plays a role. 


\begin{thebibliography}{10}

\bibitem{aryaetal1996}
Sunil Arya, David~M Mount, and Onuttom Narayan.
\newblock Accounting for boundary effects in nearest-neighbor searching.
\newblock {\em Discrete \& Computational Geometry}, 16(2):155--176, 1996.

\bibitem{aryaetal1998}
Sunil Arya, David~M Mount, Nathan~S Netanyahu, Ruth Silverman, and Angela~Y Wu.
\newblock An optimal algorithm for approximate nearest neighbor searching fixed
  dimensions.
\newblock {\em Journal of the ACM (JACM)}, 45(6):891--923, 1998.

\bibitem{bakryetal2013}
Dominique Bakry, Ivan Gentil, and Michel Ledoux.
\newblock {\em Analysis and geometry of Markov diffusion operators}, volume
  348.
\newblock Springer Science \& Business Media, 2013.

\bibitem{belkinniyogi2003}
Mikhail Belkin and Partha Niyogi.
\newblock Laplacian eigenmaps for dimensionality reduction and data
  representation.
\newblock {\em Neural computation}, 15(6):1373--1396, 2003.

\bibitem{bentley1975}
Jon~Louis Bentley.
\newblock Multidimensional binary search trees used for associative searching.
\newblock {\em Communications of the ACM}, 18(9):509--517, 1975.

\bibitem{BerryGiannakis20}
Tyrus Berry and Dimitrios Giannakis.
\newblock Spectral exterior calculus.
\newblock {\em Communications on Pure and Applied Mathematics}, 73(4):689--770,
  2019.

\bibitem{berryetal2015}
Tyrus Berry, Dimitrios Giannakis, and John Harlim.
\newblock Nonparametric forecasting of low-dimensional dynamical systems.
\newblock {\em Physical Review E}, 91(3):032915, 2015.

\bibitem{berryharlim2016}
Tyrus Berry and John Harlim.
\newblock Variable bandwidth diffusion kernels.
\newblock {\em Applied and Computational Harmonic Analysis}, 40(1):68--96,
  2016.

\bibitem{bertozzietal2018}
Andrea~L Bertozzi, Xiyang Luo, Andrew~M Stuart, and Konstantinos~C Zygalakis.
\newblock Uncertainty quantification in graph-based classification of high
  dimensional data.
\newblock {\em SIAM/ASA Journal on Uncertainty Quantification}, 6(2):568--595,
  2018.

\bibitem{nanoflann}
Jose~Luis Blanco and Pranjal~Kumar Rai.
\newblock nanoflann: a {C}++ header-only fork of {FLANN}, a library for nearest
  neighbor ({NN}) with {KD}-trees.
\newblock https://github.com/jlblancoc/nanoflann, 2014.

\bibitem{coifmanlafon2006}
Ronald~R Coifman and St{\'e}phane Lafon.
\newblock Diffusion maps.
\newblock {\em Applied and computational harmonic analysis}, 21(1):5--30, 2006.

\bibitem{CoifmanEtAl08}
Ronald~R Coifman, Yoel Shkolnisky, Fred~J. Sigworth, and Amit Singer.
\newblock Graph {L}aplacian tomography from unknown random projections.
\newblock {\em IEEE Trans. Image Process.}, 17(10):1891--1899, 2008.

\bibitem{cormenetal2009}
Thomas~H Cormen, Charles~E Leiserson, Ronald~L Rivest, and Clifford Stein.
\newblock {\em Introduction to algorithms}.
\newblock MIT press, 2009.

\bibitem{couranthilbert2008}
Richard Courant and David Hilbert.
\newblock {\em Methods of mathematical physics: partial differential
  equations}.
\newblock John Wiley \& Sons, 2008.

\bibitem{rosenblatt1956}
Richard~A Davis, Keh-Shin Lii, and Dimitris~N Politis.
\newblock Remarks on some nonparametric estimates of a density function.
\newblock In {\em Selected Works of Murray Rosenblatt}, pages 95--100.
  Springer, 2011.

\bibitem{friedmanetal1977}
Jerome~H Friedman, Jon~Louis Bentley, and Raphael~Ari Finkel.
\newblock An algorithm for finding best matches in logarithmic expected time.
\newblock {\em ACM Transactions on Mathematical Software (TOMS)},
  3(3):209--226, 1977.

\bibitem{giannakis2019}
Dimitrios Giannakis.
\newblock Data-driven spectral decomposition and forecasting of ergodic
  dynamical systems.
\newblock {\em Applied and Computational Harmonic Analysis}, 47(2):338--396,
  2019.

\bibitem{HeinEtAl07}
Matthias Hein, Jean-Yves Audibert, and Ulrike von Luxburg.
\newblock Graph {L}apliacians and their convergence on random neighborhood
  graphs.
\newblock {\em Journal of Machine Learning Research}, 8:1325--1368, 2007.

\bibitem{JiangHarlim20}
Shixiao~W Jiang and John Harlim.
\newblock Ghost point diffusion maps for solving elliptic {PDE}'s on manifolds
  with classical boundary conditions.
\newblock {\em arXiv preprint arXiv:2006.04002}, 2020.

\bibitem{nlopt}
Steven~G Johnson.
\newblock The {NL}opt nonlinear-optimization package.
\newblock http://github.com/stevengj/nlopt, 2014.

\bibitem{jonesetal2011}
Peter~Wilcox Jones, Andrei Osipov, and Vladimir Rokhlin.
\newblock Randomized approximate nearest neighbors algorithm.
\newblock {\em Proceedings of the National Academy of Sciences},
  108(38):15679--15686, 2011.

\bibitem{kimpark2013}
Yoon~Tae Kim and Hyun~Suk Park.
\newblock Geometric structures arising from kernel density estimation on
  riemannian manifolds.
\newblock {\em Journal of Multivariate Analysis}, 114:112--126, 2013.

\bibitem{ngetal2002}
Andrew~Y Ng, Michael~I Jordan, and Yair Weiss.
\newblock On spectral clustering: Analysis and an algorithm.
\newblock 2:849--856, 2002.

\bibitem{ozakingray2009}
Arkadas Ozakin and Alexander Gray.
\newblock Submanifold density estimation.
\newblock {\em Advances in Neural Information Processing Systems},
  22:1375--1382, 2009.

\bibitem{parzen1962}
Emanuel Parzen.
\newblock On estimation of a probability density function and mode.
\newblock {\em The annals of mathematical statistics}, 33(3):1065--1076, 1962.

\bibitem{pelletier2005}
Bruno Pelletier.
\newblock Kernel density estimation on riemannian manifolds.
\newblock {\em Statistics \& probability letters}, 73(3):297--304, 2005.

\bibitem{spectra}
Yixuan Qiu.
\newblock Spectra: {S}parse {E}igenvalue {C}omputation {T}oolkit as a
  {R}edesigned {ARPACK}.
\newblock https://spectralib.org/, 2017.

\bibitem{sainscott1996}
Stephan~R Sain and David~W Scott.
\newblock On locally adaptive density estimation.
\newblock {\em Journal of the American Statistical Association},
  91(436):1525--1534, 1996.

\bibitem{shimalik2000}
Jianbo Shi and Jitendra Malik.
\newblock Normalized cuts and image segmentation.
\newblock {\em IEEE Transactions on pattern analysis and machine intelligence},
  22(8):888--905, 2000.

\bibitem{singer2006}
Amit Singer.
\newblock From graph to manifold {L}aplacian: The convergence rate.
\newblock {\em Applied and Computational Harmonic Analysis}, 21(1):128--134,
  2006.

\bibitem{sproull1991}
Robert~F Sproull.
\newblock Refinements to nearest-neighbor searching in k-dimensional trees.
\newblock {\em Algorithmica}, 6(1):579--589, 1991.

\bibitem{terrellscott1992}
George~R Terrell and David~W Scott.
\newblock Variable kernel density estimation.
\newblock {\em The Annals of Statistics}, pages 1236--1265, 1992.

\bibitem{tothetal2017}
Csaba~D Toth, Joseph O'Rourke, and Jacob~E Goodman.
\newblock {\em Handbook of discrete and computational geometry}.
\newblock CRC press, 2017.

\bibitem{TrillosEtAl20}
Nicol\'{a}s~G Trillos, Moritz Gerlach, Matthias Hein, and Dejan Slep{\v{c}}ev.
\newblock Error estimates for spectral convergence of the graph {L}aplacian on
  random geometric graphs towards the {L}aplace--{B}eltrami operator.
\newblock {\em Foundations of Computational Mathematics}, 20:827--887, 2020.

\bibitem{vaughnetal2019}
Ryan Vaughn, Tyrus Berry, and Harbir Antil.
\newblock Diffusion maps for embedded manifolds with boundary with applications
  to {pde}s.
\newblock {\em arXiv preprint arXiv:1912.01391}, 2019.

\end{thebibliography}

\end{document}